\documentclass[11pt,a4paper]{amsart}

\usepackage{amsmath}
\usepackage{amsthm}
\usepackage{amssymb}
\usepackage{dsfont}
\usepackage{mathtools}
\usepackage{tikz-cd}
\usepackage{tikz}
\usepackage{enumerate}
\usepackage{pstricks}
\usepackage{import}
\usepackage{subcaption}

\usepackage{kpfonts}

\usepackage{tabu}
\usepackage{fancyhdr}
\usepackage{lastpage}
\usepackage{geometry}
\usepackage{xcolor}
\usepackage{hyperref}
\usepackage[disable]{todonotes}

\newcommand{\RR}{\mathbb{R}}

\newcommand{\ZZ}{\mathbb{Z}}
\newcommand{\NN}{\mathbb{N}}
\newcommand{\HH}{\mathbb{H}}

\renewcommand{\SS}{\mathbb{S}}
\newcommand{\II}{\mathds{1}}

\renewcommand{\epsilon}{\varepsilon}
\renewcommand{\emptyset}{\varnothing}

\newcommand{\st}{\, \colon \,}

\newcommand{\Homeo}{\operatorname{Homeo}}
\newcommand{\Isom}{\operatorname{Isom}}

\newcommand{\Aut}{\operatorname{Aut}}
\newcommand{\id}{\operatorname{id}}

\newcommand{\PSL}{\operatorname{PSL}}

\newcommand{\SO}{\operatorname{SO}}
\newcommand{\interior}{\operatorname{int}}
\renewcommand{\Im}{\operatorname{Im}}
\newcommand{\diam}{\operatorname{diam}}
\newcommand{\ol}[1]{\overline{#1}}

\newcommand{\supp}{\operatorname{supp}}

\newcommand{\Sub}{\operatorname{Sub}}

\newcommand{\abs}[1]{\left | #1 \right |}
\newcommand{\norm}[1]{ \left \| #1 \right \|}
\newcommand{\tr}{\operatorname{tr}}

\newcommand{\Subd}{\Sub_{\mathrm{d}}}
\newcommand{\Subdtf}{\Sub_{\mathrm{dtf}}}
\newcommand{\IRS}{\operatorname{IRS}}
\newcommand{\Prob}{\operatorname{Prob}}

\newcommand{\vol}{\operatorname{vol}}

\newcommand{\Sym}{\operatorname{Sym}}

\newcommand{\InjRad}{\operatorname{InjRad}}

\renewcommand{\phi}{\varphi}
\newcommand{\MCG}{\operatorname{MCG}}

\newcommand{\augTeich}{\widehat{\mc{T}}}
\newcommand{\augModuli}{\widehat{\mc{M}}}
\newcommand{\Teich}{\mc{T}}
\newcommand{\Moduli}{\mc{M}}
\newcommand{\IRSModuli}{\overline{\Moduli}^\text{\tiny{IRS}}}

\newcommand{\Rep}{\mc{R}}
\newcommand{\stab}{\operatorname{stab}}

\newcommand{\Deck}{\operatorname{Deck}}
\newcommand{\res}{\operatorname{res}}
\newcommand{\acts}{\curvearrowright}
\newcommand{\Cay}{\operatorname{Cay}}
\newcommand{\Lattices}{\mathcal{L}}
\newcommand{\conv}{\operatorname{conv}}

\newcommand{\Inn}{\operatorname{Inn}}
\newcommand{\Out}{\operatorname{Out}}
\newcommand{\ax}{\operatorname{ax}}

\newcommand{\PMCG}{\operatorname{PMCG}}

\renewcommand{\hat}[1]{\widehat{#1}}
\renewcommand{\tilde}[1]{\widetilde{#1}}

\newcommand{\im}{\operatorname{im}}
\renewcommand{\abs}[1]{\left| #1 \right|}
\renewcommand{\sinh}{\operatorname{sinh}}

\newcommand{\arcosh}{\operatorname{arcosh}}

\newif\ifshowutils
\showutilsfalse

\newcommand{\p}[1]{\left( #1 \right)}
\newcommand{\mc}[1]{\mathcal{#1}}
\newcommand{\mb}[1]{\mathbf{#1}}
\newcommand{\mf}[1]{\mathfrak{#1}}

\theoremstyle{plain}

\newtheorem*{thm*}{Theorem}
\newtheorem{thm}{Theorem}[subsection]
\newtheorem{prop}[thm]{Proposition}
\newtheorem{cor}[thm]{Corollary}
\newtheorem{lemma}[thm]{Lemma}

\theoremstyle{definition}
\newtheorem{defn}[thm]{Definition}

\newtheorem{propdefn}[thm]{Proposition and Definition}

\theoremstyle{remark}
\newtheorem{remark}[thm]{Remark}
\newtheorem{example}[thm]{Example}

\definecolor{inkscapepurple}{rgb}{0.50196078,0,0.50196078}
\definecolor{inkscapeblue}{rgb}{0,0,1}

\title{On the IRS compactification of moduli space}

\author{Yannick Krifka}
\thanks{Parts of this material are based upon work supported by the National Science Foundation under Grant No.\ 1440140, while the author was in residence at the Mathematical Sciences Research Institute in Berkeley, California, during the fall semester 2019.}

\date{\today}

\begin{document}

\begin{abstract}
	In \cite{gelanderIRS15} Gelander described a new compactification of the moduli space of finite area hyperbolic surfaces using invariant random subgroups. The goal of this paper is to relate this compactification to the classical augmented moduli space, also known as the Deligne--Mumford compactification. We define a continuous finite-to-one surjection from the augmented moduli space to the IRS compactification. The cardinalities of this map's fibers admit a uniform upper bound that depends only on the topology of the underlying surface. 
\end{abstract}
	
	\maketitle

	\tableofcontents


\section{Introduction}

In recent years invariant random subgroups became a popular subject to study as they generalize lattices in locally compact groups; see \cite{gelanderIRS15, gelanderICM} and the references therein. An \emph{invariant random subgroup}, or \emph{IRS} for short, of a locally compact group $G$ is, by definition, a Borel probability measure $\mu$ on the space of closed subgroups $\Sub(G)$ of $G$ that is invariant under the conjugation action of $G$ on $\Sub(G)$. Here the space of closed subgroups is equipped with the Chabauty topology which turns it into a compact Hausdorff space. Many results about lattices find natural generalizations to results about invariant random subgroups. For example, there is a version of the Borel density theorem, the Kazhdan--Margulis theorem and the Stuck--Zimmer rigidity theorem for invariant random subgroups; see \cite{gelanderIRS15, gelanderICM}.

It turns out that invariant random subgroups arise naturally in the study of probability measure preserving actions of $G$. In fact, if $G$  acts via measure preserving transformations on a countably separated probability space $(X,\nu)$, pushing forward the measure $\nu$ via the stabilizer map
\begin{align*}
\stab \colon X &\longrightarrow \Sub(G), \\
x &\longmapsto \stab_G(x)
\end{align*}
induces an IRS $\mu \coloneqq \stab_*(\nu) \in \IRS(G)$. Moreover, it can be shown that every IRS arises in this fashion; see \cite{7samurai, AGV14gel,  gelanderIRS15, gelanderICM}.

Now, it is easy to see that every lattice $\Gamma \leq G$ amounts to an invariant random subgroup. Indeed, there is a unique invariant Borel probability measure $\nu_\Gamma$ on the quotient $\Gamma \backslash G$. Hence, we have a probability measure preserving action of $G$ on $(\Gamma \backslash G, \nu_\Gamma)$ given by 
\[ g * \Gamma h = \Gamma (h g^{-1}) \qquad \forall g \in G \quad \forall \Gamma h \in \Gamma \backslash G.\]
We obtain by the above construction an IRS $\mu_\Gamma \coloneqq \stab_*(\nu_\Gamma)$.

In this case the stabilizer of a coset $\Gamma g \in \Gamma \backslash G$ is $\stab_G(\Gamma g) = g^{-1} \Gamma g$. Therefore, we may think of the IRS $\mu_\Gamma$ as the push-forward of the measure $\nu_\Gamma$ via the map
\begin{align*}
\phi_\Gamma \colon \Gamma \backslash G & \longrightarrow \Sub(G),\\
\Gamma g &\longmapsto g^{-1} \Gamma g.
\end{align*}
Using the invariance of $\nu_\Gamma$ one verifies that $\mu_\Gamma$ only depends on the conjugacy class of $\Gamma$ in $G$; see Lemma \ref{lem:IRSconjuginv}.

This quite general construction has an interesting application in the case of torsion-free lattices $\Gamma$ in $G \coloneqq \PSL(2,\RR) \cong \Isom^+(\HH^2)$. Given such a lattice $\Gamma$ the quotient $X = \Gamma \backslash \HH^2$ is a hyperbolic surface of finite topological type. 
Vice versa, every hyperbolic surface $X$ with finite area is a quotient $X=\Gamma \backslash \HH^2$ where $\Gamma \leq G$ is a torsion-free lattice. Two such surfaces $\Gamma \backslash \HH^2$ and $\Gamma' \backslash \HH^2$ are isometric if and only if $\Gamma$ and $\Gamma'$ are conjugate. Therefore, isometry classes of finite area hyperbolic surfaces are in one-to-one correspondence with conjugacy classes of torsion-free lattices in $G$.

We fix an orientable topological surface $\Sigma$ with negative Euler characteristic $\chi(\Sigma) < 0$ and no boundary.
Now, we can identify the quotient space $G \backslash \Lattices(\Sigma)$ of all $G$-conjugacy classes of lattices in $\Lattices(\Sigma)$ with the moduli space $\Moduli(\Sigma)$ of all finite area hyperbolic metrics on $\Sigma$; see Proposition \ref{prop:ModuliAsLattices}.
Here $\Lattices(\Sigma) \subset \Sub(G)$ is the subspace of all torsion-free lattices $\Gamma \leq G$ such that $\Gamma \backslash \HH^2$ is homeomorphic to  $\Sigma$.

Because the invariant random subgroup $\mu_\Gamma$ depends only on the conjugacy class $[ \Gamma ] \in G \backslash \Lattices(\Sigma)$ we obtain a well-defined map
\begin{align*}
\iota \colon \Moduli(\Sigma) &\longrightarrow \IRS(G),\\
[\Gamma] &\longmapsto \mu_\Gamma
\end{align*}
via the identification $\Moduli(\Sigma) \cong G \backslash \Lattices(\Sigma)$. We prove in Proposition \ref{prop:topembedding} that $\iota$ is a topological embedding.

Note that the space of invariant random subgroup $\IRS(G) \subseteq \Prob(\Sub(G))$ is a compact Hausdorff space. Therefore, one may take the closure $\ol{\iota(\Moduli(\Sigma))}$ to obtain a compactification, the so called \emph{IRS compactification $\IRSModuli(\Sigma)$ of the moduli space $\Moduli(\Sigma)$}. This term was coined by Gelander in \cite{gelanderIRS15} where the above construction was described for the first time.

It is natural to ask what can be said about this compactification; see Problem 3.2 in \cite{gelanderIRS15}. The objective of this paper is to answer this question by relating the IRS compactification $\IRSModuli(\Sigma)$ to the classical Deligne--Mumford compactification $\augModuli(\Sigma)$. In this paper we regard the Deligne--Mumford compactification as the \emph{augmented moduli space $\augModuli(\Sigma)$} that is the quotient of augmented Teichm\"uller space $\augTeich(\Sigma)$ by its mapping class group action; see subsections \ref{subsect:AugTeich} and \ref{subsect:aug_moduli} for details.

Intuitively, a point $\mb{x} \in \augModuli(\Sigma)$ in the augmented moduli space comprises the data of a nodal surface $\mb{X} \coloneqq \Sigma / \sigma$, where a family of disjoint simple closed curves $\sigma$ in $\Sigma$ is collapsed to nodes, and a hyperbolic metric of finite area on every complementary component $\Sigma' \in c(\sigma) \coloneqq \pi_0(\Sigma \setminus \sigma)$. Hence, there is a torsion-free lattice $\Gamma(\Sigma') \leq G$ for every component $\Sigma' \in c(\sigma)$ such that $\Gamma(\Sigma') \backslash \HH^2 \cong \Sigma'$. 

Keeping this description in mind there is a rather natural map 
\[\Phi \colon \augModuli(\Sigma) \longrightarrow \IRSModuli(\Sigma).\] 
Indeed, let $\mb{x} \in \augModuli(\Sigma)$ be a point in the augmented moduli space as in the previous paragraph. For every component $\Sigma' \in c(\sigma)$ there is a torsion-free lattice $\Gamma(\Sigma')$ describing the hyperbolic metric on $\Sigma' \cong \Gamma(\Sigma') \backslash \HH^2$ and we obtain a corresponding invariant random subgroup $\mu_{\Gamma(\Sigma')}$ as before. It is well known that the area of a hyperbolic surface is a topological invariant $\vol(\Gamma(\Sigma') \backslash \HH^2) = - 2\pi \chi(\Sigma')\eqqcolon \vol(\Sigma')$. Moreover, $\chi(\Sigma) = \sum_{\Sigma' \in c(\sigma)} \chi(\Sigma')$ and we define
\[ \Phi(\mb{x}) \coloneqq \sum_{\Sigma' \in c(\sigma)} \frac{\vol(\Sigma')}{\vol(\Sigma)} \cdot \mu_{\Gamma(\Sigma')} = \sum_{\Sigma' \in c(\sigma)} \frac{\chi(\Sigma')}{\chi(\Sigma)} \cdot \mu_{\Gamma(\Sigma')} \in \IRS(G),\]
which is a convex combination of invariant random subgroups.

The main result of this paper is Theorem \ref{thm:main}:
\begin{thm*}[\ref{thm:main}]
	The map 
	\[ \Phi \colon \augModuli(\Sigma) \longrightarrow \IRSModuli(\Sigma)\]
	is a continuous finite-to-one surjection extending the topological embedding $\iota \colon \Moduli(\Sigma) \hookrightarrow \IRSModuli(\Sigma)$. There is a uniform upper bound $B(\Sigma)>0$, which depends only on the topology of $\Sigma$, such that $\# \Phi^{-1}(\mu) \leq B(\Sigma)$ for all $\mu \in \IRSModuli(\Sigma)$.
\end{thm*}

Let us explain informally the upper bound on the cardinality of the preimage.
Given a point $\mb{x} \in \augModuli(\Sigma)$ in the preimage of $\mu  = \sum_{i=1}^m \lambda_i \cdot \mu_{\Gamma_i} \in \IRSModuli(\Sigma)$ we know the hyperbolic surfaces $\Gamma_1 \backslash \HH^2, \ldots, \Gamma_m \backslash \HH^2$ that arise as the components of its nodal surface $\mb{X}$. However, one does not have any information on how these fit together to form the nodal surface $\mb{X}$. One can show that the number of different ways these components may fit together depends only on the topology of the underlying surface $\Sigma$.

\subsection*{Outline of the paper}

In section \ref{sect:aug_mod_space} we review basic notions of Teichm\"uller theory. Moreover, we give an interpretation of the augmented Teichm\"uller space $\augTeich(\Sigma)$ by means of (admissible) representations and restriction maps. The augmented moduli space $\augModuli(\Sigma)$ is then introduced as the quotient of augmented Teich\"uller space $\augTeich(\Sigma)$ by the mapping class group action. 

The space of closed subgroups $\Sub(G)$ together with its Chabauty topology is introduced in section \ref{sect:space_of_closed_subgroups}. We define the geometric topology on the space of admissible representations and prove in Proposition \ref{prop:geometric_equals_algebraic} that it coincides with the algebraic topology. This is later used to prove that the moduli space $\Moduli(\Sigma)$ is in fact homeomorphic to $G \backslash \Lattices(\Sigma)$.

In section \ref{sect:irs} we introduce invariant random subgroups. First, we prove that the map $\iota \colon \Moduli(\Sigma) \longrightarrow \IRS(G)$ is a topological embedding. We then define $\Phi \colon \augModuli(\Sigma) \longrightarrow \IRSModuli(\Sigma)$ and prove Theorem \ref{thm:main} in subsection \ref{subsect:augmodandirs}.

The main tool in the proofs of continuity of both $\iota \colon \Moduli(\Sigma) \hookrightarrow \IRS(G)$  and $\Phi \colon \augModuli(\Sigma) \longrightarrow \IRSModuli(\Sigma)$ is Lemma \ref{lem:L1ConvTruncDom}. Although, it seems to be quite classical we could not find a reference in the literature. Therefore, we decided to include a complete proof using elementary hyperbolic geometry in section \ref{section:ProofOfLemma}.

\subsection*{Acknowledgements}

The author would like to thank his advisor Alessandra Iozzi for her guidance and many helpful comments. Also, he would like to thank Marc Burger for many inspiring conversations, and the Mathematical Sciences Research Institute in Berkeley for its hospitality during the research program ``Holomorphic Differentials in Mathematics and Physics'' during the fall of 2019.


\section{Augmented Moduli Space} \label{sect:aug_mod_space}

\subsection{Teichm\"uller Space} \label{subsect:teich_space}

Throughout this paper let $\Sigma$ be an oriented surface with negative Euler characteristic $\chi(\Sigma)<0$ of genus $g$ with $p$ punctures and no boundary. Further, we shall denote by $G := \PSL(2,\RR) \cong \Isom^+(\HH^2)$ the group of orientation preserving isometries of the hyperbolic plane $\HH^2$.

\begin{defn}
    A discrete and faithful representation
    \[ \rho \colon \pi_1(\Sigma) \longrightarrow G \]
    is called \emph{admissible} if it is the holonomy representation of an orientation preserving homeomorphism $f \colon \Sigma \longrightarrow X$ where $X$ is a hyperbolic surface. The set of all such representations is denoted by $\Rep^*(\Sigma)$. 
    
    If we additionally require that the hyperbolic surface $X$ above has finite area, we denote the resulting subset by $\Rep(\Sigma)$.
\end{defn}

\begin{defn}
    An element $\alpha \in \pi_1(\Sigma)$ is called \emph{peripheral} if it represents a curve homotopic to a puncture.
    
    An element $\alpha \in \pi_1(\Sigma)$ is called \emph{essential} if it is not peripheral.
\end{defn}

\begin{remark}\label{rem:PeripheralsParabolic}
    An admissible representation $\rho \in \Rep^*(\Sigma)$ is in $\Rep(\Sigma)$
    if and only if $\rho(\alpha) \in G$ is parabolic for every peripheral element $\alpha \in \pi_1(\Sigma)$.
\end{remark}

\begin{propdefn}
    The group $\pi_1(\Sigma)$ admits a finite generating set $S$, and the map
    \begin{align*}
        i \colon \Rep^*(\Sigma) &\hookrightarrow G^S,\\
        \rho &\mapsto (\rho(s))_{s \in S}
    \end{align*}
    is injective. We equip $G^S$ with the product topology and $\Rep^*(\Sigma)$ with the initial topology with respect to the injection $i$. This topology does not depend on the choice of generating set and is called the \emph{algebraic topology}.
    
    A sequence of representations $(\rho_n)_{n \in \NN} \subset \Rep^*(\Sigma)$ converges to $\rho \in \Rep^*(\Sigma)$ as $n \to \infty $ with respect to this topology if and only if $\rho_n(\gamma) \to \rho(\gamma)$ as $n \to \infty$ for every $\gamma \in \pi_1(\Sigma)$.
\end{propdefn}

\begin{propdefn}
    The group $G$ acts on $\Rep(\Sigma)$ resp.\ $\Rep^{*}(\Sigma)$ continuously from the left with closed orbits, and we denote the quotients by
    \[ \Teich(\Sigma) := G \backslash \Rep(\Sigma) \quad \text{resp.} \quad \Teich^*(\Sigma) := G \backslash \Rep^{*}(\Sigma)\]
    The quotient spaces are Hausdorff topological spaces, and $\Teich(\Sigma)$
    is called the \emph{Teichm\"uller space of $\Sigma$}.
\end{propdefn}

\begin{defn}
    The \emph{translation length} of an element $g \in G$ is defined as
    \[ \ell(g) := \inf_{x \in \HH^2} d(g x,x). \]
\end{defn}

\begin{lemma}
    Let $g \in G$. Then
    \[ \ell(g) = 2 \arcosh\p{ \tfrac{1}{2} \cdot \max\p{2, \abs{\tr(g)}} }.\]
    
    In particular, $\Rep^*(\Sigma) \longrightarrow \RR, \rho \mapsto \ell(\rho(\gamma))$ is continuous for every $\gamma \in \pi_1(\Sigma)$.
\end{lemma}

\begin{defn}
    Let $\alpha, \beta \in \pi_1(\Sigma)$. The \emph{geometric intersection number $i(\alpha, \beta)$} is defined as
    \[ i(\alpha, \beta) := \min_{c_1, c_2} \# (c_1 \cap c_2) \]
    where the minimum is taken over all loops $c_1, c_2$ in the free homotopy classes of $\alpha, \beta$ respectively.
\end{defn}

\begin{lemma}[Collar Lemma; see {\cite[Corollary 4.1.2]{buser}}] \label{lem:CollarLemma}
    Let $\alpha, \beta \in \pi_1(\Sigma)$ and suppose that $\alpha$ is primitive and $i(\alpha, \beta) > 0$. If $\rho \in \Rep^*(\Sigma)$, then
    \[ \sinh\p{ \tfrac{\ell(\rho(\alpha))}{2}} \cdot \sinh \p{ \tfrac{\ell(\rho(\beta))}{2}} \geq 1. \]
    
    In particular, if $(\rho_n)_{n \in \NN} \subset \Rep^*(\Sigma)$ is a sequence such that $\ell(\rho_n(\alpha)) \to 0$ as $n \to \infty$, then $\ell(\rho_n(\beta)) \to \infty$ as $n \to \infty$.
\end{lemma}

\subsection{Moduli Space} \label{subsect:moduli_space}

\begin{defn}
    The group
    \[ \MCG(\Sigma) := \Homeo^+(\Sigma) / \Homeo_\circ(\Sigma) \]
    is called the \emph{mapping class group of $\Sigma$}. Here $\Homeo_\circ(\Sigma)$ denotes the identity component in $\Homeo(\Sigma)$, i.e.\ all homeomorphisms isotopic to the identity.
    
    For an orientation preserving homeomorphism $f \colon \Sigma \longrightarrow \Sigma$ we denote its mapping class by $[f]$.
\end{defn}

Let $\phi = [f] \in \MCG(\Sigma)$, let $p \in \Sigma$ and let $\beta \colon [0,1] \longrightarrow \Sigma$ be a path from $\beta(0) = f(p)$ to $\beta(1) = p$. Identifying $\pi_1(\Sigma) \cong \pi_1(\Sigma,p)$ we obtain an automorphism at the level of fundamental groups $f_* \colon \pi_1(\Sigma) \longrightarrow \pi_1(\Sigma)$ given by
\[ f_*([c]) \coloneqq [ \beta \cdot (f \circ c) \cdot \beta^{-1}] \]
for every homotopy class $[c] \in \pi_1(\Sigma,p)$ of a closed loop $c$
at $p$. 

This construction depends on the choice of representative $f \in \phi$ and the choice of path $\beta$. However, we have the following Proposition.

\begin{propdefn}
    The map
    \begin{align*}
    \MCG(\Sigma) &\longrightarrow \Out(\pi_1(\Sigma)) = \Aut(\pi_1(\Sigma)) / \Inn(\pi_1(\Sigma)),\\ 
    \phi=[f] &\longmapsto \phi_* \coloneqq [f_*]
    \end{align*}
    is a well-defined injective homomorphism.
    
    We denote its image by $\Out^*(\pi_1(\Sigma)) \leq \Out(\pi_1(\Sigma))$ and its preimage under the quotient map $\Aut(\pi_(\Sigma)) \longrightarrow \Out(\pi_1(\Sigma))$ by $\Aut^*(\pi_1(\Sigma)) \leq \Aut(\pi_1(\Sigma))$. These (outer) automorphisms are called \emph{geometric} or \emph{admissible}.
\end{propdefn}

\begin{propdefn}\label{propdefn:ModuliSpace}
    The group $\Aut^*(\pi_1(\Sigma))$ acts on $\Rep^{(*)}(\Sigma)$ from the right via precomposition, and induces a right-action of $\Out^*(\pi_1(\Sigma)) \cong \MCG(\Sigma)$ on $\Teich(\Sigma)$. 
    
    The quotient space
    \[ \Moduli(\Sigma) \coloneqq \Teich(\Sigma)/\MCG(\Sigma) \]
    is called the \emph{Moduli space of $\Sigma$}. We will denote the $\MCG(\Sigma)$-equivalence class of $[\rho] \in \Teich(\Sigma)$ by $[[\rho]] \in \Moduli(\Sigma)$.
    
    Moreover, we have the following commutative diagram
    \begin{center}
        \begin{tikzcd}[column sep = huge, row sep = large]
            \Rep(\Sigma) \arrow[r,"\Aut^*(\pi_1(\Sigma))"] \arrow[d,"G"]
            & \Rep(\Sigma)/{\Aut^*(\pi_1(\Sigma))} \arrow[d,"G"]\\
            \Teich(\Sigma) \arrow[r,"\MCG(\Sigma)\cong \Out^*(\pi_1(\Sigma))"] & \Moduli(\Sigma)
        \end{tikzcd}
    \end{center}
    where every map is the quotient map with respect to the action of the annotated group.
\end{propdefn}
\begin{proof}
    Let $\rho \in \Rep^{(*)}(\Sigma)$ and $f \colon \Sigma \longrightarrow X$ an orientation preserving homeomorphism onto a hyperbolic surface $X$ such that $f_* = \rho$. Let $\alpha \in \Aut^*(\pi_1(\Sigma))$ such that $[\alpha] = [g_*] \in \Out^*(\pi_1(\Sigma))$  where $g \colon \Sigma \longrightarrow \Sigma$ is  an orientation preserving homeomorphism. Then $\rho \circ \alpha$ is the holonomy representation of $f \circ g \colon \Sigma \longrightarrow X$ which is an orientation preserving homeomorphism and $\rho \circ \alpha \in \Rep^*(\Sigma)$.
    
    The other assertions follow easily.
\end{proof}

\subsection{Augmented Teichm\"uller Space}\label{subsect:AugTeich}

Following Harvey \cite{harveyshort, harvey77} and Abikoff \cite{abikoff} we will now introduce augmented Teichm\"uller space $\augTeich(\Sigma)$; a bordification of Teichm\"uller space $\Teich(\Sigma)$. 
The idea is to allow the lengths of (homotopically) disjoint simple closed curves to go to zero as one moves to infinity in $\Teich(\Sigma)$. This will be accounted for by attaching the Teichm\"uller spaces of the subsurfaces in the complement of the pinched curves. 
Thus augmented Teichm\"uller space will admit a natural stratification in terms of the curve complex $\mc{C}(\Sigma)$.

Recall that the curve complex $\mc{C}(\Sigma)$ is a (combinatorial) simplicial complex and its vertices are given by homotopy classes of essential simple closed curves in $\Sigma$. A $(l-1)$-dimensional simplex $\sigma \subset \mc{C}(\Sigma)$ is then given by a collection $\sigma = \{ \alpha_1, \ldots, \alpha_l\}$ of homotopy classes of essential simple closed curves which are pairwise distinct and admit disjoint representatives. A maximal simplex $\hat\sigma = \{\alpha_1, \ldots, \alpha_N\}$, $N= 3g-3+p$, is a pairs of pants decomposition of $\Sigma$, such that the dimension of $\mc{C}(\Sigma)$ is $3g +p - 4$. 

In the following, we will equip $\Sigma$ with an auxiliary hyperbolic structure. Thus we may assume that every simplex $\sigma \subset \mc{C}(\Sigma)$ consists of the (unique) closed geodesic representatives with respect to that hyperbolic structure. One can check a posteriori that the following definitions are independent of this choice up to natural isomorphisms.

\begin{defn}[Augmented Teichm\"uller Space]
    Let $\sigma \subset \mc{C}(\Sigma)$ be a simplex in the curve complex. We define
    \begin{align*}
        \Teich^*_{\sigma}(\Sigma) = \prod_{\Sigma' \in c(\sigma)} \Teich^*(\Sigma'), \quad \text{ and } \quad
        \Teich_{\sigma}(\Sigma) = \prod_{\Sigma' \in c(\sigma)} \Teich(\Sigma') \subseteq \Teich_\sigma^*(\Sigma)
    \end{align*}
    where the product is taken over all connected components $c(\sigma)$ of $\Sigma \setminus \sigma$.
    
    The disjoint union over all simplices $\sigma \subset \mc{C}(\Sigma)$
    \[\augTeich(\Sigma) = \bigsqcup_{\sigma \subset \mc{C}(\Sigma)} \Teich_{\sigma}(\Sigma)\]
    is then called the \emph{augmented Teichm\"uller space of $\Sigma$}.
\end{defn}

\begin{figure}[t]
    \begin{subfigure}[b]{.9 \linewidth}
        \centering
        \def\svgwidth{0.6\linewidth}
\begingroup%
  \makeatletter%
  \providecommand\color[2][]{%
    \errmessage{(Inkscape) Color is used for the text in Inkscape, but the package 'color.sty' is not loaded}%
    \renewcommand\color[2][]{}%
  }%
  \providecommand\transparent[1]{%
    \errmessage{(Inkscape) Transparency is used (non-zero) for the text in Inkscape, but the package 'transparent.sty' is not loaded}%
    \renewcommand\transparent[1]{}%
  }%
  \providecommand\rotatebox[2]{#2}%
  \newcommand*\fsize{\dimexpr\f@size pt\relax}%
  \newcommand*\lineheight[1]{\fontsize{\fsize}{#1\fsize}\selectfont}%
  \ifx\svgwidth\undefined%
    \setlength{\unitlength}{375.73034572bp}%
    \ifx\svgscale\undefined%
      \relax%
    \else%
      \setlength{\unitlength}{\unitlength * \real{\svgscale}}%
    \fi%
  \else%
    \setlength{\unitlength}{\svgwidth}%
  \fi%
  \global\let\svgwidth\undefined%
  \global\let\svgscale\undefined%
  \makeatother%
  \begin{picture}(1,0.42121094)%
    \lineheight{1}%
    \setlength\tabcolsep{0pt}%
    \put(0,0){\includegraphics[width=\unitlength,page=1]{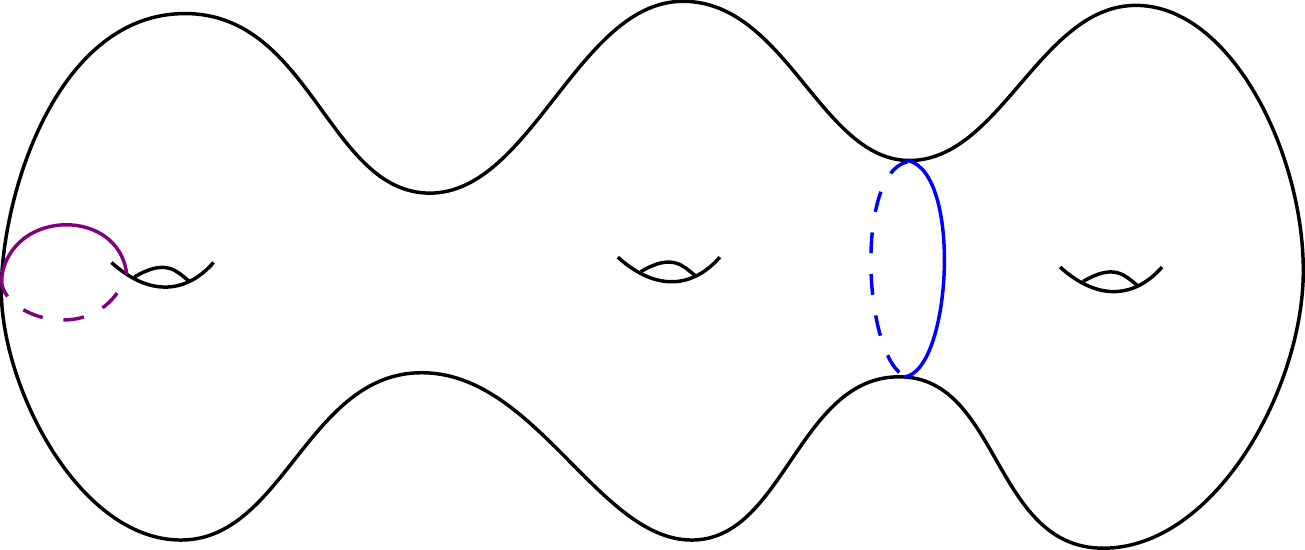}}%
    \put(0.10677082,0.34293104){\color[rgb]{0,0,0}\makebox(0,0)[lt]{\lineheight{1.25}\smash{\begin{tabular}[t]{l}$\times$\end{tabular}}}}%
    \put(0.29559807,0.20414298){\color[rgb]{0,0,0}\makebox(0,0)[lt]{\lineheight{1.25}\smash{\begin{tabular}[t]{l}$\Sigma_1$\end{tabular}}}}%
    \put(0.03218404,0.27023252){\color[rgb]{0.50196078,0,0.50196078}\makebox(0,0)[lt]{\lineheight{1.25}\smash{\begin{tabular}[t]{l}$\alpha_1$\end{tabular}}}}%
    \put(0.83989218,0.27289966){\color[rgb]{0,0,0}\makebox(0,0)[lt]{\lineheight{1.25}\smash{\begin{tabular}[t]{l}$\Sigma_2$\end{tabular}}}}%
    \put(0.72373594,0.26201464){\color[rgb]{0,0,1}\makebox(0,0)[lt]{\lineheight{1.25}\smash{\begin{tabular}[t]{l}$\alpha_2$\end{tabular}}}}%
  \end{picture}%
\endgroup%

        \caption{Let $\Sigma$ be the surface of genus $3$ with one puncture. We consider the simplex $\sigma = \{ {\color{inkscapepurple}\alpha_1}, {\color{inkscapeblue} \alpha_2} \} \subset \mc{C}(\Sigma)$, with components $c(\sigma)=\{ \Sigma_1, \Sigma_2\}$.}
        \label{subfig:ex_aug_teich:a}
    \end{subfigure}
\begin{subfigure}[b]{.9 \linewidth} 
    \centering
    \def\svgwidth{0.9\linewidth}
    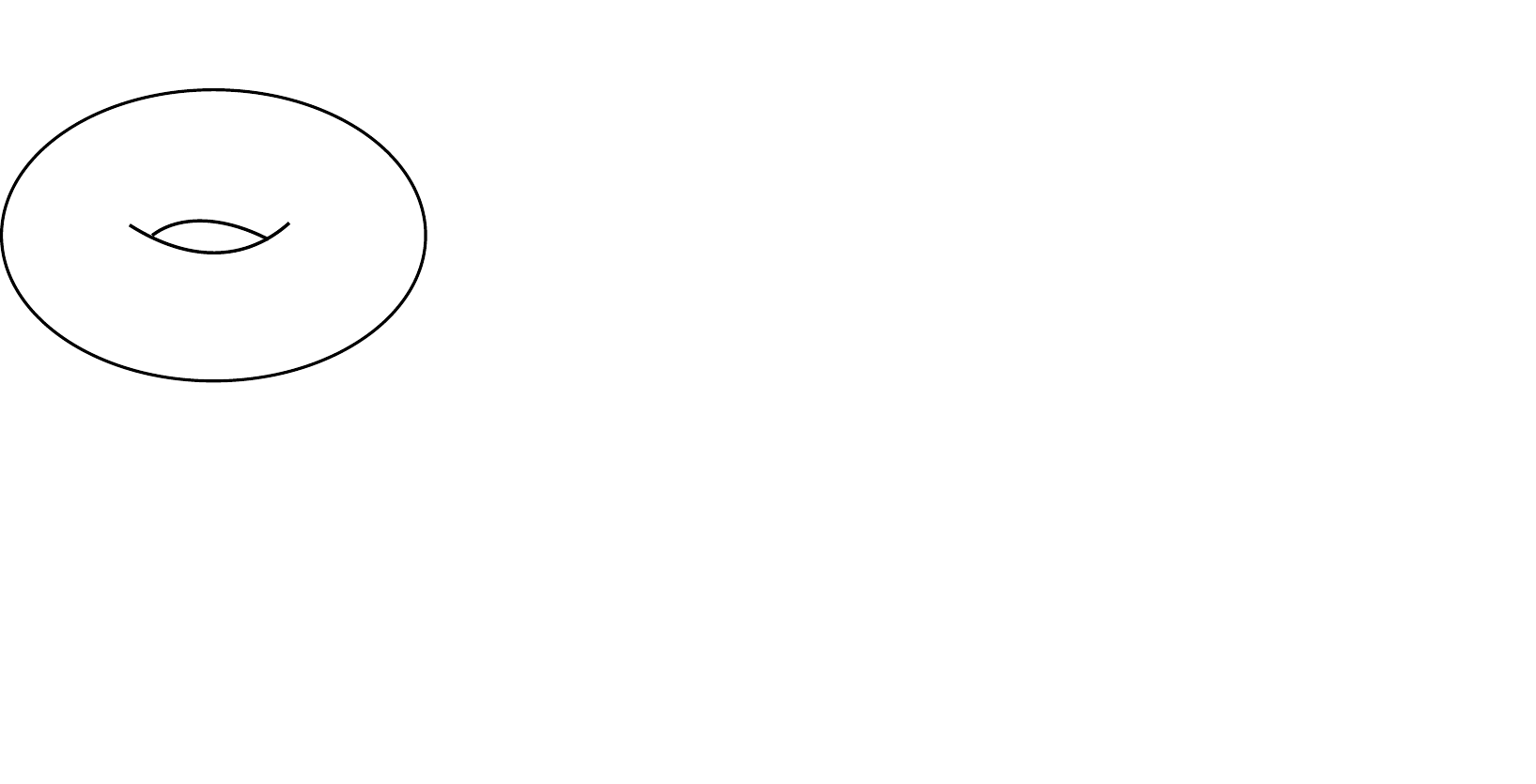
    \caption{Let $\mf{r} = ([\rho_{\Sigma_1}], [\rho_{\Sigma_2}]) \in \Teich_{\sigma}(\Sigma) \subset \augTeich(\Sigma)$ be a point in augmented Teichm\"uller space, where $[\rho_{\Sigma_i}] \in \Teich(\Sigma_i)$, $i=1,2$. Because the representation $\rho_{\Sigma_i}$ is admissible there is an orientation preserving homeomorphism $f_{\Sigma_i} \colon \Sigma_1 \longrightarrow \Gamma_i \backslash \HH^2$, $\Gamma_i \coloneqq \rho_{\Sigma_i}(\pi_1(\Sigma_i))$, with holonomy $(f_{\Sigma_i})_* = \rho_{\Sigma_i}$, $i=1,2$.} 
\end{subfigure}
    \begin{subfigure}[b]{.9 \linewidth} 
        \centering
        \def\svgwidth{0.7\linewidth}
\begingroup%
  \makeatletter%
  \providecommand\color[2][]{%
    \errmessage{(Inkscape) Color is used for the text in Inkscape, but the package 'color.sty' is not loaded}%
    \renewcommand\color[2][]{}%
  }%
  \providecommand\transparent[1]{%
    \errmessage{(Inkscape) Transparency is used (non-zero) for the text in Inkscape, but the package 'transparent.sty' is not loaded}%
    \renewcommand\transparent[1]{}%
  }%
  \providecommand\rotatebox[2]{#2}%
  \newcommand*\fsize{\dimexpr\f@size pt\relax}%
  \newcommand*\lineheight[1]{\fontsize{\fsize}{#1\fsize}\selectfont}%
  \ifx\svgwidth\undefined%
    \setlength{\unitlength}{572.80020166bp}%
    \ifx\svgscale\undefined%
      \relax%
    \else%
      \setlength{\unitlength}{\unitlength * \real{\svgscale}}%
    \fi%
  \else%
    \setlength{\unitlength}{\svgwidth}%
  \fi%
  \global\let\svgwidth\undefined%
  \global\let\svgscale\undefined%
  \makeatother%
  \begin{picture}(1,0.46391745)%
    \lineheight{1}%
    \setlength\tabcolsep{0pt}%
    \put(0,0){\includegraphics[width=\unitlength,page=1]{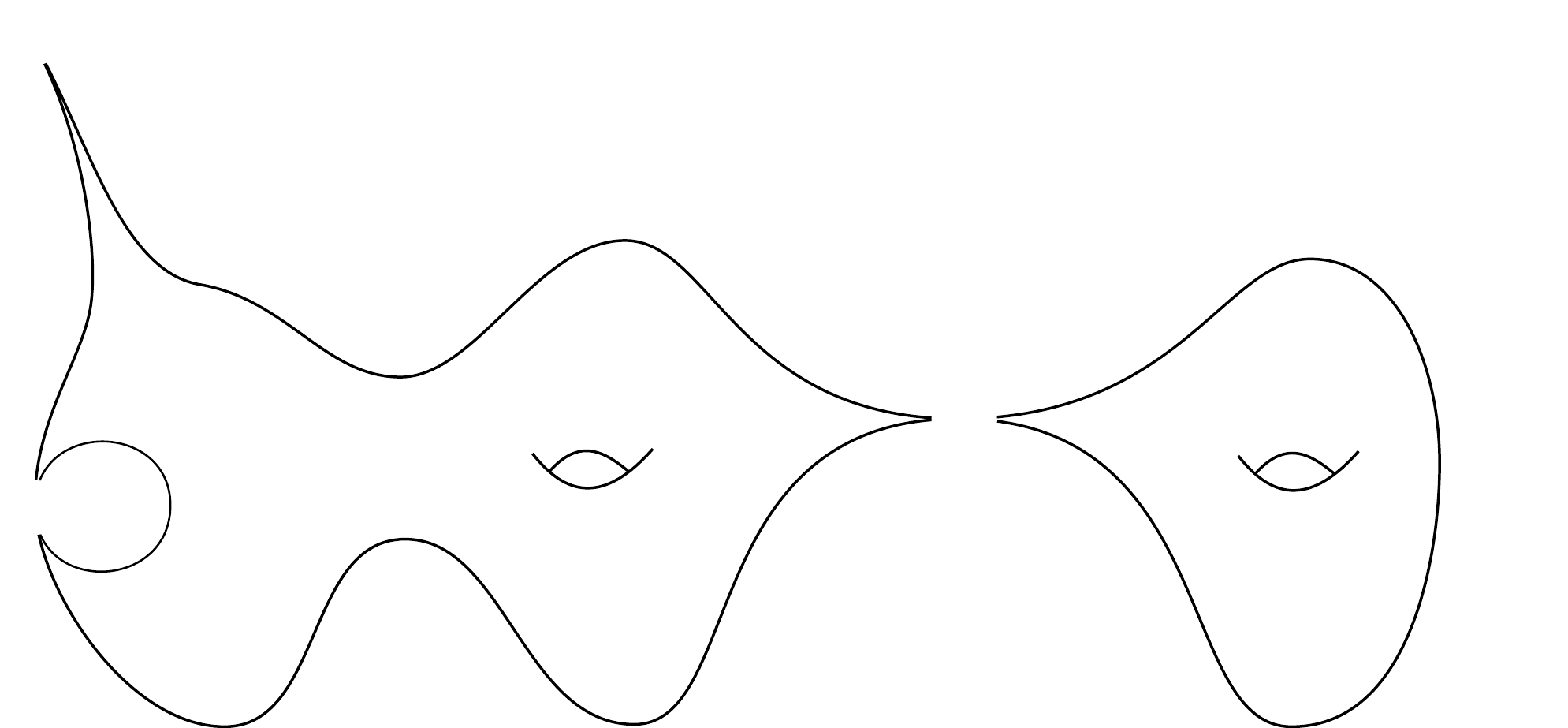}}%
    \put(0.03427067,0.13560011){\color[rgb]{0.50196078,0,0.50196078}\makebox(0,0)[lt]{\lineheight{1.25}\smash{\begin{tabular}[t]{l}$\alpha_1$\end{tabular}}}}%
    \put(0.599206,0.21168021){\color[rgb]{0,0,1}\makebox(0,0)[lt]{\lineheight{1.25}\smash{\begin{tabular}[t]{l}$\alpha_2$\end{tabular}}}}%
    \put(0.23139012,0.18165634){\color[rgb]{0,0,0}\makebox(0,0)[lt]{\lineheight{1.25}\smash{\begin{tabular}[t]{l}$[\rho_{\Sigma_1}]$\end{tabular}}}}%
    \put(0.78823174,0.21888703){\color[rgb]{0,0,0}\makebox(0,0)[lt]{\lineheight{1.25}\smash{\begin{tabular}[t]{l}$[\rho_{\Sigma_2}]$\end{tabular}}}}%
    \put(0,0){\includegraphics[width=\unitlength,page=2]{summary.pdf}}%
  \end{picture}%
\endgroup%

        \caption{The above data is summarized in this picture.} 
        \label{subfig:summary}
    \end{subfigure}
    \caption{}
\label{fig:ex_aug_teich}
\end{figure}

It is important to note that for a point $\mf{r} = ([\rho_{\Sigma'}])_{\Sigma' \in c(\Sigma)} \in \Teich_{\sigma}(\Sigma) \subset \augTeich(\Sigma)$ the simplex $\sigma \subset \mc{C}(\Sigma)$ is implicit as is the decomposition of $\Sigma$ into the components $\{ \Sigma' \}_{\Sigma' \in c(\sigma)}$. See Figure \ref{fig:ex_aug_teich} for an example.

We wish to equip augmented Teichm\"uller space with a topology. In order to do so we will need \emph{restriction maps}. Let $\sigma \subset \mc{C}(\Sigma)$ be a simplex in the curve complex and $\Sigma' \in c(\sigma)$ a connected component of $\Sigma \setminus \sigma$. Let $\pi \colon \HH^2 \cong \tilde{\Sigma} \longrightarrow \Sigma$ denote the universal covering; recall that $\Sigma$ carries an auxiliary hyperbolic structure. Denote by $\tilde{\sigma} := \pi^{-1}(\sigma) \subset \HH^2$ the disjoint union of geodesics that project to $\sigma \subset \Sigma$. Further, let $\tilde{\Sigma}' \subset \HH^2 \setminus \tilde{\sigma}$ be a connected component that projects to $\Sigma'$, i.e.\ $\pi(\tilde{\Sigma}') = \Sigma'$. Observe that $\tilde{\Sigma}' \subset \HH^2$ is a convex subset such that $\pi|_{\tilde{\Sigma}'} \colon \tilde{\Sigma}' \longrightarrow \Sigma'$ is a universal covering; see Figure \ref{fig:restriction_map}. Thus, we obtain the following commutative diagram:
\begin{center}
    \begin{tikzcd}[column sep = large, row sep = large]
        \tilde{\Sigma}' \arrow[r,hook] \arrow[d,"\pi|_{\Sigma'}"] 
        & \tilde{\Sigma} \arrow[d,"\pi"] \\
        \Sigma' \arrow[r, hook] & \Sigma
    \end{tikzcd}
\end{center}
It follows that the homomorphism $\iota_{\Sigma'} \colon \pi_1(\Sigma') \longrightarrow \pi_1(\Sigma)$ induced by inclusion is injective, and identifies $\pi_1(\Sigma')$ with the subgroup of $\pi_1(\Sigma) \cong \Deck(\pi)$ that leaves the component $\tilde{\Sigma}'$ invariant.

\begin{figure}[h!]
    \def\svgwidth{0.9\linewidth}
\begingroup%
  \makeatletter%
  \providecommand\color[2][]{%
    \errmessage{(Inkscape) Color is used for the text in Inkscape, but the package 'color.sty' is not loaded}%
    \renewcommand\color[2][]{}%
  }%
  \providecommand\transparent[1]{%
    \errmessage{(Inkscape) Transparency is used (non-zero) for the text in Inkscape, but the package 'transparent.sty' is not loaded}%
    \renewcommand\transparent[1]{}%
  }%
  \providecommand\rotatebox[2]{#2}%
  \newcommand*\fsize{\dimexpr\f@size pt\relax}%
  \newcommand*\lineheight[1]{\fontsize{\fsize}{#1\fsize}\selectfont}%
  \ifx\svgwidth\undefined%
    \setlength{\unitlength}{453.66247559bp}%
    \ifx\svgscale\undefined%
      \relax%
    \else%
      \setlength{\unitlength}{\unitlength * \real{\svgscale}}%
    \fi%
  \else%
    \setlength{\unitlength}{\svgwidth}%
  \fi%
  \global\let\svgwidth\undefined%
  \global\let\svgscale\undefined%
  \makeatother%
  \begin{picture}(1,1.31290114)%
    \lineheight{1}%
    \setlength\tabcolsep{0pt}%
    \put(0,0){\includegraphics[width=\unitlength,page=1]{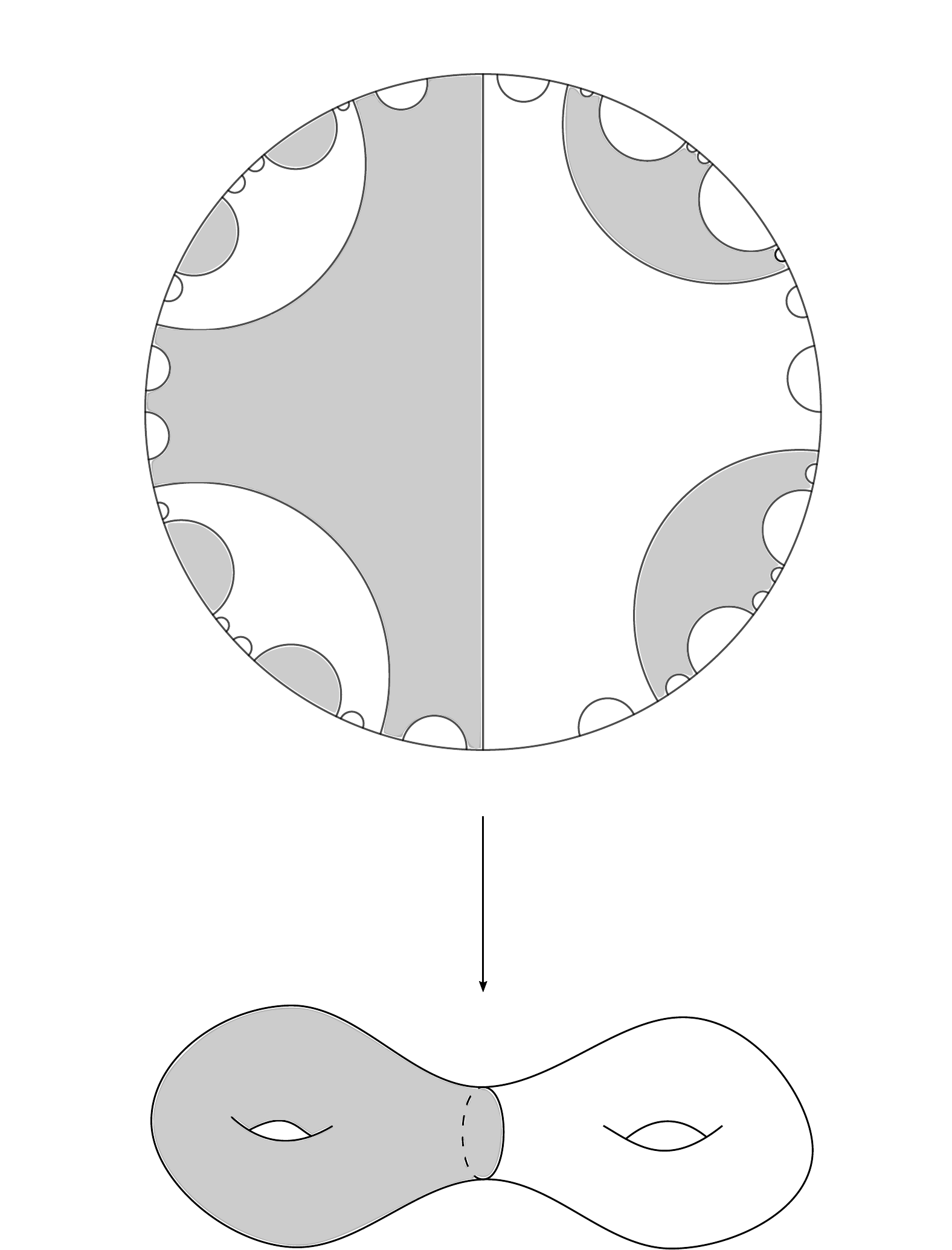}}%
    \put(0.3292818,0.92027407){\color[rgb]{0,0,0}\makebox(0,0)[lt]{\lineheight{1.25}\smash{\begin{tabular}[t]{l}$\tilde{\Sigma}'$\end{tabular}}}}%
    \put(0.30027724,0.1919368){\color[rgb]{0,0,0}\makebox(0,0)[lt]{\lineheight{1.25}\smash{\begin{tabular}[t]{l}$\Sigma'$\end{tabular}}}}%
    \put(0.51724833,0.36219538){\color[rgb]{0,0,0}\makebox(0,0)[lt]{\lineheight{1.25}\smash{\begin{tabular}[t]{l}$\pi$\end{tabular}}}}%
    \put(0.5431007,0.13204381){\color[rgb]{0,0,0}\makebox(0,0)[lt]{\lineheight{1.25}\smash{\begin{tabular}[t]{l}$\sigma= \{ \alpha\}$\end{tabular}}}}%
    \put(0.53988312,0.8447411){\color[rgb]{0,0,0}\makebox(0,0)[lt]{\lineheight{1.25}\smash{\begin{tabular}[t]{l}$\tilde{\sigma}=\pi^{-1}(\sigma)$\end{tabular}}}}%
  \end{picture}%
\endgroup%

    \caption{We are considering a genus $2$ surface $\Sigma$ and a simplex $\sigma = \{\alpha\} \subset \mc{C}(\Sigma)$ consisting of one separating curve $\alpha \subset \Sigma$. The preimage $\pi^{-1}(\Sigma')$ of the component $\Sigma' \in c(\sigma)$ is shaded. We chose one connected component $\tilde{\Sigma}'\subset \HH^2 \setminus \tilde{\sigma}$.} 
    \label{fig:restriction_map}
\end{figure}
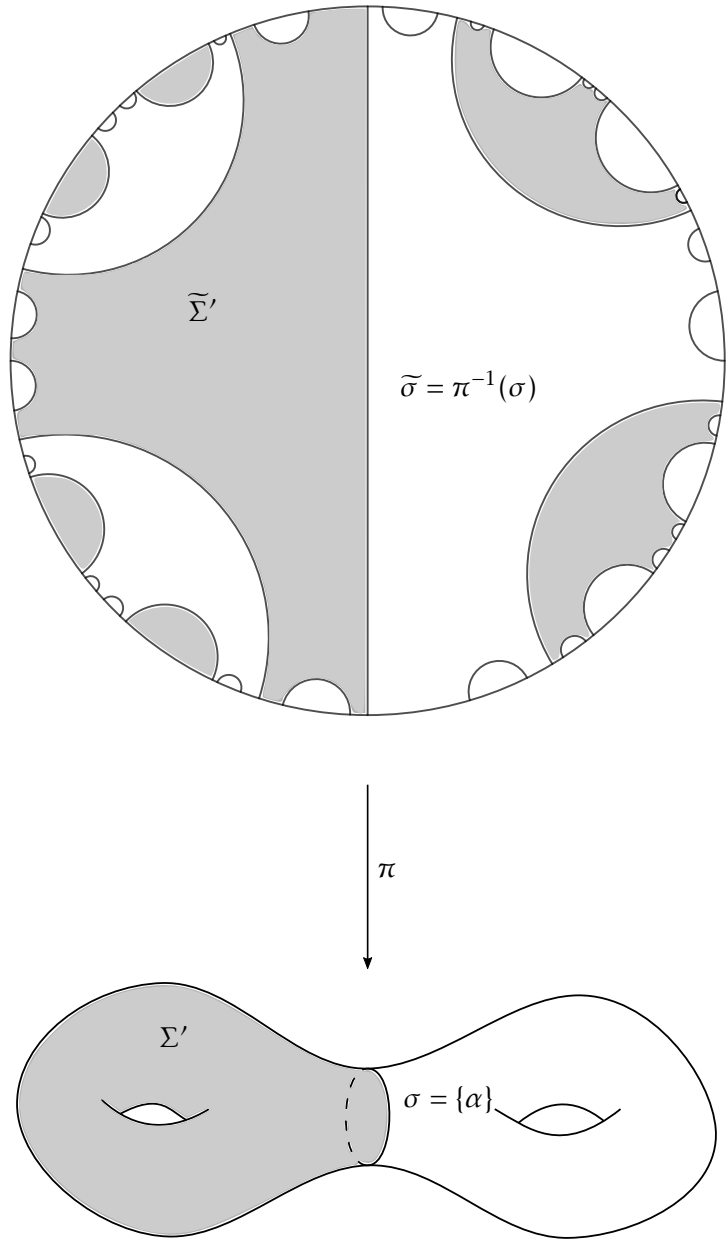

\begin{remark} \label{rem:DependenceOnComponent}
    The monomorphism $\iota_{\Sigma'} \colon \pi_1(\Sigma') \hookrightarrow \pi_1(\Sigma)$ depends on the choice of connected component $\tilde{\Sigma}' \subset \pi^{-1}(\Sigma')$.
    Different choices amount to monomorphisms $\pi_1(\Sigma') \hookrightarrow \pi_1(\Sigma)$ which are conjugate in $\pi_1(\Sigma)$.
\end{remark}

\begin{remark}
	Although we will not need this in the following, we want to mention that every simplex $\sigma \subset \mc{C}(\Sigma)$ gives rise to a graph of groups structure on $\pi_1(\Sigma)$; see \cite{serre}. Indeed, $\pi_1(\Sigma)$ is the fundamental group of the graph of groups whose vertices are the fundamental groups $\pi_1(\Sigma')$ of the components $\Sigma' \in c(\sigma)$. Identifying the peripheral subgroups corresponding to curves in $\sigma$ then amounts to the edge homomorphisms. 
	
	Our discussion of the inclusion homomorphisms already hints at how to obtain a tree on which $\pi_1(\Sigma)$ acts: The preimage $\tilde{\sigma} = \pi^{-1}(\sigma)$ decomposes $\tilde{\Sigma}$ into convex domains. We obtain a simplicial tree whose vertices are the convex domains and where any two vertices are connected by an edge if the corresponding domains are adjacent. The fundamental group $\pi_1(\Sigma)$ acts on this tree with vertex stabilizers isomorphic to the fundamental groups of the corresponding components $\Sigma' \in c(\sigma)$ and edge stabilizers isomorphic to $\ZZ$. 
\end{remark}

\begin{propdefn}\label{propdefn:RestrictionMaps}
    In the above situation, we obtain a well-defined \emph{restriction map}
    \begin{align*}
    \res^{\Sigma}_{\Sigma'} \colon \Teich^*(\Sigma) &\longrightarrow \Teich^*(\Sigma'),\\
    [\rho]  &\longmapsto [\rho \circ \iota_{\Sigma'}].
    \end{align*}
    
    For a face $\sigma' \subseteq \sigma \subset \mc{C}(\Sigma)$ these maps induce a restriction map
    \begin{align*}
        \res^{\sigma'}_{\sigma} \colon \Teich^*_{\sigma'}(\Sigma) &\longrightarrow \Teich^*_{\sigma}(\Sigma), \\
        ([\rho_{\Sigma''}])_{\Sigma'' \in c(\sigma')} &\longmapsto \p{\res^{\Sigma''}_{\Sigma'}([\rho_{\Sigma''}])}_{\Sigma' \in c(\sigma)},
    \end{align*}
    where on the right-hand-side $\Sigma''$ is the unique connected component that contains $\Sigma'$.
\end{propdefn}

\begin{proof}

Let $[\rho] \in \Teich^*(\Sigma)$ with $\rho \in \Rep^*(\Sigma)$. Consider the composition $\rho \circ \iota_{\Sigma'} \colon \pi_1(\Sigma') \hookrightarrow G$. Then the image $\Gamma' \coloneqq \rho(\iota_{\Sigma'}(\pi_1(\Sigma'))) \leq \Gamma \coloneqq \rho(\pi_1(\Sigma))$ is discrete, such that $\rho \circ \iota_{\Sigma'}$ is a discrete and faithful representation. Our goal is to show that $\rho \circ \iota_{\Sigma'} \in \Rep^*(\Sigma')$. It will then be immediate that
\begin{align*}
    \res^{\Sigma}_{\Sigma'} \colon \Teich^*(\Sigma) &\longrightarrow \Teich^*(\Sigma'),\\
    [\rho]  &\longmapsto [\rho \circ \iota_{\Sigma'}],
\end{align*}
is a well-defined map.
Indeed, by Remark \ref{rem:DependenceOnComponent} the injective homomorphism $\iota_{\Sigma'}$ is well-defined only up to conjugation in $\pi_1(\Sigma)$. However, this issue is resolved after taking the quotient with respect to the conjugation action of $G$ on $\Rep^*(\Sigma)$.

We are left to show that $\rho \circ \iota_{\Sigma'}$ is the holonomy of an orientation preserving homeomorphism $f' \colon \Sigma' \longrightarrow \Gamma' \backslash \HH^2$. 
Let $f \colon \Sigma \longrightarrow X \coloneqq \Gamma \backslash \HH^2$ be an orientation preserving homeomorphism with holonomy $\rho$. We may isotope $f$ in such a way that it sends the curves in $\sigma$ to geodesics $f(\sigma) \subseteq X$. Therefore, it sends $\Sigma'$ to a connected component $X'$ of $X \setminus f(\sigma)$. Consider a lift $\tilde{f} \colon \tilde{\Sigma} \longrightarrow \HH^2$ of $f \colon \Sigma \longrightarrow X$ with respect to the universal coverings $\pi \colon \tilde{\Sigma} \longrightarrow \Sigma$ and $\pi_\Gamma \colon \HH^2 \longrightarrow \Gamma \backslash \HH^2$. Then $\tilde{X}' \coloneqq \tilde{f}(\tilde{\Sigma}') \subseteq \HH^2 \setminus \pi_{\Gamma}^{-1}(f(\sigma))$ is a connected component, and $\tilde{f}|_{\tilde{\Sigma}'} \colon \tilde{\Sigma}' \longrightarrow \tilde{X}'$ is $(\rho \circ \iota_{\Sigma'})$-equivariant. Thus, it descends to an orientation preserving homeomorphism
$f' \colon \Sigma' \longrightarrow \Gamma' \backslash \tilde{X}' \cong X'$, $\Gamma' \coloneqq (\rho \circ \iota_{\Sigma'})(\pi_1(\Sigma'))$. 
The complement $\HH^2 \setminus \tilde{X}'$ is a disjoint union of closed half-spaces $\{H_j\}_{j \in \NN}$ each bordering on a geodesic of $\tilde{f}(\tilde{\sigma})$ adjacent to $\tilde{X}'$. 

We want to understand the action of $\Gamma'$ on each half-space $H_j$,$j \in \NN$. Note that the disjoint union $\bigsqcup_{j \in \NN} H_j$ is $\Gamma'$-invariant such that $\Gamma'$ acts via permutations on $\{ H_j \}_{j \in \NN}$. Thus, if $\gamma \in \Gamma'$ is an element such that $\gamma H_j \cap H_j \neq \emptyset$ then $\gamma H_j = H_j$. If $I_j = \partial H_j \subseteq \partial \HH^2 \cong \SS^1$ denotes the interval that $H_j$ borders on then $\gamma$ has to fix $I_j$. Because $\Gamma'$ is torsion-free it does not contain any elliptic elements such that $\gamma$ must be a hyperbolic element that fixes the end points of $I_j$. By discreteness of $\Gamma'$ there is a hyperbolic element $\gamma_j \in \Gamma'$ for every $j \in \NN$ such that any element $\gamma \in \Gamma'$ satisfying $\gamma H_j \cap H_j \neq \emptyset$ is a power of $\gamma_j$. It follows that the quotient of $H_j$ under the quotient map $\pi' \colon \HH^2 \longrightarrow \Gamma' \backslash \HH^2$ is a hyperbolic funnel $F_j \coloneqq \pi'(H_j) \cong \langle \gamma_j \rangle \backslash H_j$.

Therefore, the complement of $\Gamma' \backslash \tilde{X}'$ in $Y' = \Gamma' \backslash \HH^2$ is a disjoint union of hyperbolic funnels. In particular, $Y'$ deformation retracts to $\Gamma' \backslash \tilde{X}'$, and we can easily modify $f'$ to an orientation preserving homeomorphism $f' \colon \Sigma' \longrightarrow Y'$ whose holonomy is $\rho \circ\iota_{\Sigma'}$. We conclude that $\rho \circ \iota_{\Sigma'} \in \Rep^*(\Sigma')$, and $\res^{\Sigma}_{\Sigma'}$ is well-defined.

\end{proof}

We may now define a topology on $\augTeich(\Sigma)$.

\begin{defn}[Topology on $\augTeich(\Sigma)$]
    For every $\mf{r} = ([\rho_{\Sigma'}])_{\Sigma' \in c(\sigma)} \in \Teich_\sigma(\Sigma)$ we define a system of open neighborhoods by
    \[ (\res^{\sigma'}_\sigma)^{-1}(U) \cap \Teich_{\sigma'}(\Sigma) \]
    where $\sigma' \subseteq \sigma$ and $U \subseteq \Teich^*_\sigma(\Sigma) =  \prod_{\Sigma' \in c(\sigma)} \Teich^*(\Sigma')$ runs over all open neighborhoods of $\mf{r} \in \Teich^*_\sigma(\Sigma)$ in the product topology. This system of neighborhoods defines a topology on $\augTeich(\Sigma)$.
    
    A sequence $(\mf{r}^{(n)}) \subset \augTeich(\Sigma)$ converges to 
    $\mf{r} \in \Teich_{\sigma}(\Sigma)$ if and only if $\mf{r}^{(n)} \in \Teich_{\sigma_n}(\Sigma)$ with $\sigma_n \subseteq \sigma$ for large $n$, and
    $$ \res^{\sigma_n}_{\sigma}(\mf{r}^{(n)}) \to \mf{r} \qquad (n \to \infty) $$
    in $\Teich^*_{\sigma}(\Sigma)$.
\end{defn}

\begin{remark}
    In \cite{abikoff} the topology on augmented Teichm\"uller space is defined utilizing generalized Fenchel--Nielsen coordinates. However, both definitions yield the same topology.
    
    We want to mention that Loftin and Zhang \cite{loftinzhang} define a topology on the (larger) augmented deformation space of convex real projective structures quite similarly.
\end{remark}

\begin{prop}[\cite{abikoff}]
    The augmented Teichm\"uller space $\augTeich(\Sigma)$ is a first-countable Hausdorff space, and $\Teich(\Sigma) = \Teich_{\emptyset}(\Sigma) \subset \augTeich(\Sigma)$ is an open and dense subset.
\end{prop}

\subsection{Augmented Moduli Space} \label{subsect:aug_moduli}

The mapping class group action on Teichm\"uller space extends to $\augTeich(\Sigma)$ in the following way. Let $\varphi = [f] \in \MCG(\Sigma)$ and $\mf{r} = ([\rho_{\Sigma'}])_{\Sigma' \in c(\sigma)} \in \Teich_\sigma(\Sigma) \subset \augTeich(\Sigma)$. The mapping class group acts simplicially on the curve complex $\mc{C}(\Sigma)$ such that  $\varphi^{-1}(\sigma) \subset \mc{C}(\Sigma)$ is another simplex in the curve complex. Up to an isotopy we may assume that $f^{-1}$ sends $\sigma$ to a geodesic representative of $f^{-1}(\sigma)$. Hence, $f^{-1}$ induces a bijection between the components $c(\sigma)$ and $c(f^{-1}(\sigma)) = f^{-1}(c(\sigma))$, and acts from the right via restriction:
$$ \mf{r} \cdot \varphi 
\coloneqq ([\rho_{f(\Sigma')} \circ (f|_{\Sigma'})_*])_{\Sigma' \in c(f^{-1}(\sigma))}. $$
The cutting homomorphism ensures that the action is well-defined; see \cite[section 3.6.3]{farbmarg}. By definition this action extends the mapping class group action on $\Teich(\Sigma)$ such that the embedding 
$\Teich(\Sigma) \hookrightarrow \augTeich(\Sigma) $
is $\MCG(\Sigma)$-equivariant.

Note that this action permutes the different strata $\{ \Teich_{\sigma}(\Sigma)\}_{\sigma \subset \mc{C}(\Sigma)}$ of augmented Teichm\"uller space:
\[\Teich_{\sigma}(\Sigma) \stackrel{\phi}{\longrightarrow} \Teich_{\phi^{-1}(\sigma)}(\Sigma), \qquad \phi \in \MCG(\Sigma).\]

\begin{defn}
    The quotient space $\augModuli(\Sigma) \coloneqq \augTeich(\Sigma) / \MCG(\Sigma)$ of augmented Teichm\"uller space $\augTeich(\Sigma)$ by the mapping class group action is called \emph{augmented moduli space}.
\end{defn}

\begin{remark}
	Changing perspective one can see the moduli space $\Moduli(\Sigma)$ as the moduli space of smooth genus $g$ curves with $p$ marked points. In this setting, the augmented moduli space $\augModuli(\Sigma)$ corresponds to the Deligne--Mumford compactification of stable curves. We will not use this point of view in what follows and refer the reader to \cite{harveyshort, harvey77} and \cite{hubbardkoch} for details.
\end{remark}

The significance of this whole construction is that the augmented moduli space is \emph{compact}.

\begin{thm}[\cite{abikoff}]
    The augmented moduli space $\augModuli(\Sigma)$ is a compact Hausdorff space. The embedding $\Teich(\Sigma) \hookrightarrow \augTeich(\Sigma)$ descends to an embedding $\Moduli(\Sigma) \hookrightarrow \augModuli(\Sigma)$ with open and dense image.
\end{thm}

\subsubsection*{Assembly Maps}

We will now give an interpretation of how elements in $\augModuli(\Sigma)$ may be assembled from elements in the moduli spaces of the components. In order to explain this, we will need some more notation.

\begin{defn}[Pure Mapping Class Group]
    The subgroup $\PMCG(\Sigma) \leq \MCG(\Sigma)$ of all mapping classes that fix each puncture of $\Sigma$ individually is called the \emph{pure mapping class group}.
\end{defn}

Observe that there is the following short exact sequence
\[ 1 \longrightarrow \PMCG(\Sigma) \longrightarrow \MCG(\Sigma) \longrightarrow \Sym(p) \longrightarrow 1 \]
coming from the action of the mapping class group on the $p$ punctures of $\Sigma$. Here $\Sym(p)$ denotes the symmetric group on $p$ elements. In particular, the pure mapping class group is a normal subgroup of the mapping class group of index $p ! = \# \Sym(p)$.

We may now form a slightly bigger moduli space of hyperbolic structures on $\Sigma$ by keeping track of the punctures individually.

\begin{defn}
    We define the quotient
    \[ \Moduli^*(\Sigma) := \Teich(\Sigma)/\PMCG(\Sigma) \]
    and denote the quotient map by
    \[ \pi_\Sigma \colon \Teich(\Sigma) \longrightarrow \Moduli^*(\Sigma).\]
\end{defn}

Let us now fix a simplex $\sigma \subseteq \mc{C}(\Sigma)$ in the curve complex. We denote by
\[ \PMCG_{\sigma}(\Sigma) \coloneqq \{ \phi \in \PMCG(\Sigma) \colon \phi(\alpha) = \alpha, \text{ for all } \alpha \in \sigma \} \leq \PMCG(\Sigma)\]
the subgroup of mapping classes fixing the homotopy class of each curve of $\sigma$ individually. By definition $\PMCG_\sigma(\Sigma)$ acts on $\Teich_\sigma(\Sigma)$ and we define
\[p_\sigma \colon \Teich_\sigma(\Sigma) \longrightarrow \Moduli^*_{\sigma}(\Sigma) \coloneqq \Teich_\sigma(\Sigma)/\PMCG_\sigma(\Sigma).\]

Observe that if $f \colon \Sigma \longrightarrow \Sigma$ represents a mapping class $[f] \in \PMCG_\sigma(\Sigma)$ then we may isotope $f$ so that it fixes each curve of $\sigma$ individually. Because $f$ is orientation preserving it follows that $f$
also fixes each component $\Sigma' \in c(\sigma)$ individually,
\[ f|_{\Sigma'} \colon \Sigma' \longrightarrow \Sigma'.\]

\begin{lemma}\label{lem:CuttingHomom}
    The map
    \begin{align*}
    \phi_\sigma \colon \PMCG_{\sigma}(\Sigma) &\longrightarrow \prod_{\Sigma' \in c(\sigma)} \PMCG(\Sigma'), \\
    [f] &\longmapsto ([f|_{\Sigma'})_{\Sigma' \in c(\sigma)} 
    \end{align*}
    is a well-defined homomorphism. Moreover, $\phi_\sigma$ fits in the following short exact sequence
    \[ 1 \longrightarrow \ker \phi_\sigma \longrightarrow \PMCG_{\sigma}(\Sigma) \stackrel{\phi_{\sigma}}{\longrightarrow} \prod_{\Sigma' \in c(\sigma)} \PMCG(\Sigma') \longrightarrow 1 \]
    and $\ker \phi_\sigma$ is generated by Dehn twists about the curves in $\sigma$.
\end{lemma}

\begin{proof}
    The proof that $\phi_\sigma$ is well-defined with kernel generated by Dehn twists about the curves in $\sigma$ is the same as for \cite[Proposition 3.20]{farbmarg}.
    
    We are left with proving that $\phi_\sigma$ surjects onto $\prod_{\Sigma' \in c(\sigma)} \PMCG(\Sigma')$. Let $[f_{\Sigma'}] \in \PMCG(\Sigma')$ for every $\Sigma' \in c(\sigma)$. Let $\{ N_c \colon c \in \sigma \}$ be a collection of disjoint closed tubular neighborhoods about the curves in $\sigma$. 
    
    We define an orientation preserving homeomorphism $f \colon \Sigma \longrightarrow \Sigma$ in the following way. For every $x \in \Sigma' \setminus \bigsqcup_{c \in \sigma} N_c$, $\Sigma' \in c(\sigma)$, we set
    \[ f(x) = f_{\Sigma'}(x).\]
    If $N_c$ is a tubular neighborhood intersecting the components $\Sigma', \Sigma'' \in c(\sigma)$ we may interpolate continuously on $N_c$ such that 
    \begin{align*}
    f(x') = f_{\Sigma'}(x') &\qquad \forall x' \in \partial N_c \cap \Sigma', \\
    f(x'') = f_{\Sigma''}(x'') &\qquad \forall x'' \in \partial N_c \cap \Sigma'',\\
    f(y) = y &\qquad \forall y \in c.
    \end{align*}
    Then $[f] \in \PMCG_\sigma(\Sigma)$ by definition and $f|_{\Sigma'}$ coincides with $f_{\Sigma'}$ outside some disjoint discs about the punctures of $\Sigma'$ for every $\Sigma' \in c(\sigma)$. Because the mapping class group of a punctured disc is trivial it follows that
    $[f|_{\Sigma'}] = [f_{\Sigma'}]$ for every $\Sigma' \in c(\sigma)$.
    
\end{proof}

\begin{prop}
    Define a map
    \begin{align*}
    F_\sigma \colon \Teich_{\sigma}(\Sigma) = \prod_{\Sigma' \in c(\sigma)} \Teich(\Sigma') &\longrightarrow \prod_{\Sigma' \in c(\sigma)} \Moduli^*(\Sigma'),\\
    ([\rho_{\Sigma'}])_{\Sigma' \in c(\sigma)} &\longmapsto (\pi_{\Sigma'}([\rho_{\Sigma'}]))_{\Sigma' \in c(\sigma)}.
    \end{align*}
    Then $F_\sigma$ descends to a homeomorphism:
    
    \begin{center}
        \begin{tikzcd}
            & \Teich_{\sigma}(\Sigma) \arrow[dl,"p_\sigma"] \arrow[dr,"F_\sigma"] & \\
            \Moduli^*_\sigma(\Sigma) \arrow[rr,dashed,"\ol{F}_\sigma","\cong"'] & & \displaystyle \prod_{\Sigma' \in c(\sigma)} \Moduli^*(\Sigma')
        \end{tikzcd} 
    \end{center}
\end{prop}
\begin{proof}
    Let us see that $\ol{F}_\sigma$ is well-defined. Let $([\rho_{\Sigma'}])_{\Sigma' \in c(\sigma)}, ([\rho'_{\Sigma'}])_{\Sigma' \in c(\sigma)} \in \Teich_{\sigma}(\Sigma)$ such that $p_{\sigma}(([\rho_{\Sigma'}])_{\Sigma' \in c(\sigma)}) = p_{\sigma}(([\rho'_{\Sigma'}])_{\Sigma' \in c(\sigma)})$. Then there is $[f] \in \PMCG_\sigma(\Sigma)$ such that 
    \[([\rho'_{\Sigma'}])_{\Sigma' \in c(\sigma)} = ([\rho_{\Sigma'}])_{\Sigma' \in c(\sigma)} \cdot  [f] 
    =([\rho_{\Sigma'} \circ (f|_{\Sigma'})_*])_{\Sigma' \in c(\sigma)}
    =( [\rho_{\Sigma'}] \cdot [f|_{\Sigma'}])_{\Sigma' \in c(\sigma)}.\]
    Thus $F_{\sigma}(([\rho_{\Sigma'}])_{\Sigma' \in c(\sigma)}) = F_\sigma(([\rho'_{\Sigma'}])_{\Sigma' \in c(\sigma)})$ as required. 
    
    Let us see that $\ol{F}_\sigma$ is injective. Let $([\rho_{\Sigma'}])_{\Sigma' \in c(\sigma)},([\rho'_{\Sigma'}])_{\Sigma' \in c(\sigma)} \in \Teich_{\sigma}(\Sigma)$ such that $F_\sigma(([\rho_{\Sigma'}])_{\Sigma' \in c(\sigma)}) = F_\sigma(([\rho'_{\Sigma'}])_{\Sigma' \in c(\sigma)})$. Then there are $[f_{\Sigma'}] \in \PMCG(\Sigma')$ such that $[\rho'_{\Sigma'}] = [\rho_{\Sigma'}] \cdot [f_{\Sigma'}]$ for every $\Sigma' \in c(\sigma)$. 
    By Lemma \ref{lem:CuttingHomom} there is an $[f] \in \PMCG_{\sigma}(\Sigma)$ such that $\phi_\sigma([f])=([f|_{\Sigma'}])_{\Sigma' \in c(\sigma)} = ([f_{\Sigma'}])_{\Sigma' \in c(\sigma)}$. Thus, $([\rho_{\Sigma'}])_{\Sigma' \in c(\sigma)} \cdot [f]  = ([\rho'_{\Sigma'}])_{\Sigma' \in c(\sigma)}$ and $p_\sigma(([\rho_{\Sigma'}])_{\Sigma' \in c(\sigma)}) = p_\sigma(([\rho'_{\Sigma'}])_{\Sigma' \in c(\sigma)})$.
    
    Finally, $\ol{F}_\sigma$ is surjective because $F_\sigma$ is. Moreover, because $F_\sigma$ and $p_\sigma$ are both open and continuous maps, also $\ol{F}_\sigma$ is open and continuous, whence $\ol{F}_\sigma$ is a homeomorphism.
\end{proof}

Projecting $\Teich_{\sigma}(\Sigma) \subset \augTeich(\Sigma)$ to the augmented moduli space $\augModuli(\Sigma)$ we obtain a continuous map $\Teich_{\sigma}(\Sigma) \longrightarrow \augModuli(\Sigma)$. This map descends to a continuous map $\Moduli_\sigma^*(\Sigma) \longrightarrow \augModuli(\Sigma)$. Using the homeomorphism $\ol{F}_{\sigma}$ we obtain a continuous map $A_{\sigma} \colon \prod_{\Sigma' \in c(\sigma)} \Moduli^*(\Sigma') \longrightarrow \augModuli(\Sigma)$, such that the following diagram commutes:
\begin{center}
    \begin{tikzcd}[row sep = large, column sep = large]
        \Teich_{\sigma}(\Sigma) \arrow[r] \arrow[d,"p_\sigma"] 
            & \augModuli(\Sigma) \\
        \Moduli^*_{\sigma}(\Sigma) \arrow[r,"\ol{F}_\sigma","\cong"'] 
            \arrow[ur]
            & \displaystyle \prod_{\Sigma' \in c(\sigma)} \Moduli^*(\Sigma') \arrow[u,dashed,"A_\sigma"]
    \end{tikzcd}
    
\end{center}

\begin{defn} \label{defn:AssemblyMap}
    We will call the map $A_{\sigma} \colon \prod_{\Sigma' \in c(\sigma)} \Moduli^*(\Sigma') \longrightarrow \augModuli(\Sigma)$ the \emph{assembly map with respect to $\sigma \subset \mc{C}(\Sigma)$}. 
\end{defn} 

It will be useful to know that the augmented moduli space is covered by the images of a \emph{finite number} of assembly maps. The key here is the following observation due to Harvey \cite{harveyCurveComplex}.

\begin{lemma}[\cite{harveyCurveComplex}]
    The mapping class group $\MCG(\Sigma)$ acts on the curve complex $\mc{C}(\Sigma)$ simplicially and the quotient is a \emph{finite} simplicial complex.
\end{lemma}

One may think of this statement as the fact that one can dissect a given surface $\Sigma$ only in finitely many different ways up to homeomorphism.

\begin{prop}\label{prop:FinitelyCoveredByAssemblyMaps}
    Let $\{ \sigma_i \colon i \in I\}$ be a finite system of representatives for the simplices in the curve complex $\mc{C}(\Sigma)$ with respect to the mapping class group action. Then augmented moduli space is covered by the images of the assembly maps $A_{\sigma_i}$, $i \in I$.
\end{prop}
\begin{proof}
    Because $\{ \sigma_i \colon i \in I\}$ is a system of representatives the projection 
    \[ \bigsqcup_{i \in I} \Teich_{\sigma_i}(\Sigma) \longrightarrow \augModuli(\Sigma)\]
    is surjective. We obtain as above a commutative diagram:
    \begin{center}
        \begin{tikzcd}[row sep = large, column sep = large]
            \displaystyle \bigsqcup_{i \in I} \Teich_{\sigma_i}(\Sigma) \arrow[r,two heads] \arrow[d] 
            & \augModuli(\Sigma) \\
            \displaystyle \bigsqcup_{i \in I} \Moduli^*_{\sigma_i}(\Sigma) \arrow[r,"\cong"] 
            \arrow[ur,two heads]
            & \displaystyle \bigsqcup_{i \in I} \prod_{\Sigma' \in c(\sigma_i)} \Moduli^*(\Sigma') \arrow[u,two heads, "A"]
        \end{tikzcd}
        
    \end{center}
    Here the map $A$ is defined as $A_{\sigma_i}$ on each $\prod_{\Sigma' \in c(\sigma_i)} \Moduli^*(\Sigma')$, $i\in I$, such that augmented moduli space is indeed covered by the images of all $\{A_{\sigma_i} :i\in I \}$.
    
\end{proof}


\section{The Space of Closed Subgroups} \label{sect:space_of_closed_subgroups}

Previously, we have considered the algebraic topology on $\Rep(\Sigma)$. We will now introduce another topology which seems to be stronger apriori. However, we will see that for surfaces of finite type this topology coincides with the algebraic topology.

\subsection{Chabauty Topology}

This new topology makes use of the so-called Chabauty topology on the space of closed subgroups of $G$.

\begin{defn}[Space of closed subgroups $\Sub(G)$]
	We denote by $\Sub(G)$ the set of closed subgroups of $G$.
	For open subsets $U \subseteq G$ and compact subsets $K \subseteq G$ we define the sets
	$$ \mc{O}(K) \coloneqq \{ A \in \Sub(G) \st A \cap K = \emptyset \}, \qquad \mc{O}'(U) \coloneqq \{ A \in \Sub(G) \st A \cap U \neq \emptyset \}.$$
	One can show that these form a subbasis for a topology on $\Sub(G)$. This topology is called the $\emph{Chabauty topology}$.
\end{defn}

The most important property of this topology is that $\Sub(G)$ is compact.

\begin{lemma}[\cite{canary}]
	The space of closed subgroups $\Sub(G)$ is compact and metrizable.
\end{lemma}

The following characterization is often useful.

\begin{prop}[\cite{canary}] \label{prop:charact_con_Chab}
	A sequence $(H_n)_{n \in \NN} \subseteq \Sub(G)$ converges to $H \in \Sub(G)$ if and only if the following two conditions are satisfied:
	\begin{enumerate}
		\item[(C1)]  For every $h \in H$ there is a sequence $(h_n)_{n \in \NN} \subseteq G$ such that $h_n \in H_n$ for every $n \in \NN$ and $h = \lim\limits_{n\to \infty} h_n$.
		\item[(C2)] \label{chab_conv2} If $h \in G$ is the limit of a sequence $(h_{n_k})_{k \in \NN} \subseteq H$ such that $h_{n_k} \in H_{n_k}$ for every $k \in \NN$, then $h \in H$.
	\end{enumerate}
\end{prop}

\begin{lemma}[\cite{gelanderIRS15}]
	The group $G$ acts continuously on $\Sub(G)$ via conjugation
	\begin{align*}
	G \times \Sub(G) &\longrightarrow \Sub(G)\\
	(g,H) &\longmapsto g H g^{-1}.
	\end{align*}
\end{lemma}

\begin{defn}
	We define the following subsets
	\begin{align*}
	\Subd(G) &:= \{ \Gamma \in \Sub(G) \,| \, \Gamma \text{ is discrete} \}, \text{ and}\\
	\Subdtf(G) &:= \{ \Gamma \in \Sub(G) \, | \, \Gamma \text{ is discrete and torsion-free} \},
	\end{align*}
	and equip them with the subspace topology.
\end{defn}

\begin{lemma}[\cite{canary}] \label{lem:nbhddiscrete}
	Let $\Gamma \in \Subd(G)$. Then there is an open neighborhood $U \subseteq G$ of the identity $1 \in G$ and an open neighborhood $\mc{U} \subseteq \Sub(G)$ such that 
	\[ \Gamma' \cap U = \{1 \} \]
	for every $\Gamma' \in \mc{U}$.
\end{lemma}

\begin{cor}
	$\Subd(G) \subseteq \Sub(G)$ is open.
\end{cor}

\begin{remark}
	All of the above results hold in a much more general context where $G$ is an arbitrary Lie group; see \cite{benedettipetronio}.
\end{remark}

In order to understand the Chabauty topology geometrically, the following Proposition is very useful.

\begin{prop}[\cite{benedettipetronio}] \label{prop:ChabGeomCharact}
	Let $o \in \HH^2$ and $\Gamma \in \Subdtf(G)$. Then the following holds:
	
	For every $r > 0 , \epsilon >0$ there is an open neighborhood $\mc{U} \subset \Subdtf(G)$ of $\Gamma$ such that for every $\Gamma' \in \mc{U}$ there are open neighborhoods $\Omega, \Omega'\subseteq \HH^2$ of the closed ball $\ol{B}_o(r)$ and a diffeomorphism $f \colon \Omega \longrightarrow \Omega'$ satisfying:
	\begin{enumerate}
		\item $f(o) = o$,
		\item  $\pi_{\Gamma'}( f(x)) = \pi_{\Gamma'}(f(y)) \iff \pi_\Gamma(x) = \pi_\Gamma(y)$, for every $x,y \in \Omega$, and
		\item $D_{\ol{B}_o(r)}(f,\id) < \epsilon$,
	\end{enumerate}
	where $\pi_\Gamma \colon \HH^2 \longrightarrow \Gamma \backslash \HH^2, \pi_{\Gamma'} \colon \HH^2 \longrightarrow \Gamma' \backslash \HH^2$ are the respective quotient maps, and $D_K(f,g)$ denotes the $C^\infty$-distance between two diffeomorphisms $f$, $g$ defined on a neighborhood of a compact set $K \subseteq \HH^2$.
	
	In particular, the diffeomorphism $f \colon \Omega \longrightarrow \Omega'$ descends to a diffeomorphism $F$ mapping  $\pi_\Gamma(\Omega)\subseteq \Gamma \backslash \HH^2$ to $\pi_{\Gamma'}(\Omega')\subseteq \Gamma' \backslash \HH^2$:
	\begin{center}
		\begin{tikzcd}
			\Omega \subseteq \HH^2 \arrow[r, "f"] \arrow[d, "\pi_{\Gamma}"] 
			&  \Omega' \subseteq \HH^2 \arrow[d,"\pi_{\Gamma'}"] \\
			\pi_\Gamma(\Omega) \subseteq \Gamma \backslash \HH^2 \arrow[r,"F"]
			&\pi_{\Gamma'}(\Omega') \subseteq \Gamma' \backslash \HH^2
		\end{tikzcd}
	\end{center}
\end{prop}

Informally, one can think of Proposition \ref{prop:ChabGeomCharact} as saying that large balls centered at base points $\pi_{\Gamma}(o)$ in $\Gamma \backslash \HH^2$ and $\pi_{\Gamma'}(o)$ in $\Gamma' \backslash \HH^2$ are almost identical if $\Gamma, \Gamma' \in \Subdtf(G)$ are close.

\subsection{Geometric Topology}

Using the Chabauty topology we can define the so-called geometric topology on $\Rep^*(\Sigma)$. This terminology is justified by its geometric implications; see Proposition \ref{prop:ChabGeomCharact}.

\begin{defn}[Geometric topology on $\Rep^*(\Sigma)$]
	The \emph{geometric topology} on $\Rep^*(\Sigma)$ is the initial topology with respect to the following two maps:
	\begin{enumerate}
		\item inclusion $\Rep^*(\Sigma) \hookrightarrow G^S, \rho \mapsto (\rho(s))_{s \in S}$, where $S \subset \pi_1(\Sigma)$ is a generating set for $\pi_1(\Sigma)$, and
		\item $ \im \colon \Rep^*(\Sigma) \longrightarrow \Subdtf(G), \rho \mapsto \im \rho, $
		which sends every (discrete) representation $\rho \in \Rep^*(\Sigma)$ to its image in $\Sub(G)$.
	\end{enumerate}
\end{defn}

Recall that the algebraic topology on $\Rep^*(\Sigma)$ is the initial topology with respect to just the inclusion $\Rep^*(\Sigma) \hookrightarrow G^S$, such that the geometric topology is apriori stronger than the algebraic topology. Using the fact that $\Sigma$ is of finite type we will show that it is \emph{not}. However, there are counterexamples for surfaces of infinite type; see \cite{canary} for more details.

\begin{prop} \label{prop:geometric_equals_algebraic}
	Consider $\Rep^*(\Sigma)$ with the algebraic topology. Then the map 
	\[\im \colon \Rep^*(\Sigma) \longrightarrow \Subdtf(G)\] 
	is a local homeomorphism onto its image $\mc{D}(\Sigma)$ which is the set of all discrete and torsion-free subgroups $\Gamma' < G$ such that there is an orientation preserving homeomorphism $f \colon \Sigma \longrightarrow \Gamma' \backslash\HH^2$.
	
	If we consider $\Rep(\Sigma)$ instead then the image of $\im \colon \Rep(\Sigma) \longrightarrow \Subdtf(G)$ consists of all lattices $\Gamma \in \mc{D}(\Sigma)$; we denote this set by $\mc{L}(\Sigma)$.
\end{prop}

\begin{remark}
	In particular, Proposition \ref{prop:geometric_equals_algebraic} proves that the geometric topology coincides with the algebraic topology on $\Rep^*(\Sigma)$.
\end{remark}

A proof of Proposition \ref{prop:geometric_equals_algebraic} may be found in \cite{harvey77}. Nevertheless, we decided to include a slightly different and more geometric proof using Proposition \ref{prop:ChabGeomCharact}.

\begin{proof}
	First, we shall prove that the map $\im$ is continuous.
	We need to prove that if a sequence $(\rho_n)_{n \in \NN} \subset \Rep^*(\Sigma)$ converges to $\rho_\infty \in \Rep^*(\Sigma)$ in the algebraic topology then it converges to $\rho_\infty$ in the geometric topology, too. To do so we will check (C1) and (C2) from Proposition \ref{prop:charact_con_Chab} and see that $\im \rho_n \to \im \rho_\infty$ as $n \to \infty$.
	
	\begin{enumerate}
		\item[(C1)] If $\rho_\infty(\gamma) \in \im \rho_\infty$, $\gamma \in \pi_1(\Sigma)$, then $\rho_n(\gamma) \to \rho_\infty(\gamma)$ as $n \to \infty$ by algebraic convergence.
		
		\item[(C2)] Let $(n_k) \subset \NN$ be a subsequence and $\gamma_{n_k} \in \pi_1(\Sigma)$ such that $\rho_{n_k}(\gamma_{n_k}) \in \im \rho_{n_k}$ converges to $g \in G$. Consider a standard generating set $$S = \{a_1, b_1, \ldots, a_g, b_g, c_1, \ldots, c_p\} \subset \pi_1(\Sigma)$$ 
		such that $\pi_1(\Sigma) = \langle S \, | \, \prod_{i=1}^g [a_i,b_i] \cdot c_1 \cdots c_p = 1 \rangle$. We set $\Gamma_n \coloneqq \rho_n(\pi_1(\Sigma)) = \im \rho_n$, and denote by $\pi_n \colon \HH^2 \longrightarrow \Gamma_n \backslash \HH^2=:X_n$ the respective covering maps for $n \in \NN \cup \{\infty\}$. Now, choose $o \in \HH^2$ and consider its images $x_n = \pi_n(o) \in X_n$. By construction the geodesics connecting $o$ to $\rho_n(s)o$, $s \in S$, project to closed geodesics at $x_n$ bounding a disc in $X_n$. This disc lifts to a geodesic polygon $P_n \subset \HH^2$ with vertices
		\[ \textstyle \left\{o, \rho_n(a_1)o, \rho_n(a_1 b_1)o, \rho_n(a_1 b_1 a_1^{-1}) o, \ldots, \rho_n(\prod_{i=1}^g [a_i,b_i] \cdot c_1 \cdots c_{p-1})o \right\}.\]
		Note that the interiors of the translates of the polygon $P_n$ are pairwise disjoint, i.e.
		$P_n^\circ \cap \rho_n(\gamma) P_n^\circ \neq \emptyset$ for all $\gamma \in \pi_1(\Sigma) \setminus\{1\}$. Further, the orbit maps $\pi_1(\Sigma) \longrightarrow \HH^2, \gamma \mapsto \rho_n(\gamma)o$ induce homeomorphisms from the Cayley graph $\Cay(\pi_1(\Sigma),S)$ to $\rho_n(\pi_1(\Sigma)) \cdot \partial P_n$.
		
		The vertices of $P_n$ are determined by the elements $\rho_n(s)$, $s \in S$, and by algebraic convergence $\rho_n \to \rho_\infty$ it follows that $v(P_n) \to v(P_\infty)$ and $\diam(P_n) \to \diam(P_\infty)$ as $n \to \infty$. Here $v$ denotes the hyperbolic area measure in $\HH^2$.
		
		Since $\rho_{n_k}(\gamma_{n_k})$ converges to $g$ there is an upper bound 
		\[d(o, \rho_{n_k}(\gamma_{n_k})o) < r.\] 
		Hence, $\rho_{n_k}(\gamma_{n_k})o$ is adjacent to a $\rho_{n_k}(\pi_1(\Sigma))$-translate of $P_{n_k}$ contained in the ball of radius $r + \diam(P_{n_k})$ about $o$. Since the interiors of two translates are pairwise disjoint the number of such translates is bounded from above by $ v(B_o(r+\diam(P_{n_k})))/ v(P_{n_k})$. Further, every translate of $P_{n_k}$ has $4 g + p$ edges. This yields the following rough estimate
		$$\ell_S(\gamma_{n_k}) \leq (4g + p) \cdot \frac{v(B_o(r+\diam(P_{n_k}))}{v(P_{n_k})}$$
		where $\ell_S(\gamma)$ denotes the word length of an element $\gamma \in \pi_1(\Sigma)$ with respect to the generating set $S$. Because $\diam(P_{n_k}) \to \diam(P_\infty)$ and $v(P_{n_k}) \to v(P_\infty) > 0$ there is an upper bound $\ell_{S}(\gamma_{n_k}) \leq L$ for every $k \in \NN$.
		
		The set $\{ \gamma \in \pi_1(\Sigma) \, | \, \ell_S(\pi_1(\Sigma)) \leq L \}$ is finite such that we may assume that $\gamma_{n_k} = \gamma' \in \pi_1(\Sigma)$ for every $k \in \NN$, up to a subsequence. Then by algebraic convergence
		$$ g = \lim\limits_{k \to \infty} \rho_{n_k}(\gamma_{n_k}) = \lim\limits_{k \to \infty} \rho_{n_k}(\gamma') = \rho_\infty(\gamma') \in \im \rho_\infty. $$
		
	\end{enumerate}
	Therefore, the map $\im \colon \Rep^*(\Sigma) \longrightarrow \Subdtf(G)$ is continuous.
	
	Next, let us see that $\im$ is locally injective. Let $\rho \in \Rep^*(\Sigma)$. By Lemma \ref{lem:nbhddiscrete} there is an open neighborhood $U \subseteq G$ of the identity and an open neighborhood $\mc{V} \subseteq \Subd(G)$ of $\Gamma := \im \rho$ such that $\Gamma' \cap U = \{1\}$ for every $\Gamma' \in \mc{V}$. Consider the open preimage $\mc{U}':=\im^{-1}(\mc{V})\subseteq \Rep^*(\Sigma)$ and let $V \subseteq U$ be an open neighborhood of the identity such that $V^{-1} V \subseteq U$. Let $\mc{U} \subseteq \mc{U}'$ be a smaller open neighborhood consisting of all $\rho' \in \mc{U}$ such that $\rho'(s) \in \rho(s) V$ for every $s \in S$.
	
	Let $\rho_1,\rho_2 \in \mc{U}$ with $\im \rho_1 = \im \rho_2$. Then for every $s \in S$ there is $v_i=v_i(s) \in V$ such that $\rho_i(s) = \rho(s) v_i$, $i=1,2$. Hence $\rho_2(s)^{-1} \rho_1(s) = v_2^{-1} v_1 \in V^{-1} V \subseteq U$. Since $\im \rho_1 = \im \rho_2 \in \mc{V}$ we have that $\rho_2(s)^{-1} \rho_1(s) \in \im \rho_1 \cap U =\{1\}$ for every $s \in S$. Therefore $\rho_1 = \rho_2$. This shows that $\im|_{\mc{U}}$ is injective.
	
	Finally, we want to prove that $\im|_{\im(\mc{U})}^{-1} \colon \im(\mc{U}) \longrightarrow \mc{U}$ is continuous. Let $\Gamma_n = \im \rho_n \in \im(\mc{U})$ converge to $\Gamma = \im \rho \in \im(\mc{U})$ as $n \to \infty$. 
	Note that $\Gamma \backslash \HH^2$ is a hyperbolic surface of finite type such that there is a compact connected subset $C \subseteq \Gamma \backslash \HH^2$ onto which $\Gamma \backslash \HH^2$ deformation retracts. For example, we can take its convex core $C(\Gamma)$ and cut off the cusps.
	Let $o \in \tilde{C} = \pi_\Gamma^{-1}(C)$. By Proposition \ref{prop:ChabGeomCharact} we find sequences $\epsilon_n \to 0, r_n \to \infty$ as $n\to \infty$, open neighborhoods $\Omega_n, \Omega'_n$ of $\ol{B}_o(r_n)$ and diffeomorphisms $f_n \colon \Omega_n \longrightarrow \Omega'_n$ such that 
	\begin{enumerate}
		\item $f_n(o) = o$,
		\item  $\pi_{\Gamma_n}( f_n(x)) = \pi_{\Gamma_n}(f_{n}(y)) \iff \pi_\Gamma(x) = \pi_\Gamma(y)$, for every $x,y \in \Omega_n$, and
		\item $D_{\ol{B}_o(r_n)}(f_{n},\id) < \epsilon_n$.
	\end{enumerate}
	In the following we will abbreviate $\pi_n = \pi_{\Gamma_n}, \pi = \pi_\Gamma$. Let us denote by 
	$F_n \colon \pi(\Omega_n) \longrightarrow \pi_n(\Omega'_n), \pi(x) \mapsto f_n(\pi(x))$ the induced diffeomorphisms. Since the diameter of $C$ is finite $\pi(\ol{B}_o(r_n))$ contains $C$ for large $n$. Recall that $\Gamma \backslash \HH^2$ is homotopy equivalent to $C$ such that we may identify the fundamental group of $C$ with $\Gamma$. In this way $F_n \colon C \longrightarrow \Gamma_n \backslash \HH^2$ induce homomorphisms $\sigma_n \colon \Gamma \longrightarrow \Gamma_n$ at the level of fundamental groups. Let us consider the lifts $\tilde{F}_n \colon \tilde{C} \longrightarrow \HH^2$ of $F_n \colon C \longrightarrow \Gamma_n \backslash \HH^2$ such that $\tilde{F}_n(o) = o$. These are equivariant meaning that
	\[ \tilde{F}_n(\gamma x) = \sigma_n(\gamma) \tilde{F}_n(x)\]
	for every $\gamma \in \Gamma$, $x \in \tilde{C}$. Moreover, $\tilde{F}_n$ coincides with $f_n$ on $\ol{B}_o(r_n) \cap \tilde{C}$ by uniqueness of lifts.
	
	We want to show that $\sigma_n(\rho(s)) \to \rho(s)$ as $n \to \infty$ for every $s \in S$. This is equivalent to convergence of the differentials $D_o \sigma_n(\rho(s)) \to D_o\rho(s)$ as $n\to \infty$. There is $R>0$ such that $\rho(s)o \in B_o(R)$ for every $s \in S$, and $\ol{B}_o(R) \subseteq \ol{B}_o(r_n)$ for large $n$. By equivariance 
	\begin{align*}
	D_o \sigma_n(\rho(s)) = D_o (f_n \circ \rho(s) \circ f^{-1}_n) = D_{\rho(s)o} f_n \circ D_o\rho(s) \circ D_o f_n^{-1}.
	\end{align*}
	Since $f_n$ converges to $\id$ on $\ol{B}_o(R)$ in the $C^\infty$-distance it follows that 
	\[ D_o \sigma_n(\rho(s)) \to D_o \rho(s) \]
	as $n \to \infty$. This implies that $\sigma_n(\rho(s)) \to \rho(s)$ as $n \to \infty$.
	
	Next, let us see that $\sigma_n \colon \Gamma \longrightarrow \Gamma_n$ is injective for large $n$. To this end it suffices to prove that $F_n \colon C \longrightarrow \Gamma_n \backslash \HH^2$ embeds $C$ as a subsurface with homotopically non-trivial peripheral curves. Indeed, such subsurface embeddings are $\pi_1$-injective (cf.\ restriction maps to subsurfaces in subsection \ref{subsect:AugTeich}). Let $\gamma_1 = \rho(c_1), \ldots, \gamma_p = \rho(c_p) \in \Gamma \cong \pi_1(\Gamma \backslash \HH^2)$ correspond to the peripheral elements $\{c_1, \ldots, c_p\} \subset \pi_1(\Sigma)$. Then $\sigma_n(\gamma_i) \to \rho(c_i) \neq 1$ as $n \to \infty$ such that $(F_n)_*(\gamma_i) = \sigma_n(\gamma_i)$ are homotopically non-trivial for large $n$, $i=1, \ldots, p$.
	
	More is true by the above. The diffeomorphism $F_n$ embeds $C \cong \Sigma$ as a subsurface with homotopically non-trivial peripheral curves in $\Gamma_n \backslash \HH^2 \cong \Sigma$. This is only possible if $F_n$ sends peripheral curves to peripheral curves. Therefore, $\Gamma_n \backslash \HH^2$ deformation retracts to $F_n(\interior(C))$. Because $\Gamma \backslash \HH^2$ deformation retracts to $\interior(C)$ we may extend $F_n$ to a diffeomorphism $\Gamma \backslash \HH^2 \longrightarrow \Gamma_n \backslash \HH^2$ inducing $\sigma_n \colon \Gamma \longrightarrow \Gamma_n$.
	In particular, $\sigma_n$ is an isomorphism induced by an orientation preserving diffeomorphism. 
	
	It follows that $\sigma_n \circ \rho \in \Rep^*(\Sigma)$.
	Moreover, $\sigma_n \circ \rho \to \rho$ as $n \to \infty$ such that $\sigma_n \circ \rho \in \mc{U}$ for large $n$. But also $\rho_n \in \mc{U}$ and $\im \rho_n = \Gamma_n = \im (\sigma_n \circ \rho)$. By injectivity of $\im \colon \mc{U} \longrightarrow \im(\mc{U})$ we have that $\rho_n = \sigma_n \circ \rho$ for large $n$, and $\rho_n = \sigma_n \circ \rho \to \rho$ as $n \to \infty$. Thus $\im|_{\mc{U}}^{-1}$ is indeed continuous.
	
	That $\im(\Rep^*(\Sigma)) = \mc{D}(\Sigma)$ and $\im(\Rep(\Sigma)) = \mc{L}(\Sigma)$ follows at once from the definition of $\Rep^*(\Sigma)$ and $\Rep(\Sigma)$ respectively.
	
\end{proof}

Proposition \ref{prop:geometric_equals_algebraic} allows us to identify moduli space $\Moduli(\Sigma)$ with the space $G \backslash \Lattices(\Sigma)$ of $G$-conjugacy classes of lattices.

\begin{prop}\label{prop:ModuliAsLattices}
	The space $\Lattices(\Sigma)$ is invariant under the conjugation action of $G$ and we may identify its quotient with moduli space via the following homeomorphism
	\begin{align*}
	\psi \colon \Moduli(\Sigma) &\longrightarrow G \backslash \Lattices(\Sigma), \\
	[[ \rho ]] &\longmapsto [ \im \rho].
	\end{align*}
\end{prop}

\begin{proof}
	Let $\Gamma \in \Lattices(\Sigma)$ and $g \in G$. Then $\Gamma' \coloneqq g \Gamma g^{-1}$ is a torsion-free lattice too. Moreover, the element $g \in G \cong \Isom_+(\HH^2)$ induces an orientation preserving isometry 
	\begin{align*}
	g \colon \Gamma \backslash \HH^2 &\longrightarrow \Gamma' \backslash \HH^2,\\
	\Gamma x &\longmapsto \Gamma' g x,
	\end{align*}
	whence $\Gamma' = g \Gamma g^{-1} \in \Lattices(\Sigma)$.
	
	Let us consider the right-action $\Aut^*(\pi_1(\Sigma)) \acts \Rep(\Sigma)$. We claim that the map 
	\[\im \colon \Rep(\Sigma) \longrightarrow \Lattices(\Sigma)\] 
	induces a homeomorphism:
	\begin{center}
		\begin{tikzcd}
			\Rep(\Sigma) \arrow[d] \arrow[dr,"\im",two heads] & \\
			\Rep(\Sigma) / \Aut^*(\pi_1(\Sigma)) \arrow[r,"\varphi"',dashed] & \Lattices(\Sigma)
		\end{tikzcd}
	\end{center}
	Clearly, if $\rho_1 = \rho_2 \circ \alpha$ for $\rho_1, \rho_2 \in \Rep(\Sigma)$ and $\alpha \in \Aut^*(\pi_1(\Sigma))$ then $\im \rho_1 = \im \rho_2$. On the other hand, if $\Gamma = \im \rho_1 = \im \rho_2$ for $\rho_1, \rho_2 \in \Rep(\Sigma)$ and $f_1, f_2 \colon \Sigma \longrightarrow \Gamma \backslash \HH^2$ are orientation preserving homeomorphisms inducing $\rho_1, \rho_2$ respectively, then  $\rho_2^{-1} \circ \rho_1$ is induced by $f_2^{-1} \circ f_1$ such that $ \alpha \coloneqq \rho_2^{-1} \circ \rho_1 \in \Aut^*(\pi_1(\Sigma))$, and $\rho_1 = \rho_2 \circ \alpha$. Hence, $\varphi$ is a bijection.
	
	By definition of the quotient topology, $\varphi$ is continuous. Finally, $\im$ is a local homeomorphism such that $\varphi^{-1}$ is continuous too. This shows that $\varphi$ is a homeomorphism.
	
	Observe that $\varphi$ is equivariant with respect to the conjugation action of $G$ both on $\Rep(\Sigma)/\Aut^*(\pi_1(\Sigma))$ and on $\Lattices(\Sigma)$. Therefore, taking the quotient by the conjugation actions yields a homeomorphism
	\begin{center}
		\begin{tikzcd}
			\Rep(\Sigma) / \Aut^*(\pi_1(\Sigma)) \arrow[r,"\varphi","\cong"'] \arrow[d]
			& \Lattices(\Sigma) \arrow[d] \\
			\Teich(\Sigma)/\Out^*(\pi_1(\Sigma)) \arrow[r,"\psi","\cong"',dashed] & G \backslash \Lattices(\Sigma)
		\end{tikzcd}
	\end{center}
	given by $\psi([[\rho]]) = [\im \rho]$ for every $[[\rho]] \in \Moduli(\Sigma) = \Teich(\Sigma)/\Out^*(\Sigma)$.
\end{proof}

We conclude this section with the following Lemma, that restricts the kind of closed subgroups that arise in the closure of the $G$-orbit of a lattice $\Gamma \in \Lattices(\Sigma)$ in $\Sub(G)$.

\begin{lemma}\label{lem:LimitsOfConjugates}
	Let $\Gamma \in \Lattices(\Sigma)$ and let $(g_n)_{n \in \NN} \subset G$ be a sequence of elements such that
	\[ g_n^{-1} \Gamma g_n \to H \qquad (n\to \infty) \]
	converges to $H \in \Sub(G)$. Then, either
	\begin{enumerate}
		\item $H$ is abelian, or
		\item $H$ is conjugate to $\Gamma$.
	\end{enumerate}
\end{lemma}

\begin{proof}
	Let us assume that $H$ is not abelian. Further, let $\pi \colon \HH^2 \longrightarrow \Gamma \backslash \HH^2$ denote the quotient map and $o \in \HH^2$. There are two cases to consider:
	\begin{enumerate}
		\item[a)] The sequence $\pi(g_n o)$ is contained in a compact set $K \subseteq \Gamma \backslash \HH^2$, or
		\item[b)] The sequence $\pi(g_n o)$ goes to infinity in $\Gamma \backslash \HH^2$.
	\end{enumerate}
	In case a) we may find a compact set $K' \subseteq \HH^2$ such that $\pi(K') = K$ and therefore elements $\gamma_n \in \Gamma$ such that $\gamma_n g_no \in K'$. Then there is a convergent subsequence $\gamma_{n_k} g_{n_k} \to g$ as $k \to \infty$ such that
	\[ H = \lim\limits_{n \to \infty} g_n^{-1} \Gamma g_n = \lim\limits_{k \to \infty} g_{n_k}^{-1} \gamma_{n_k}^{-1} \Gamma \gamma_{n_k} g_{n_k} = g^{-1} \Gamma g. \]
	
	In case b) there is a subsequence such that $\pi(g_{n_k}o) \in C$ is contained in a cusp region $C$ of the thin-part of $\Gamma \backslash \HH^2$. Let $H \subseteq \pi^{-1}(C)$ be a horoball in the preimage of $C$ centered at $\xi \in \partial \HH^2$, and let $\hat{\gamma}_{n_k} \in \Gamma$ be  such that $\hat{\gamma}_{n_k} g_{n_k} o \in H$ for every $k \in \NN$, and $\xi = \lim\limits_{k \to \infty} \hat{\gamma}_{n_k} g_{n_k} o$. If we set $\hat{g}_k := \hat{\gamma}_{n_k} g_{n_k}$, then $\xi = \lim\limits_{k \to \infty} \hat{g}_k o $, and
	\[ H = \lim_{k \to \infty} g_{n_k}^{-1} \Gamma g_{n_k} = \lim_{k \to \infty} (\hat{g}_{k})^{-1} (\hat{\gamma}_{n_k})^{-1} \Gamma \hat{\gamma}_{n_k} \hat{g}_{k} = \lim_{k \to \infty} (\hat{g}_{k})^{-1} \Gamma \hat{g}_{k} . \]
	
	Let $h \in H$ and let $\gamma_k \in \Gamma$ such that $(\hat{g}_k)^{-1} \gamma_k \hat{g}_k \to h$ as $k \to \infty$. Note that 
	\[ d(h o, o) = \lim_{k \to \infty} d((\hat{g}_{k})^{-1} \gamma_{k} \hat{g}_{k}o,o) 
	=\lim_{k \to \infty} d(\gamma_{k} \hat{g}_{k}o,\hat{g}_{k} o). \]
	Hence there is $D>0$, such that
	\[ d(\gamma_{k} \hat{g}_{k}o,\hat{g}_{k} o) < D\]
	for all $k \in \NN$.
	Thus $\gamma_{k} H \cap H \neq \emptyset$ and $\gamma_{k} \in P=\stab_\xi(\Gamma)$ for large $k$. Because $P$ is abelian, this implies that $H$ is abelian.
\end{proof}


\section{Invariant Random Subgroups} \label{sect:irs}

\subsection{Embedding Moduli Space} \label{subsect:embeddingmodulispace}

In subsection \ref{subsect:aug_moduli} we constructed augmented moduli space, a compactification of moduli space $\Moduli(\Sigma)$. We will now describe another compactification using invariant random subgroups (IRS). This construction is due to Gelander \cite{gelanderIRS15}.

\begin{defn}
	Let $G$ be a locally compact Hausdorff group.
	An \emph{invariant random subgroup of $G$} is a Borel probability measure on $\Sub(G)$ that is $G$-invariant. We denote the space of all invariant random subgroups of $G$ by 
	\[\IRS(G) = \Prob(\Sub(G))^G.\]
\end{defn}

\begin{lemma}
	The space of invariant random subgroups $\IRS(G)$ of a locally compact Hausdorff group $G$ is compact.
\end{lemma}
\begin{proof}
	Since $\Sub(G)$ is compact, so is $\Prob(\Sub(G))$ by Banach-Alaoglu. The space of invariant random subgroups is a closed subspace
	\[ \Prob(\Sub(G))^G = \bigcap_{g \in G} \{ \mu \in \Prob(\Sub(G)) \, | \, g_* \mu = \mu \}. \]
\end{proof}

In order to obtain a compactification of $\Moduli(\Sigma)$ we will embed it in $\IRS(G)$ for $G= \PSL(2,\RR)$. We will use Proposition \ref{prop:ModuliAsLattices} identifying $\Moduli(\Sigma) \cong G \backslash \Lattices(\Sigma)$ to do so.

Let $\Gamma \in \Lattices(\Sigma)$ and let $\nu_\Gamma$ denote the (unique) right-invariant Borel probability measure on $\Gamma \backslash G$. Further, consider the map
\begin{align*}
\varphi_\Gamma \colon \Gamma \backslash G &\longrightarrow \Sub(G),\\
\Gamma g &\longmapsto g^{-1} \Gamma g.
\end{align*}
Using $\varphi_\Gamma$ we push $\nu_\Gamma$ forward to obtain a probability measure $\mu_\Gamma = (\phi_\Gamma)_* \nu_\Gamma$ on $\Sub(G)$.

\begin{lemma} \label{lem:IRSconjuginv}
	The measure $\mu_\Gamma$ depends only on the conjugacy class of $[\Gamma] \in G \backslash \Lattices(\Sigma)$ and is an invariant random subgroup in $\IRS(G)$.
\end{lemma}
\begin{proof}
	Let $\Gamma' = h \Gamma h^{-1}$ be a conjugate, $h \in G$. There is a $G$-equivariant homeomorphism
	\[ \hat{h} \colon \Gamma \backslash G \longrightarrow \Gamma' \backslash G, \quad\Gamma g \longmapsto  \Gamma' h g. \]
	One can check that the following diagram commutes:
	
	\begin{center}
		\begin{tikzcd}
			\Gamma \backslash G \arrow[drr,"\phi_\Gamma"] \arrow[dd,"\hat{h}"]& & \\
			& & \Sub(G) \\
			\Gamma' \backslash G \arrow[urr,"\phi_{\Gamma'}"] & & 
		\end{tikzcd}
	\end{center}
	Moreover, $\hat{h}$ is $G$-equivariant such that $\hat{h}_* \nu_\Gamma$ is another $G$-invariant probability measure and by uniqueness $\hat{h}_* \nu_\Gamma = \nu_{\Gamma'}$. Thus $(\phi_\Gamma)_* \nu_\Gamma = \mu_\Gamma = (\phi_{\Gamma'})_* \nu_{\Gamma'}$ so that $\mu_\Gamma$ depends only on the conjugacy class of $\Gamma$.
	
	In order to show that $\mu_\Gamma$ is an invariant random subgroup, it suffices to observe that the map $\varphi_\Gamma \colon \Gamma \backslash G \to \Sub(G)$ is $G$-equivariant with respect to the right-translation action on $\Gamma\backslash G$ and the conjugation action on $\Sub(G)$. Invariance of $\mu_\Gamma$ then follows from invariance of $\nu_\Gamma$.
\end{proof}

We have thus proven that there is a well-defined map
\begin{align*}
\iota \colon G \backslash \Lattices(\Sigma) &\longrightarrow \IRS(G), \\
[\Gamma] &\longmapsto \mu_{\Gamma}.
\end{align*}

Let us start our investigation of this construction by observing the following two basic Lemmas.

\begin{lemma}
	Let $\Gamma \in \Lattices(\Sigma)$.
	Then the support of $\mu_\Gamma$ is the closure of the $G$-orbit of $\Gamma$ in $\Sub(G)$, $\supp(\mu_{\Gamma}) = \ol{G * \Gamma}$.
\end{lemma}

\begin{proof}
	Note that $\ol{G * \Gamma}=\ol{\phi_\Gamma(\Gamma \backslash G)}$ by definition.
	Further, recall that $H \in \Sub(G)$ is in the support $\supp(\mu_{\Gamma})$ if and only if every open neighborhood $\mc{U} \subseteq \Sub(G)$ of $H$ has positive mass $\mu_\Gamma(\mc{U}) > 0$. 
	
	Let $H \notin \supp(\mu_\Gamma)$. We want to show that $H \notin \ol{\phi_{\Gamma}(\Gamma \backslash G)}$, i.e.\ there is an open neighborhood $\mc{U}\subseteq \Sub(G)$ of $H$ such that $\mc{U} \cap \phi_\Gamma(\Gamma \backslash G) = \emptyset$. Because $H \notin \supp(\mu_{\Gamma})$ there is an open neighborhood $\mc{U} \subseteq \Sub(G)$ such that $\mu_\Gamma(\mc{U}) = 0$. Then $V = \phi_\Gamma^{-1}(\mc{U}) \subseteq \Gamma \backslash G$ is an open subset such that 
	\[\nu_\Gamma(V) = \nu_\Gamma ( \phi_\Gamma^{-1}(\mc{U})) = \mu_\Gamma(\mc{U}) = 0.\] 
	Because $\nu_\Gamma$ has full support on $\Gamma \backslash G$ the set $V$ must be empty. Therefore, $\mc{U} \cap \phi_\Gamma(\Gamma \backslash G) = \emptyset$.
	
	Vice versa, let $H \in \supp(\mu_\Gamma)$, and let $\mc{U} \subseteq \Sub(G)$ be an open neighborhood of $H$. Then
	\[ 0<\mu_\Gamma(\mc{U}) = \nu_\Gamma(\phi_\Gamma^{-1}(\mc{U})). \]
	Hence, $V=\phi_\Gamma^{-1}(\mc{U}) \neq \emptyset \subseteq \Gamma \backslash G $ is a non-empty open subset, and
	$\mc{U} \cap \phi_\Gamma(\Gamma \backslash G) \neq \emptyset$. Because $\mc{U}$ was an arbitrary open neighborhood of $H$, it follows that $H \in \ol{\phi_\Gamma(\Gamma \backslash G)} = \ol{G*\Gamma}$.
\end{proof}

\begin{lemma}\label{lem:LinearIndependenceIRS}
	Let $[\Gamma_1], \ldots, [\Gamma_m]$ be pairwise distinct conjugacy classes of lattices in $G$.
	Then the associated invariant random subgroups
	$\mu_{\Gamma_1}, \ldots, \mu_{\Gamma_m} \in \IRS(G) \subseteq C(\Sub(G))^*$ are linearly independent.
\end{lemma}
\begin{proof}
	Let $\lambda_1, \ldots, \lambda_m \in \RR$ such that
	\[ 0 = \sum_{i=1}^m \lambda_i \cdot \mu_{\Gamma_i} .\]
	For all $i,j \in \{1,\ldots, m\}, i \neq j,$ there is an open neighborhood $U_{i,j} \subseteq \Sub(G)$ of $\Gamma_i$ such that $U_{i,j} \cap G * \Gamma_j = \emptyset$. Indeed, otherwise there would be $i \neq j$ and a sequence $(g_n)_{n \in \NN} \subset G$ such that $g_n \Gamma_j g_n^{-1} \to \Gamma_i$ as $n \to \infty$. Because $\Gamma_i$ is not abelian this implies that $\Gamma_i$ is conjugate to $\Gamma_j$ by Lemma \ref{lem:LimitsOfConjugates}; contradiction.
	
	We set 
	\[U_i = \bigcap_{\substack{j=1 \\j\neq i}}^m U_{i,j}\] 
	such that $U_i$ is an open neighborhood of $\Gamma_i$ satisfying
	\[ \emptyset = U_i \cap \ol{G*\Gamma_j} = U_i \cap \supp(\mu_{\Gamma_j})\]
	for every $j \neq i$. Then
	\[ 0 = \left( \sum_{i=1}^m \lambda_i \cdot \mu_{\Gamma_i} \right) (U_i) = \lambda_i \cdot \underbrace{\mu_{\Gamma_i}(U_i)}_{>0}  \]
	such that $\lambda_i = 0$ for every $i=1, \ldots,m$.
\end{proof}

From Lemma \ref{lem:LinearIndependenceIRS} it is immediate that $\iota \colon \Moduli(\Sigma) \longrightarrow \IRS(G)$ is injective. Our next goal is to prove the following.

\begin{propdefn}\label{prop:topembedding}
	The map $\iota \colon \Moduli(\Sigma) \longrightarrow \IRS(G)$ is a topological embedding.
	
	We call the closure of its image $\IRSModuli(\Sigma) \coloneqq \ol{\iota(\Moduli(\Sigma))}$ the \emph{IRS compactification of moduli space}.
\end{propdefn}

\begin{remark}
	In the case where $\Sigma$ is compact it follows from \cite[Proposition 11.2]{locrig16} that the map $\iota$ is continuous.
\end{remark}

Before we attempt a proof let us understand the measure $\mu_\Gamma \in \IRS(G)$, $\Gamma \in \Lattices(\Sigma)$, in terms of a Haar measure $\nu$ on $G$. Let $F \in C(\Sub(G))$ be a continuous function. Then
\begin{align*}
\int_{\Sub(G)} F(H) \, d \mu_\Gamma(H) &= \int_{\Gamma \backslash G} F(g^{-1} \Gamma g) \, d\nu_\Gamma(g \Gamma) \\
&= \nu(D)^{-1} \int_G \II_D(g) \cdot F(g^{-1} \Gamma g) \, d\nu(g),
\end{align*}
where $D \subset G$ is a fundamental domain for the action $\Gamma \acts G$. 

We shall pick a particularly convenient Haar measure $\nu$ on $G$ now. Recall that $G = \PSL(2,\RR) = \Isom_+(\HH^2)$, and let $o \in \HH^2$. Then the map $p \colon G \to \HH^2, g \mapsto go$ is surjective and induces an identification $G/K \cong \HH^2$ where $K = \stab_G(o)$ is a conjugate of $\SO(2,\RR) < \PSL(2,\RR)$. Thus the hyperbolic area measure $v$ on $\HH^2$ amounts via this identification to a $G$-invariant measure on $G/K$ which we shall denote by $v$ as well. Further, we choose a normalized Haar measure $\eta$ on $K$ such that $\eta(K) = 1$. By Weil's quotient formula we obtain a Haar measure $\nu$ on $G$ such that
\[ \int_G f(g)\, d\nu(g) = \int_{G/K} \int_K f(gk) \, d\eta(k) \, dv(gK) \]
for every $f \in L^1(G,\nu)$. With this choice we have that 
$\nu(p^{-1}(B)) = v(B)$ for every measurable subset $B \subseteq \HH^2$.

We will need the following Lemma.
\begin{lemma} \label{lem:convL1quotmeas}
	Let $(f_n)_{n \in \NN} \subset L^1(G/K,v)$ be a sequence converging to $f \in L^1(G/K,v)$ with respect to the $L^1$-norm. Then 
	the sequence $(f_n \circ p)_{n \in \NN}$ converges to $f \circ p$ in $L^1(G,\nu)$ as $n \to \infty$.
\end{lemma}
\begin{proof}
	We compute
	\begin{align*}
	\norm{ f_n \circ p - f \circ p}_{L^1(G)} &= \int_G \abs{f_n (p(g)) - f (p(g)) } \, d\nu(g) \\
	&= \int_{G/K} \int_K \abs{f_n(p(gk)) - f(p(gk))} \, d\eta(k) \, dv(gK) \\
	&= \int_{G/K} \int_K \abs{f_n(gK) - f(gK)} \, d\eta(k) \, dv(gK) \\
	&= \int_{G/K} \eta(K) \cdot \abs{f_n(gK) - f(gK)} \, dv(gK)\\
	&= \int_{G/K} \abs{f_n(gK) - f(gK)} \, dv(gK) \to 0 \quad (n \to \infty).
	\end{align*}
\end{proof}

Given a point $o \in \HH^2$ there is a particularly nice choice of fundamental domain for a discrete and torsion-free subgroup $\Gamma < G$.

\begin{defn}
	Let $\Gamma < G$ be a discrete torsion-free subgroup and $o \in \HH^2$. The \emph{Dirichlet fundamental domain} of $\Gamma$ with respect to $o \in \HH^2$ is defined as
	\[ D_o(\Gamma) = \{ x \in \HH^2 \, | \, d(x,o) \leq d(x,\gamma o) \quad \forall \gamma \in \Gamma \setminus \{1\} \}. \]
\end{defn}

The following Lemma is immediate from the definitions.

\begin{lemma}
	Let $\Gamma < G$ be a discrete torsion-free subgroup and $o \in \HH^2$.
	Then the Dirichlet domain $D_o(\Gamma)$ is a fundamental domain for the $\Gamma$-action on $\HH^2$. Further, its preimage $F_o(\Gamma) = p^{-1}(D_o(\Gamma))$ is a fundamental domain for the $\Gamma$-action on $G$.
\end{lemma}

If the group $\Gamma$ above is not a lattice then its fundamental domain does not have finite area. However, we can truncate the Dirichlet domain to obtain a domain with finite area:

Recall that the \emph{limit set} of $\Gamma$ is defined as $L(\Gamma) = \overline{ \Gamma o} \cap \partial \HH^2$. Its complement $\Omega(\Gamma) = \partial \HH^2 \setminus L(\Gamma)$ is called the \emph{domain of discontinuity}. If $\Gamma$ is a lattice then $L(\Gamma) = \partial\HH^2$. If $\Gamma \backslash\HH^2$ has free ends then the domain of discontinuity is non-empty. Every connected component of $\Omega(\Gamma)$ is an interval in $\partial \HH^2 \cong \SS^1$. The endpoints of these intervals are exactly the fixed points of hyperbolic elements in $\Gamma$ corresponding to free ends of $\Gamma \backslash\HH^2$. We define
\[ \tilde{C}(\Gamma) = \conv(L(\Gamma)) \cap \HH^2\]
and call
\[ C(\Gamma) = \Gamma \backslash \tilde{C}(\Gamma)\]
its convex core.
By what we said before $C(\Gamma) = \Gamma \backslash \HH^2$ if $\Gamma$ is a lattice. Otherwise, $\HH^2 \setminus \tilde{C}(\Gamma)$ are half-spaces bounded by the axes of hyperbolic elements representing waist curves of free ends. 

If $\Gamma \backslash \HH^2$ has free ends it has infinite volume. However, its convex core $C(\Gamma) = \Gamma \backslash \tilde{C}(\Gamma)$ has finite volume. Indeed, considering only the convex core amounts to cutting off the free ends of $\Gamma \backslash\HH^2$ at their waist geodesic. Thus $C(\Gamma)$ may be decomposed in $\abs{\chi(\Gamma \backslash \HH^2)}$ (possibly degenerate) pairs of pants all of which have area $2 \pi$ such that $\vol(C(\Gamma)) = 2\pi \abs{\chi(\Gamma \backslash \HH^2)}$. 

We define the \emph{truncated Dirichlet domain} to be the intersection
\[ \hat{D}_o(\Gamma) = D_o(\Gamma) \cap \tilde{C}(\Gamma). \]
Since the convex core has finite volume, we have that $\II_{\hat{D}_o(\Gamma)} \in L^1(\HH^2,v)$. Analogously, we define
\[ \hat{F}_o(\Gamma) := p^{-1}(\hat{D}_o(\Gamma))\subseteq G. \]
Note that
\[ \nu(\hat{F}_o(\Gamma)) = v(\hat{D}_o(\Gamma)) =  2\pi \abs{\chi(\Gamma \backslash \HH^2)}.\]

This construction is well-behaved with respect to variations of the group $\Gamma$ as the following Lemma asserts.

\begin{lemma}\label{lem:L1ConvTruncDom}
	Let $(\rho_n)_{n \in \NN} \subset \Rep^*(\Sigma)$ be a sequence converging to $\rho \in \Rep^*(\Sigma)$ and denote $\Gamma = \im \rho$, $\Gamma_n = \im \rho_n$, $n \in \NN$. Then
	\[ \II_{\hat{D}_o(\Gamma_n)} \to \II_{\hat{D}_o(\Gamma)} \quad (n \to \infty) \]
	in $L^1(\HH^2,v)$. 
	
	In particular,
	\[ \II_{\hat{F}_o(\Gamma_n)} \to \II_{\hat{F}_o(\Gamma)} \quad (n \to \infty) \]
	in $L^1(G,\nu)$ by Lemma \ref{lem:convL1quotmeas}.
\end{lemma}
Although Lemma \ref{lem:L1ConvTruncDom} seems to be classical, we could not find any proof in the literature. Hence, we decided to include a complete proof in section \ref{section:ProofOfLemma}.

Also, the following consequence of the dominated convergence theorem will be useful later on.
\begin{lemma}\label{lem:ProdConvL1}
	Let $(X,\mu)$ be a measure space, let $(f_n)_{n \in \NN} \subset L^1(X,\mu)$ and let $(g_n)_{n \in \NN} \subset L^\infty(X,\mu)$. Assume that there is $f \in L^1(X,\mu)$ such that
	\[ \| f_n - f\|_{L^1} \to 0 \quad (n \to \infty), \]
	and that there is $C > 0$ and $g \in L^\infty(X,\mu)$ such that $\|g_n\|_{L^\infty} \leq C$, for every $n \in \NN$, and 
	\[ g_n(x) \to g(x) \quad (n\to \infty)\]
	for $\mu$-almost-every $x \in X$.
	
	Then
	\[ \| f_n \cdot g_n - f \cdot g\|_{L^1} \to 0 \quad (n\to \infty).\]
\end{lemma}
\begin{proof}
	We compute
	\begin{align*}
	\| f_n \cdot g_n - f \cdot g\|_{L^1} 
	&\leq \|f_n \cdot g_n - f \cdot g_n \|_{L^1} 
	+ \| f \cdot g_n - f \cdot g\|_{L^1}\\
	&\leq C \cdot \| f_n - f \|_{L^1} + \int_X \abs{f(x)} \cdot \abs{g_n(x) - g(x)} \, d\mu(x).
	\end{align*}
	Note that 
	\[\abs{f(x)} \cdot \abs{g_n(x) - g(x) } \to 0 \quad(n \to \infty) \] 
	for $\mu$-almost-every $x \in X$ and the functions $\abs{f(x)} \cdot \abs{g_n(x) - g(x) }$ are $\mu$-almost-everywhere dominated by the integrable function $2 C \abs{f(x)}$. By the dominated convergence theorem we conclude that 
	\[\int_X \abs{f(x)} \cdot \abs{g_n(x) - g(x)} \, d\mu(x) \to 0 \quad (n\to \infty),\]
	which in turn implies
	\[ \| f_n \cdot g_n - f \cdot g\|_{L^1} \to 0 \quad (n\to \infty).\]
\end{proof}

After these preparations we are now ready to prove Proposition \ref{prop:topembedding}.
\begin{proof}[Proof of Proposition \ref{prop:topembedding}]
	We want to show that $\iota \colon G \backslash \Lattices(\Sigma) \to \IRS(G)$ is a topological embedding. 
	
	First, let us prove that $\iota$ is continuous.
	Let $([\Gamma_n])_{n\in \NN} \subset G \backslash \Lattices(\Sigma)$ be a convergent sequence with limit $[\Gamma] \in G \backslash \Lattices(\Sigma)$. Up to taking conjugates we may assume that $\Gamma_n \to \Gamma$ in $\Lattices(\Sigma)$. Let $o \in \HH^2$ and we consider the fundamental domains $F_o(\Gamma_n)= p^{-1}(D_o(\Gamma_n))$ for $\Gamma_n \acts G$. Since $\Gamma_n$ is a lattice we have that $\tilde{C}(\Gamma_n) = \HH^2$ and $\hat{D}_o(\Gamma_n) = D_o(\Gamma_n)$.
	
	Let $f \in C(\Sub(G))$. Then
	\begin{align*}
	&\quad \abs{ \int_{\Sub(G)} f(H) \, d\mu_{\Gamma_n}(H) - \int_{\Sub(G)} f(H) \, d\mu_{\Gamma}(H) } \\
	& = \abs{ \nu(F_o(\Gamma_n))^{-1} \cdot \int_G \II_{F_o(\Gamma_n)}(g) \cdot f(g^{-1} \Gamma_n g) \, d \nu(g) -  \nu(F_o(\Gamma))^{-1} \cdot \int_G \II_{F_o(\Gamma)}(g) \cdot f(g^{-1} \Gamma g) \, d \nu(g)}\\
	& \leq \frac{1}{2\pi \abs{\chi(\Sigma)}} \cdot \left \| \II_{F_o(\Gamma_n)} \cdot \bar{f}_n - \II_{F_o(\Gamma)} \cdot \bar{f} \right\|_{L^1},
	\end{align*}
	where we set $\bar{f}_n(g) \coloneqq f(g^{-1} \Gamma_n g), \bar{f}(g) \coloneqq f(g^{-1} \Gamma g)$ for every $g \in G$. Note that $f$ is uniformly bounded because $\Sub(G)$ is compact, so that $(\bar{f}_n)_{n \in \NN}$ are uniformly bounded, too. Moreover, 
	\[\bar{f}_n(g) = f(g^{-1} \Gamma_n g) \to \bar{f}(g) = f(g^{-1} \Gamma g)  \qquad (n \to \infty)\]
	for every $g \in G$, by continuity. By Lemma \ref{lem:L1ConvTruncDom} we know that 
	\[ \|\II_{F_o(\Gamma_n)} - \II_{F_o(\Gamma)}\|_{L^1(G,\nu)} \to 0 \qquad (n \to \infty).\]
	Thus we may apply Lemma \ref{lem:ProdConvL1} and conclude that
	\[ \int_{\Sub(G)} f(H) \, d\mu_{\Gamma_n}(H) \to \int_{\Sub(G)} f(H) \, d\mu_{\Gamma}(H) \qquad (n \to \infty). \]
	This shows that $\iota$ is continuous.

	Finally, let $(\Gamma_n)_{n \in \NN} \subset \Lattices(\Sigma)$ and let $\Gamma \in \Lattices(\Sigma)$, such that $\mu_{\Gamma_n} \to \mu_{\Gamma}$ as $n \to \infty$. We want to show that $[\Gamma_n] \to [\Gamma]$ as $n \to \infty$. Let $\mc{U} \subset \Sub(G)$ be an open neighborhood of $\Gamma$, and let $\ol{\mc{V}} \subseteq \mc{U}$ be a compact neighborhood of $\Gamma$. By Urysohn's Lemma we find a continuous function $f \colon \Sub(G) \to [0, \infty)$ such that $f|_{\ol{\mc{V}}} \equiv 1$ and $f|_{\mc{U}^c} \equiv 0$. Because $\mu_{\Gamma_n} \to \mu_\Gamma$ we have that 
	\[ \int_{\Sub(G)} f(H) \, d\mu_{\Gamma_n}(H) \to \int_{\Sub(G)} f(H) \, d\mu_\Gamma(H) \qquad (n \to \infty). \]
	Because
	\[ \int_{\Sub(G)} f(H) \, d\mu_\Gamma(H) \geq \int_{\ol{\mc{V}}} f(H) \, d\mu_{\Gamma}(H) = \mu_\Gamma(\ol{\mc{V}}) > 0,\]
	also
	\[ \nu_{\Gamma_n}(\phi_{\Gamma_n}^{-1}(\mc{U})) = \mu_{\Gamma_n}(\mc{U}) \geq \int_{\Sub(G)} f(H) \, d\mu_{\Gamma_n}(H) > 0\]
	for large $n \in \NN$. Therefore, $\phi_{\Gamma_n}^{-1}(\mc{U}) \subseteq \Gamma \backslash G$ is a non-empty open subset, whence there are $g_n \in G$ such that $\phi_{\Gamma_n}(\Gamma_n g_n) = g_n^{-1} \Gamma_n g_n \in \mc{U}$.
	Because $\mc{U}$ was an arbitrary open neighborhood of $\Gamma$ it follows that $[\Gamma_n] \to [\Gamma]$ as $n\to \infty$.
	
	This shows that $\iota \colon G \backslash \Lattices(\Sigma) \hookrightarrow \IRS(G)$ is a topological embedding.
\end{proof}

\subsection{Augmented Moduli Space and the IRS compactification} \label{subsect:augmodandirs}

We shall now construct an extension $\Phi \colon \augModuli(\Sigma) \longrightarrow \IRSModuli(\Sigma)$ of the topological embedding 
$\iota \colon \Moduli(\Sigma) \hookrightarrow \IRSModuli(\Sigma) \subseteq \IRS(G):$
\begin{center}
	\begin{tikzcd}
		& \augModuli(\Sigma) \arrow[dd,"\Phi"] \\
		\Moduli(\Sigma) \arrow[ur,hook] \arrow[dr,hook,"\iota"] & \\
		& \IRSModuli(\Sigma)
	\end{tikzcd}
\end{center}

\begin{defn}
	We define a map $\tilde{\Phi} \colon \augTeich(\Sigma) \longrightarrow \IRS(G)$ by
	\[\tilde{\Phi}( ([\rho_{\Sigma'}])_{\Sigma' \in c(\sigma)} ) 
	:= \sum_{\Sigma' \in c(\sigma)} 
	\frac{\chi(\Sigma')}{\chi(\Sigma)} 
	\cdot     \mu_{\im \rho_{\Sigma'}}\]
	for every $([\rho_{\Sigma'}])_{\Sigma' \in c(\sigma)} \in \Teich_\sigma(\Sigma) \subset \augTeich(\Sigma), \sigma \subset \mc{C}(\Sigma)$.
\end{defn}

\begin{thm}\label{thm:main}
	The map $\tilde{\Phi} \colon \augTeich(\Sigma) \longrightarrow \IRS(G)$ is well-defined and continuous. 
	
	Moreover,
	$\tilde{\Phi}$ descends to a continuous finite-to-one surjection
	\[ \Phi \colon \augModuli(\Sigma) \longrightarrow \IRSModuli(\Sigma)\]
	extending the topological embedding $\iota \colon \Moduli(\Sigma) \hookrightarrow \IRSModuli(\Sigma)$. There is a uniform upper bound $B(\Sigma)>0$ which depends only on the topology of $\Sigma$, such that $\# \Phi^{-1}(\mu) \leq B(\Sigma)$ for all $\mu \in \IRSModuli(\Sigma)$.
\end{thm}

We will need the following Lemmas for the proof. Recall that for every component $\Sigma' \in c(\sigma)$, $\sigma \subset \mc{C}(\Sigma)$, we obtain a monomorphism $\iota_{\Sigma'} \colon \pi_1(\Sigma') \hookrightarrow \pi_1(\Sigma)$ induced by the inclusion $\Sigma' \subset \Sigma$, which is well-defined up to conjugation; see Remark \ref{rem:DependenceOnComponent}.

\begin{lemma} \label{lem:ConvexCoreSubsurface}
	Let $\rho \in \Rep(\Sigma)$, let $\sigma \subset \mc{C}(\Sigma)$ be a simplex in the curve complex, let $\Sigma' \in c(\sigma)$ be a component, and let $\iota_{\Sigma'} \colon \pi_1(\Sigma') \hookrightarrow \pi_1(\Sigma)$ be an inclusion monomorphism. Denote $\Gamma = \im \rho$, $\Gamma' = \im (\rho \circ \iota_{\Sigma'})$ and let $\pi \colon \HH^2 \longrightarrow \Gamma \backslash \HH^2$ be the quotient map. Let $f \colon \Sigma \longrightarrow \Gamma \backslash \HH^2$ be an orientation preserving homeomorphism whose holonomy is $\rho$ such that $f(\sigma) = \tau$ is a collection of closed geodesics.
	
	Then $\tilde{\tau} \coloneqq \pi^{-1}(\tau) \subset \HH^2$ is a collection of disjoint geodesics and $\tilde{C}(\Gamma') \subseteq \HH^2$ is the closure of a connected component of $\HH^2 \setminus \tilde{\tau}$.
\end{lemma}

\begin{proof}
	Let $\tilde{f} \colon \tilde{\Sigma} \cong \HH^2 \longrightarrow \HH^2$ be a lift of $f \colon \Sigma \longrightarrow \Gamma \backslash \HH^2$. Let $\tilde{\Sigma}' \subset \HH^2 \setminus \tilde{\sigma}$ be a connected component over $\Sigma'$ such that the inclusion $ \tilde{\Sigma}' \hookrightarrow \tilde{\Sigma}\cong \HH^2$ is $\iota_{\Sigma'}$-equivariant; cf.\ Proposition and Definition \ref{propdefn:RestrictionMaps}. We set $\tilde{X}' \coloneqq \tilde{f}(\tilde{\Sigma}')$, and we want to show that $\tilde{C}(\Gamma') = \ol{\tilde{X}'}$.
	We will do so by showing that
	\[ \partial \tilde{X}' \subseteq L(\Gamma') \subseteq \ol{\partial \tilde{X}'}. \]
	
	Note that $\tilde{X}'$ is $\Gamma' = \rho(\iota_{\Sigma'}(\pi_1(\Sigma')))$-invariant. Therefore $\ol{\partial \tilde{X}'}$ is a closed $\Gamma'$-invariant subset of $\partial \HH^2$ that must contain the limit set $L(\Gamma')$ because the limit set is the smallest such subset.
	
	Let $\xi \in \partial \tilde{X}'$. If $\xi$ is fixed by a parabolic element $\eta \in \Gamma'$ then 
	\[\xi = \lim\limits_{n \to \infty} \eta^n o \in L(\Gamma'), \qquad o \in \HH^2.\]
	
	Hence, let us assume that $\xi$ is not fixed by any parabolic element in $\Gamma'$.
	Let $\gamma \subset \tilde{X}'$ be a geodesic from $\tilde{p} = \gamma(0)$ to $\xi = \gamma(\infty)$. Further, let $\{ P_j \}_{j \in \NN}$ be a system of disjoint horoballs centered at all the fixed points $\{\xi_j\}_{j \in \NN}$ of parabolic elements in $\Gamma'$. Then there is a sequence $(t_n)_{n \in \NN}$ such that $t_n \to \infty$ as $n \to \infty$ and $\gamma(t_n) \notin \bigsqcup_{j \in \NN} P_j$. Indeed, otherwise there would be a $T >0$ and $j_0 \in \NN$ such that $\gamma(t) \in P_{j_0}$ for all $t \geq T$. This in turn would imply that $\gamma(\infty) = \xi_{j_0} = \xi$; contradicting our assumption.
	
	Observe that $\Gamma'$ acts coboundedly on $\tilde{X}' \setminus \bigsqcup_{j \in \NN} P_j$. Therefore, there is $r>0$, $o \in \tilde{X}' \setminus \bigsqcup_{j \in \NN} P_j$, and $\gamma_n \in \Gamma'$, such that $d(\gamma(t_n),\gamma_n \cdot o) \leq r$ for all $n \in \NN$. Hence, 
	\[ \xi = \lim\limits_{n\to \infty} \gamma(t_n) = \lim\limits_{n \to \infty } \gamma_n \cdot o \in L(\Gamma'). \]
\end{proof}

\begin{lemma}\label{lem:SplittingFundDomains}
	Let $\rho \in \Rep(\Sigma)$ and let $\sigma \subset C(\Sigma)$ be a simplex in the curve complex. Further, let $\iota_{\Sigma'} \colon \pi_1(\Sigma') \hookrightarrow \pi_1(\Sigma)$ be an inclusion monomorphism for every component $\Sigma' \in c(\sigma)$. Denote $\Gamma = \im \rho$ and $\Gamma(\Sigma') = \im (\rho \circ \iota_{\Sigma'})$ for every $\Sigma' \in c(\sigma)$. Let $\{p(\Sigma') \in \HH^2 \, | \, \Sigma' \in c(\sigma)\}$ be a collection of points. Then
	\[ \bigcup_{\Sigma' \in c(\sigma)} \hat{D}_{p(\Sigma')}(\Gamma(\Sigma')) \]
	is a fundamental domain for the action of $\Gamma$ on $\HH^2$, and $\hat{D}_{p(\Sigma')}(\Gamma(\Sigma')) \cap \hat{D}_{p(\Sigma'')}(\Gamma(\Sigma''))$ has measure zero for distinct $\Sigma', \Sigma'' \in c(\sigma)$.
\end{lemma}

\begin{proof}
	For simplicity we enumerate $\{ \Sigma'_i \colon i = 1, \ldots, l \} = c(\sigma)$ and set $p_i= p(\Sigma'_i)$, $\Gamma(\Sigma'_i) = \Gamma'_i$ for every $i= 1, \ldots, l$. Further, denote by $q \colon \tilde{\Sigma} \longrightarrow \Sigma$ and $\pi \colon \HH^2 \longrightarrow \Gamma \backslash \HH^2$ the usual universal coverings, and let $f \colon \Sigma \longrightarrow \Gamma \backslash \HH^2$ be an orientation preserving homeomorphism with holonomy $\rho$ mapping $\sigma$ to a system of closed geodesics $f(\sigma) = \tau$. Let $\tilde{f} \colon \tilde{\Sigma} \longrightarrow \HH^2$ be a lift of $f$ to the universal cover. We set $\tilde{\sigma} \coloneqq q^{-1}(\sigma)$ and $\tilde{\tau} \coloneqq \pi^{-1}(\tau) = \tilde{f}(\tilde{\sigma})$. Let $\tilde{\Sigma}'_i \subset \tilde{\Sigma} \setminus \tilde{\sigma}$ be a connected component such that the inclusion $\tilde{\Sigma}'_i \hookrightarrow \tilde{\Sigma}$ is $\iota_{\Sigma'_i}$-equivariant. Because $\tilde{\Sigma} = \bigcup_{i =1}^l \pi_1(\Sigma) \cdot \ol{\tilde{\Sigma}'_i}$ we have that $\HH^2 =\bigcup_{i =1}^l \Gamma \cdot \ol{f(\tilde{\Sigma}'_i)} $. By Lemma \ref{lem:ConvexCoreSubsurface} $\ol{f(\tilde{\Sigma}'_i)} = \tilde{C}(\Gamma_i)$ such that
	\[\HH^2= \bigcup_{i=1}^l \Gamma \cdot \tilde{C}(\Gamma_i).\]
	Because $\hat{D}_{p_i}(\Gamma'_i)$ is a fundamental domain for the $\Gamma'_i$-action on $\tilde{C}(\Gamma'_i)$ it is readily verified that $\bigcup_{i=1}^l \hat{D}_{p_i}(\Gamma'_i)$ is a fundamental domain for the $\Gamma$-action on $\HH^2$.
	
	Finally, $\hat{D}_{p_i}(\Gamma'_i) \cap \hat{D}_{p_j}(\Gamma'_j) \subseteq \tilde{\tau}$ for every $i \neq j$ which has measure zero.
\end{proof}

\begin{lemma} \label{lem:FullGroupConv}
	Let $(\rho_n)_{n \in \NN} \subset \Rep^*(\Sigma)$, let $\sigma \subset C(\Sigma)$, let $\Sigma' \in c(\sigma)$ be a component, let $\iota_{\Sigma'} \colon \pi_1(\Sigma') \hookrightarrow \pi_1(\Sigma)$ be an inclusion monomorphism, and suppose that
	\[ \rho_n \circ \iota_{\Sigma'} \to \rho' \in \Rep(\Sigma') \qquad (n \to \infty).\]
	Denote $\Gamma_n = \im \rho_n$, $\Gamma_n(\Sigma') = \im (\rho_n \circ \iota_{\Sigma'})$ and $\Gamma' = \im \rho' \cong \pi_1(\Sigma')$.
	
	Then
	\[ \Gamma_n \to \Gamma' \qquad (n \to \infty). \]
\end{lemma}
\begin{proof}
	We shall check (C1) and (C2) of Proposition \ref{prop:charact_con_Chab}.
	\begin{enumerate}
		\item[(C1)] Let $\gamma' = \rho(c) \in \Gamma' = \im \rho'$. Then $\rho_n(\iota_{\Sigma'}(c)) \in \Gamma_n(\Sigma') \subseteq \Gamma_n$ converges to $\gamma'$ as $n\to \infty$.
		
		\item[(C2)] Let $(\gamma_{n_k})_{k \in \NN}$ be a convergent sequence with limit $g \in G$ and $\gamma_{n_k} \in \Gamma_{n_k}$. We need to show that $g \in \Gamma'$. By Proposition \ref{prop:geometric_equals_algebraic} we know that $\Gamma_{n}(\Sigma') = \im (\rho_n \circ \iota_{\Sigma'}) \to \Gamma'$. Thus it will be sufficient to prove that $\gamma_{n_k} \in \Gamma_{n_k}(\Sigma')$ for large $k$.
		
		Observe that an element $\alpha \in \pi_1(\Sigma)$ is in $\pi_1(\Sigma')$ if and only if $i(\alpha, c) = 0$ for every peripheral curve $c \in \pi_1(\Sigma')$. Indeed, if $\alpha$ is in $\pi_1(\Sigma')$ then clearly $i(\alpha,c) = 0$ for every peripheral curve $c \in \pi_1(\Sigma')$. Vice versa, let $\{c_1, \ldots, c_r\} \subset \pi_1(\Sigma')$ be primitive peripheral elements corresponding to the punctures of $\Sigma'$. If $i(\alpha,c_j) = 0$ then $\alpha$ and $c_j$ form a bigon for every $j=1,\ldots,r$. Now, we may homotope $\alpha$ by pushing these bigons inside of $\Sigma'$ such that $\alpha \in \pi_1(\Sigma')$.
		
		Suppose that there is a (further) subsequence such that $\gamma_{n_k} = \rho_{n_k}(\alpha_{n_k}) \notin \Gamma_{n_k}(\Sigma')$ for every $k \in \NN$. Then, for every $k \in \NN$, there is a primitive peripheral element $c'_k \in \pi_1(\Sigma')$ such that $i(c'_k,\alpha_{n_k}) \neq 0$. Up to passing to a further subsequence we may assume that $c'_k = c'$ is constant.
		
		Recall that $\rho' \in \Rep(\Sigma')$ such that peripheral elements are parabolic; see Remark \ref{rem:PeripheralsParabolic}. Thus $\ell(\rho_{n_k}(c')) \to 0 $ as $k \to \infty$. But then 
		\[ \ell(\rho_{n_k}(\alpha_{n_k})) \to \infty \quad (k \to \infty) \]
		by the Collar Lemma \ref{lem:CollarLemma}, contradicting the convergence $\gamma_{n_k} \to g$ as $k \to \infty$.
	\end{enumerate}
\end{proof}

After these preparations, we are ready to prove Theorem \ref{thm:main}.

\begin{proof}[Proof of Theorem \ref{thm:main}]
	Let $([\rho_{\Sigma'}])_{\Sigma' \in c(\sigma)} \in \Teich_\sigma(\Sigma) \subset \augTeich(\Sigma)$, $\sigma \subset \mc{C}(X)$. For each component $\Sigma' \in c(\sigma)$ the measure
	$ \mu_{\im \rho_{\Sigma'}}$ is an invariant random subgroup, since $[\rho_{\Sigma'}] \in \Teich(\Sigma')$. In order to verify that $\Phi$ is well-defined, we only need to check that 
	\[\tilde{\Phi}( ([\rho_{\Sigma'}])_{\Sigma' \in c(\sigma)} ) = \sum_{\Sigma' \in c(\sigma)} \frac{\chi(\Sigma')}{\chi(\Sigma)} \cdot \mu_{\im \rho_{\Sigma'}}\]
	is a convex combination. Observe that $\chi(\Sigma)<0$, and $\chi(\Sigma')<0$ for every $\Sigma' \in c(\sigma)$, such that
	\[ \frac{\chi(\Sigma')}{\chi(\Sigma)} > 0. \]
	Moreover, by the inclusion-exclusion principle, we have that
	\[ \chi(\Sigma) = \sum_{\Sigma' \in c(\sigma)} \chi(\Sigma'), \]
	whence
	\[ \sum_{\Sigma' \in c(\sigma)} 
	\frac{\chi(\Sigma')}{\chi(\Sigma)} =1. \]
	
	Let us prove that $\tilde{\Phi}$ is continuous. We will first prove this for a sequence
	$([\rho_n])_{n \in \NN} \subset \Teich(\Sigma)$ converging to $\mf{r} = ([\rho_{\Sigma'}])_{\Sigma' \in c(\sigma)}\in \Teich_\sigma(\Sigma) \subset \augTeich(\Sigma), \sigma \subset \mc{C}(\Sigma)$.
	By definition of the topology of $\augTeich(\Sigma)$ we know that for every $\Sigma' \in c(\sigma)$ we have $[\rho_n \circ \iota_{\Sigma'}] \to [\rho_{\Sigma'}]$ as $n \to \infty$, i.e.\ there are $g_n(\Sigma') \in G$ such that
	\[ g_n(\Sigma')^{-1}  \cdot  (\rho_n \circ \iota_{\Sigma'} ) \cdot g_n(\Sigma')\to \rho_{\Sigma'} \qquad (n \to \infty).\]
	In particular,
	\[ g_n(\Sigma')^{-1}  \cdot  \Gamma_n(\Sigma') \cdot g_n(\Sigma')\to \Gamma(\Sigma') \qquad (n \to \infty)\]
	where we set $\Gamma_n(\Sigma') \coloneqq \im (\rho_n \circ \iota_{\Sigma'} )$, $\Gamma(\Sigma') \coloneqq \im (\rho \circ \iota_{\Sigma'})$.
	
	Let $o \in \HH^2$. By Lemma \ref{lem:SplittingFundDomains} the set 
	\[ D_n := \bigcup_{\Sigma' \in c(\sigma)} \hat{D}_{g_n(\Sigma')o}(\Gamma_n(\Sigma')) \]
	is a fundamental domain for the action of $\Gamma_n = \im \rho_n$ on $\HH^2$. 
	
	Let $f \in C(\Sub(G))$. We have that
	\begin{align*}
	\int_{\Sub(G)} f(H) \, d\mu_{\Gamma_n}(H) 
	&= \nu(p^{-1}(D_n))^{-1} \cdot \int_{G} \II_{D_n}(go) \cdot f(g^{-1} \Gamma_n g) \, d\nu(g).
	\end{align*}
	Observe that $\nu(p^{-1}(D_n))= v(D_n) = 2 \pi \abs{\chi(\Sigma)}$. Further,
	\begin{align*}
	\int_{G} \II_{D_n}(go) \cdot f(g^{-1} \Gamma_n g) \, d\nu(g)
	&= \sum_{\Sigma' \in c(\sigma)} \int_{G}  \II_{\hat{D}_{g_n(\Sigma')o}(\Gamma_n(\Sigma'))}(go) \cdot f(g^{-1} \Gamma_n g) \, d\nu(g).
	\end{align*}
	Let $\Sigma' \in c(\sigma)$. Then
	\begin{align*}
	&\quad \int_{G}  \II_{\hat{D}_{g_n(\Sigma')o}(\Gamma_n(\Sigma'))}(go) \cdot f(g^{-1} \Gamma_n g) \, d\nu(g) \\
	&= \int_{G}  \II_{\hat{D}_{g_n(\Sigma')o}(\Gamma_n(\Sigma'))}(g_n(\Sigma')go) \cdot f(g^{-1} g_n(\Sigma')^{-1} \Gamma_ng_n(\Sigma') g) \, d\nu(g)\\ 
	&= \int_{G}  \II_{\hat{D}_{o}(g_n(\Sigma')^{-1}\Gamma_n(\Sigma')g_n(\Sigma'))}(go) \cdot f(g^{-1} g_n(\Sigma')^{-1} \Gamma_ng_n(\Sigma') g) \, d\nu(g)\\
	&= \int_{G}  \II_{\hat{F}_{o}(g_n(\Sigma')^{-1}\Gamma_n(\Sigma')g_n(\Sigma'))}(g) \cdot \bar{f}_{n,\Sigma'}(g) \, d\nu(g),\\
	\end{align*}
	where we used the left-invariance of the Haar measure $\nu$, that $\hat{D}_{o}(g_n(\Sigma')^{-1}\Gamma_n(\Sigma')g_n(\Sigma')) = g_n(\Sigma')^{-1} \cdot \hat{D}_{g_n(\Sigma')o}(\Gamma_n(\Sigma'))$, and set
	\[ \bar{f}_{n,\Sigma'}(g) \coloneqq f(g^{-1} g_n(\Sigma')^{-1} \Gamma_ng_n(\Sigma') g) \qquad \forall g \in G. \]
	Note that $\|\bar{f}_{n,\Sigma'}\|_{L^\infty} \leq \|f\|_\infty < \infty$, $n \in \NN$, and
	\[ g_n(\Sigma')^{-1} \cdot  \Gamma_n \cdot g_n(\Sigma')\to \Gamma(\Sigma') \quad (n\to \infty) \]
	by Lemma \ref{lem:FullGroupConv}. Thus, if we set
	\[ \bar{f}_{\Sigma'}(g) \coloneqq f(g^{-1}  \Gamma(\Sigma')  g) \qquad \forall g \in G, \]
	then $\bar{f}_{n, \Sigma'}(g) \to \bar{f}(g)$ as $n \to \infty$ for every $g \in G$, by continuity. Moreover,
	\[\left \| \II_{\hat{F}_{o}(g_n(\Sigma')^{-1}\Gamma_n(\Sigma')g_n(\Sigma'))} - \II_{\hat{F}_o(\Gamma(\Sigma'))} \right \|_{L^1(g,\nu)} \to 0 \quad (n\to \infty)\]
	by Lemma \ref{lem:L1ConvTruncDom}. It follows that
	\[\int_{G}  \II_{\hat{F}_{o}(g_n(\Sigma')^{-1}\Gamma_n(\Sigma')g_n(\Sigma'))}(g) \cdot \bar{f}_{n,\Sigma'}(g) \, d\nu(g)
	\to 
	\int_{G}  \II_{\hat{F}_{o}(\Gamma(\Sigma'))}(g) \cdot \bar{f}_{\Sigma'}(g) \, d\nu(g)\]
	as $n \to \infty$, by Lemma \ref{lem:ProdConvL1}.

	All in all, we obtain that the integral
	\[\int_{\Sub(G)} f(H) \, d\mu_{\Gamma_n}(H)\]
	tends to
	\begin{align*}
	&\quad (2\pi \abs{\chi(\Sigma)})^{-1} \sum_{\Sigma' \in c(\sigma)} \int_G \II_{\hat{F}_o(\Gamma(\Sigma'))}(g) \cdot f(g^{-1} \Gamma(\Sigma') g) \, d\nu(g)\\
	&= \sum_{\Sigma' \in c(\sigma)} \frac{2\pi \abs{\chi(\Sigma')}}{2 \pi \abs{\chi(\Sigma)}} \cdot \nu(\hat{F}_o(\Gamma(\Sigma')))^{-1} \int_G \II_{\hat{F}_o(\Gamma(\Sigma'))}(g) \cdot f(g^{-1} \Gamma(\Sigma') g) \, d\nu(g)\\
	&= \sum_{\Sigma' \in c(\sigma)} \frac{\chi(\Sigma')}{\chi(\Sigma)} \cdot \int_{G} f(H) \, d\mu_{\im \rho_{\Sigma'}}(H) 
	= \int_G f(H) \, d\tilde{\Phi}(\mf{r})
	\end{align*}
	as $n \to \infty$.
	
	In general, let $\mf{r}_n = ([\rho^{(n)}_{\Sigma''}])_{\Sigma'' \in c(\sigma_n)} \subset \augTeich(\Sigma)$ converge to $\mf{r} = ([\rho_{\Sigma'}])_{\Sigma' \in c(\sigma)}$ as $n \to \infty$. Then $\sigma_n \subseteq \sigma$ for large $n$. Because the simplex $\sigma$ has only finitely many faces we may assume without loss of generality\footnote{Just pass to a subsequence and treat every face separately.} that  $\sigma_n = \sigma'$ for large $n$. Applying our previous discussion to every component $\Sigma'' \in c(\sigma')$ we obtain
	\begin{align*}
	&\quad \int_{\Sub(G)} f(H) \, d\tilde{\Phi}(\mf{r})(H) \\
	&= \sum_{\Sigma' \in c(\sigma)} \frac{\chi(\Sigma')}{\chi(\Sigma)} \int_{\Sub(G)} f(H) \, d\mu_{\im \rho_{\Sigma'}}(H) \\
	&= \sum_{\Sigma'' \in c(\sigma')} \frac{\chi(\Sigma'')}{\chi(\Sigma)} \sum_{\substack{\Sigma' \in c(\sigma) \\ \Sigma' \subseteq \Sigma''}}
	\frac{\chi(\Sigma')}{\chi(\Sigma'')} \int_{\Sub(G)} f(H) \, d\mu_{\im \rho_{\Sigma'}}(H) \\
	&= \sum_{\Sigma'' \in c(\sigma')} \frac{\chi(\Sigma'')}{\chi(\Sigma)} 
	\lim\limits_{n \to \infty} \int_{\Sub(G)} f(H) \, d \mu_{\im \rho^{(n)}_{\Sigma''}}(H) \\
	&= \lim\limits_{n\to \infty} \int_{\Sub(G)} f(H) \, d\tilde{\Phi}(\mf{r}_n)(H)
	\end{align*}
	for every $f \in C(\Sub(G))$. This shows that $\tilde{\Phi} \colon \augTeich(\Sigma) \longrightarrow \IRS(G)$ is continuous.
	
	Observe that $\tilde{\Phi}(\mf{r})$ depends only on the set of $G$-conjugacy classes $\{ [\im \rho_{\Sigma'}] \in \Moduli(\Sigma') \, | \, \Sigma' \in c(\sigma)\}$. This set remains unaffected by the mapping class group action such that $\tilde{\Phi}$ descends to a continuous map 
	$\Phi \colon \augModuli(\Sigma) \longrightarrow \IRS(G)$.
	
	By definition $\Phi|_{\Moduli(\Sigma)} = \iota \colon \Moduli(\Sigma) \longrightarrow \IRS(G)$ holds. We want to show that $\Phi$ is surjective, i.e.\
	\[ \Phi(\augModuli(\Sigma)) = \ol{\iota(\Moduli(\Sigma))}.\]
	Since $\augModuli(\Sigma)$ is compact and $\Phi$ is continuous the image $\Phi(\augModuli(\Sigma)$ is compact and contains $\iota(\Moduli(\Sigma)) = \Phi(\Moduli(\Sigma))$. Because $\IRS(G)$ is Hausdorff, compact subsets are closed such that $\ol{\iota(\Moduli(\Sigma))} \subseteq \Phi(\augModuli(\Sigma))$. Vice versa, let $\mu \in \ol{\iota(\Moduli(\Sigma))}$ and let $[[\rho_n]]_{n \in \NN} \subset \Moduli(\Sigma)$ be a sequence such that $\iota([[\rho_n]]) = \mu_{\im \rho_n}$ converges to $\mu$ as $n \to \infty$. Because $\augModuli(\Sigma)$ is compact there is a convergent subsequence $[[\rho_{n_k}]] \to [\mf{r}] \in \augModuli(\Sigma)$ as $k \to \infty$. Because $\Phi$ is continuous it follows that
	\[ \mu = \lim_{k\to \infty} \Phi([[\rho_{n_k}]]) = \Phi([\mf{r}]) \in \Phi(\augModuli(\Sigma)).\]
	
	Finally, we want to prove that $\Phi$ is finite-to-one. Given $\sigma \subset \mc{C}(\Sigma)$ consider the assembly map from Definition \ref{defn:AssemblyMap}
	\[ A_\sigma \colon \prod_{\Sigma' \in c(\sigma)} \Moduli^*(\Sigma') \longrightarrow \augModuli(\Sigma).\]
	
	Recall that $\Moduli^*(\Sigma') = \Teich(\Sigma')/ \PMCG(\Sigma')$ and that we have a short exact sequence 
	\[1 \longrightarrow \PMCG(\Sigma') \longrightarrow \MCG(\Sigma') \longrightarrow \Sym(p(\Sigma')) \longrightarrow 1\]
	where $p(\Sigma')$ denotes the number of punctures of $\Sigma'$ and $\Sym(p(\Sigma'))$ its symmetric group. Hence, we obtain a well-defined action of $\Sym(p(\Sigma')) \cong \MCG(\Sigma')/\PMCG(\Sigma')$ on $\Moduli^*(\Sigma')$ and a quotient map 
	\[\Moduli^*(\Sigma') \longrightarrow \Moduli(\Sigma') \cong \Moduli^*(\Sigma') / \Sym(p(\Sigma')) \] 
	that is at most $p(\Sigma') ! $-to-one. These induce a quotient map
	\[ Q_\sigma \colon \prod_{\Sigma' \in c(\sigma)} \Moduli^*(\Sigma') \longrightarrow \prod_{\Sigma' \in c(\sigma)} \Moduli(\Sigma').\]
	Consequently, $\prod_{\Sigma' \in c(\sigma)} (p(\Sigma') !)$ is an upper bound for the cardinality of any fiber of this map.

	In fact, there is an upper bound on the cardinality of the fiber that depends only on $\Sigma$. First, observe that $\chi(\Sigma') \leq -1$ for every $\Sigma' \in c(\sigma)$, so that $\# c(\sigma) \leq \abs{\chi(\Sigma)}$ by $\chi(\Sigma) = \sum_{\Sigma' \in c(\sigma)} \chi(\Sigma')$. Further, if $p(\Sigma')$ is the number of punctures and $g(\Sigma')$ is the genus of $\Sigma' \in c(\sigma)$ then $\chi(\Sigma') = 2 - 2 g(\Sigma') - p(\Sigma')$, such that
	\[ p(\Sigma') = 2 - 2 g(\Sigma') + \abs{\chi(\Sigma')} \leq \abs{\chi(\Sigma)} + 2.\]
	Hence,
	\[\prod_{\Sigma' \in c(\sigma)} (p(\Sigma') !) \leq (\abs{\chi(\Sigma)}+2) !^{\abs{\chi(\Sigma)}} \eqqcolon B_1(\Sigma)\]
	is an upper bound for the cardinality of the fiber of $Q_\sigma$.
	
	Now, consider the restriction
	\begin{align*}
	\widetilde{\Phi}_\sigma \coloneqq \widetilde{\Phi}|_{\Teich_\sigma(\Sigma)} \colon \Teich_{\sigma}(\Sigma) &\longrightarrow \IRSModuli(\Sigma),\\
	([\rho_{\Sigma'}])_{\Sigma' \in c(\sigma)} &\longmapsto \sum_{\Sigma' \in c(\sigma)} \frac{\chi(\Sigma')}{\chi(\Sigma)} \cdot \mu_{\im \rho_{\Sigma'}}.
	\end{align*}
	Note that the right-hand-side depends only on the conjugacy classes $([\im \rho_{\Sigma'}])_{\Sigma' \in c(\sigma)}$ such that we obtain a map
	\begin{align*}
	\Psi_{\sigma} \colon \prod_{\Sigma' \in c(\sigma)} \Moduli(\Sigma') &\longrightarrow \IRSModuli(\Sigma), \\
	([\Gamma_{\Sigma'}])_{\Sigma' \in c(\sigma)} &\longmapsto \sum_{\Sigma' \in c(\sigma)} \frac{\chi(\Sigma')}{\chi(\Sigma)} \cdot \mu_{\Gamma_{\Sigma'}}.
	\end{align*}
	Thus we have the following commutative diagram:
	\begin{center}
		\begin{tikzcd}[row sep=large]
			& \displaystyle \prod_{\Sigma' \in c(\sigma)} \Moduli^*(\Sigma') 
			\arrow[r,"Q_\sigma"]
			\arrow[d,"A_\sigma"]
			& \displaystyle \prod_{\Sigma' \in c(\sigma)} \Moduli(\Sigma') 
			\arrow[d,"\Psi_\sigma"]
			\\
			& \augModuli(\Sigma)
			\arrow[r,"\Phi"]
			& \IRSModuli(\Sigma)\\
			\Teich_{\sigma}(\Sigma)
			\arrow[uur, bend left]
			\arrow[ur]
			\arrow[urr,bend right,"\widetilde{\Phi}_\sigma"]
			& &
		\end{tikzcd}
	\end{center}
	
	Given an element $\mu$ in the image of $\Psi_\sigma$ there is a unique set of distinct conjugacy classes of lattices $\{ [\Gamma_1], \ldots, [\Gamma_m]\}$ and $\lambda_1, \ldots, \lambda_m > 0$, $\lambda_1 + \cdots + \lambda_m =1$, such that $\mu = \sum_{i=1}^m \lambda_i \cdot \mu_{\Gamma_i}$, by Lemma \ref{lem:LinearIndependenceIRS}. A rough upper bound on the number of preimages in $\Psi_\sigma^{-1}(\mu)$ is given by 
	\[\# c(\sigma) ! \leq \abs{\chi(\Sigma)} ! \eqqcolon B_2(\Sigma)\] 
	as we have at most that many different possibilities of assigning the conjugacy classes $\{[\Gamma_1], \ldots, [\Gamma_m]\}$ to the components $c(\sigma)$. \footnote{In fact, we are ignoring that  $\Gamma_i \backslash \HH^2$ needs to be homeomorphic to $\Sigma' \in c(\sigma)$ if $[\Gamma_i]$ was assigned to $\Sigma' \in c(\sigma)$.}
	
	By Proposition \ref{prop:FinitelyCoveredByAssemblyMaps} $\augModuli(\Sigma)$ is covered by finitely many images of assembly maps
	\[ A \colon \bigsqcup_{i \in I} \prod_{\Sigma' \in c(\sigma_i)} \Moduli^*(\Sigma') \longrightarrow \augModuli(\Sigma),\]
	where $\{ \sigma_i \colon i \in I \}$ is a finite collection of representatives for each orbit of simplices in $\mc{C}(\Sigma)$ under the mapping class group action. We define maps 
	\[Q \colon \bigsqcup_{i \in I} \prod_{\Sigma' \in c(\sigma_i)} \Moduli^*(\Sigma') \longrightarrow  \bigsqcup_{i \in I} \prod_{\Sigma' \in c(\sigma_i)} \Moduli(\Sigma')\] 
	via $Q_{\sigma_i}$ on $\prod_{\Sigma' \in c(\sigma_i)} \Moduli^*(\Sigma')$, and
	\[\Psi \colon \bigsqcup_{i \in I} \prod_{\Sigma' \in c(\sigma_i)} \Moduli(\Sigma') \longrightarrow \IRSModuli(\Sigma)\] 
	via $\Psi_{\sigma_i}$ on $\prod_{\Sigma' \in c(\sigma_i)} \Moduli(\Sigma')$, for every $i \in I$. The cardinality of the fibers of $Q$ and $\Psi$ are then bounded by $B_1(\Sigma)$ and $B_2(\Sigma)$ respectively. Moreover, we have the following commutative diagram of surjective maps:
	\begin{center}
		\begin{tikzcd}[row sep = large, column sep = large]
			\displaystyle\bigsqcup_{i \in I} \prod_{\Sigma' \in c(\sigma_i)} \Moduli^*(\Sigma')
			\arrow[r,"Q", two heads] \arrow[d, "A", two heads]
			&\displaystyle\bigsqcup_{i \in I} \prod_{\Sigma' \in c(\sigma_i)} \Moduli(\Sigma')
			\arrow[d,"\Psi",two heads]
			\\
			\augModuli(\Sigma)
			\arrow[r,"\Phi",two heads]
			&\IRSModuli(\Sigma)
		\end{tikzcd}
	\end{center}
	
	Let $\mu \in \IRSModuli(\Sigma)$. Then
	\[ \# \Phi^{-1}(\mu) = \# A( Q^{-1}(\Psi^{-1}(\mu))) \leq \# Q^{-1}(\Psi^{-1}(\mu)) \leq B_1(\Sigma) \cdot B_2(\Sigma) \eqqcolon B(\Sigma), \]
	and the proof is complete.
\end{proof}

Before we provide a proof of Lemma \ref{lem:L1ConvTruncDom} in the next section, we conclude our current discussion with a minimal example that shows that there are points $\mu \in \IRSModuli(\Sigma)$ whose preimages $\Phi^{-1}(\mu)$ consists of more than one point. 

\begin{figure}[b]
	\centering
	\def\svgwidth{\linewidth}
	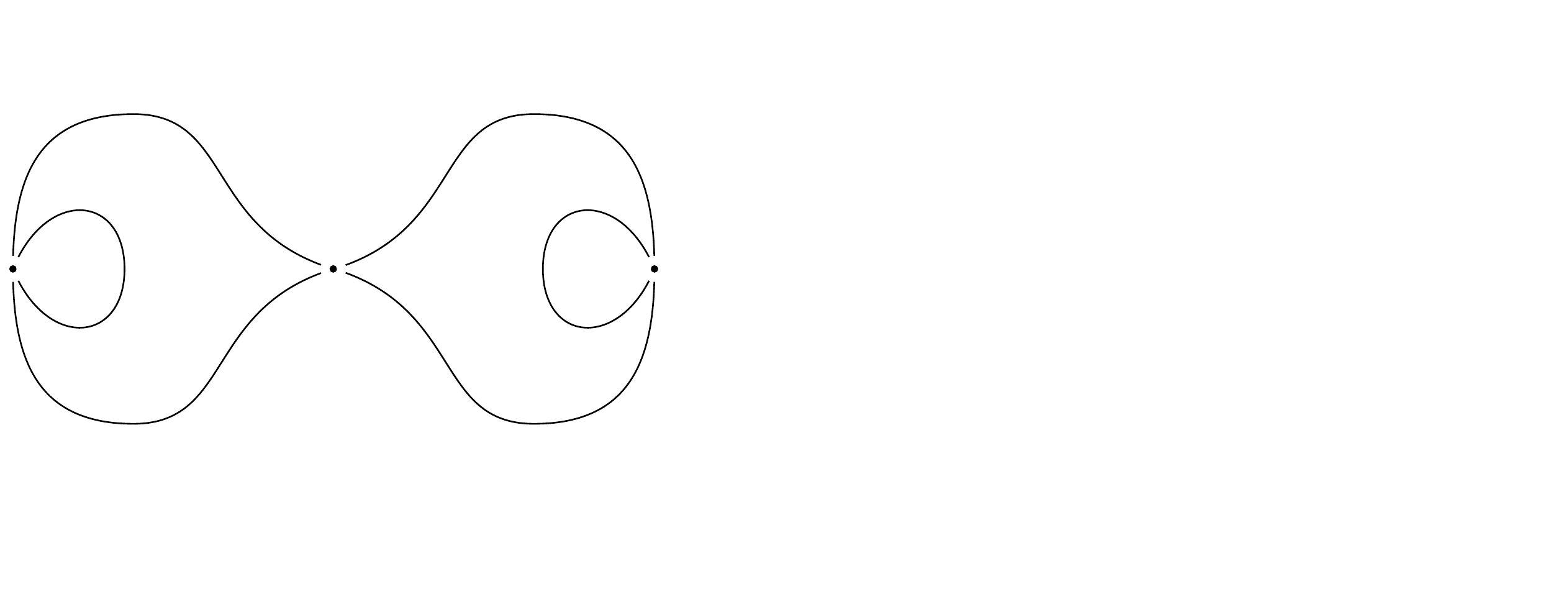
	\caption{Both points $[\mf{r}_1], [\mf{r}_2] \in \augModuli(\Sigma)$ are mapped to the same invariant random subgroup $\mu_{\Gamma_0} = \Phi([\mf{r}_1]) = \Phi([\mf{r}_2])  \in \IRS(G)$; see Example \ref{ex:preimage}.}
	\label{fig:example}
\end{figure}

\begin{example}\label{ex:preimage}
	Let $\Sigma = \Sigma_{2,0}$ be a closed surface of genus $2$.
	Let $\sigma_1 = \{ \alpha_1, \beta_1, \gamma_1\} \subset \mc{C}(\Sigma)$ be a pants decomposition of $\Sigma$ where $\alpha_1, \gamma_1$ are non-separating curves and $\beta_1$ is separating. Further, let $\sigma_2 = \{ \alpha_2, \beta_2, \gamma_2\} \subset \mc{C}(\Sigma)$ be a pants decomposition where $\alpha_2, \beta_2, \gamma_2$ are all non-separating. The Teichm\"uller space $\Teich(\Sigma_{0,3}) =\{ [\rho_0] \}$ of a thrice-punctured-sphere is just one point. Consider the elements $\mf{r}_1=  ([\rho_0])_{\Sigma' \in c(\sigma_1)} \in \Teich_{\sigma_1}(\Sigma)$, $\mf{r}_2=  ([\rho_0])_{\Sigma' \in c(\sigma_2)} \in \Teich_{\sigma_1}(\Sigma)$ and their images $[\mf{r}_1], [\mf{r}_2] \in \augModuli(\Sigma)$; see Figure \ref{fig:example}. Clearly, $[\mf{r}_1] \neq [\mf{r}_2]$ because $\sigma_1$ and $\sigma_2$ are not in the same mapping class group orbit in $\mc{C}(\Sigma)$. However, 
	\[ \Phi([\mf{r}_1]) = \mu_{\Gamma_0} = \Phi([\mf{r}_2]) \in \IRSModuli(\Sigma)\]
	where $\Gamma_0 = \im \rho_0$.
\end{example}


\section{Proof of Lemma \ref{lem:L1ConvTruncDom}} \label{section:ProofOfLemma}

We will use elementary hyperbolic geometry in our proof of Lemma \ref{lem:L1ConvTruncDom}.
The proof is split into several sub-lemmas, some of which might be of individual interest.

The first two Lemmas are concerned with the pointwise convergence of the characteristic functions $\II_{\tilde{C}(\Gamma_n)}$ and $\II_{D_o(\Gamma_n)}$ outside of a set of measure zero.

\begin{lemma}\label{lem:CoreConv}
	Let $(\Gamma_n)_{n \in \NN} \subset \mc{D}(\Sigma)$ be a sequence converging to $\Gamma \in \mc{D}(\Sigma)$. Then 
	\[ \II_{\tilde{C}(\Gamma_n)}(x) \to \II_{\tilde{C}(\Gamma)}(x) \qquad (n \to \infty)\]
	for every $x \in \HH^2 \setminus \partial\tilde{C}(\Gamma)$.
\end{lemma}

\begin{proof}
	By Proposition \ref{prop:geometric_equals_algebraic} we may choose $\rho, \rho_n \in \Rep^*(\Sigma)$ such that $\im \rho = \Gamma$, $\im \rho_n = \Gamma_n$ for large $n$, and $\rho_n \to \rho$ as $n \to \infty$. 
	
	First, let $x \in \HH^2 \setminus \tilde{C}(\Gamma)$. Let $\gamma \in \Gamma$ be a hyperbolic element whose axis bounds an open half-space $H(\gamma)$ such that $H(\gamma) \cap \tilde{C}(\Gamma) = \emptyset$ and $x \in H(\gamma)$. Let $c \in \pi_1(\Sigma)$ such that $\rho(c) = \gamma$. Then $\rho_{n}(c) = \gamma_{n}$ are boundary elements converging to $\gamma$. But then $x \in H(\gamma_{n})$  such that $x \in \HH^2 \setminus\tilde{C}(\Gamma_{n})$, for large $n \in \NN$.
	
	Let $x \in \interior(\tilde{C}(\Gamma))$ and suppose that there is a subsequence $(n_k)_{k \in \NN}$ such that $x \notin \tilde{C}(\Gamma_{n_k})$ for all $k \in \NN$. Then there are primitive hyperbolic boundary elements $\gamma_{n_k} \in \Gamma_{n_k}$ whose axes $\ax(\gamma_{n_k})$ bound half-spaces $H(\gamma_{n_k}) \subset \HH^2$ such that $x \in H(\gamma_{n_k})$. 
	
	We claim that there is $D \geq 0$ such that $d(x,\ax(\gamma_{n_k}))<D$. 
	Suppose to the contrary that there is a further subsequence, also denoted by $(n_k)_{k \in \NN}$, such that $d(x, \ax(\gamma_{n_k})) \to \infty$ as $k \to \infty$. If $\eta_{n_k}, \eta'_{n_k} \in \Gamma_{n_k}$ converge to some elements $\eta, \eta' \in \Gamma$ then $\eta_{n_k} x, \eta'_{n_k} x \in H(\gamma_{n_k})$ for large $k$. However, this means that $\eta_{n_k}$ and $\eta'_{n_k}$ must leave the entire half-space $H(\gamma_{n_k})$ invariant. Thus they commute having the same axis $\ax(\gamma_{n_k})$. But then also $\eta$ and $\eta'$ commute. This contradicts the fact that $\Gamma$ is not abelian!
	
	Because $\gamma_{n_k}$ are primitive boundary elements their translation length $\ell(\gamma_{n_k})$ is uniformly bounded by the maximal length $L$ of a boundary curve in $C(\Gamma_{n_k})$.
	It is not hard to see that the subset
	\[ C = \{ g \in \PSL_2(\RR) \, | \, g \text{ is hyperbolic}, \ell(g) \leq L, \ax(g) \cap \ol{B}_x(D) \neq \emptyset  \}\]
	is compact, such that $\gamma_{n_k} \to \gamma \in \Gamma$ up to a subsequence.
	
	Let $c \in \pi_1(\Sigma)$ such that $\rho(c) = \gamma$. Because $\gamma_{n_k} \to \gamma$ and $\rho_{n_k}(c) \to \gamma$ as $k \to \infty$, we have that $\gamma_{n_k} = \rho_{n_k}(c)$ by Lemma \ref{lem:nbhddiscrete}. Hence, $\gamma=\rho(c)$ is a boundary element. Since $x \in H(\gamma_{n_k})$ and $\gamma_{n_k} \to \gamma$, it follows that $x \in \ol{H(\gamma)}$ contradicting $x \in \interior(\tilde{C}(\Gamma))$.
\end{proof}

\begin{lemma} \label{lem:DomConv}
	Let $(\Gamma_n)_{n \in \NN} \subset \mc{D}(\Sigma)$ be a sequence converging to $\Gamma \in \mc{D}(\Sigma)$. Then 
	\[ \II_{D_o(\Gamma_n)}(x) \to \II_{D_o(\Gamma)}(x) \qquad (n \to \infty)\]
	for every $x \in \HH^2 \setminus \partial D_o(\Gamma)$.
\end{lemma}
\begin{proof}
	Let $x \in \interior(D_o(\Gamma))$. Then $d(x,o) < d(x,\gamma o)$ for every $\gamma \in \Gamma \setminus \{1\}$. We want to show that $x \in D_o(\Gamma_n)$ for large $n$. Assume to the contrary that there is a subsequence $(n_k)_{k \in \NN}$ such that $x \not\in D_o(\Gamma_{n_k})$, i.e.\ there are $\gamma_{n_k} \in \Gamma_{n_k} \setminus \{1\}$ such that $d(x,o) > d(x, \gamma_{n_k} o)$ for every $k \in \NN$. Up to passing to a subsequence we may assume that $\gamma_{n_k} \to \gamma \in \Gamma$, and $\gamma \neq 1$ by Lemma \ref{lem:nbhddiscrete}. But
	\[ d(x,o) \geq \lim_{k \to \infty} d(x, \gamma_{n_k} o) = d(x, \gamma o)\]
	contradicting $x \in \interior(D_o(\Gamma))$.
	
	Let $x \in \HH^2 \setminus D_o(\Gamma)$. Then there is $\gamma \in \Gamma \setminus \{1\}$ such that $d(x,o) > d(x, \gamma o)$. We want to show that $x \in \HH^2 \setminus D_o(\Gamma_n)$ for large $n$. Let $\gamma_n \in \Gamma_n$ such that $\gamma_n \to \gamma \neq 1$ as $n \to \infty$. Then $\gamma_n \neq 1$ and $d(x,\gamma_n o) < d(x,o)$ such that $x \notin D_o(\Gamma_n)$, for large $n$.
\end{proof}

We give a characterization of peripheral curves, now.

\begin{lemma}\label{lem:CharctPeripheralElements}
	Let $\mu = \{ \gamma_1, \ldots, \gamma_r\} \subset \Sigma$ be a filling collection of essential simple closed curves such that
	\begin{enumerate}
		\item $\gamma_i, \gamma_j$ are in minimal position for all $i,j \in \{1,\ldots,r\}$,
		\item the curves in $\mu$ are pairwise non-isotopic, and
		\item for distinct triples $i,j,k \in \{1, \ldots, r\}$ at least one of the intersections $\gamma_i \cap \gamma_j, \gamma_j \cap \gamma_k, \gamma_i \cap \gamma_k$ is empty.
	\end{enumerate}
	
	Let $\alpha \subset \Sigma$ be a homotopically non-trivial closed curve. Then
	\begin{center}
		$\alpha$ is peripheral $\iff$ $i(\alpha,\gamma_i) = 0 $ for every $i=1, \ldots, r$.
	\end{center}
\end{lemma}

\begin{proof}
	Suppose $\alpha \subset \Sigma$ is peripheral. Then $\alpha$ is homotopic to one of the punctures of $\Sigma$. Since $\mu$ fills $\Sigma$ there is a punctured disk $\mathbb{D}^\times \subset \Sigma \setminus \mu$ surrounding this puncture. Thus we may homotope $\alpha$ into $\mathbb{D}^\times$ such that $i(\alpha,\gamma_i) = 0$ for every $i=1, \ldots, r$.
	
	If $i(\alpha,\gamma_i) = 0$ for every $i=1,\ldots,r$, then there are isotopies moving $\gamma_i$ to $\tilde{\gamma}_i$ such that $\tilde{\gamma}_i \cap \alpha = \emptyset$. Because our system satisfies the hypotheses (i)-(iii), there is an isotopy of $\Sigma$ moving $\bigcup_{i=1}^r \gamma_i$ to $\bigcup_{i=1}^r \tilde{\gamma}_i$; see \cite[Lemma 2.9]{farbmarg}. The collection $\tilde{\mu} = \{ \tilde{\gamma}_1,\ldots,\tilde{\gamma}_r\}$ is still filling and $\alpha$ is a homotopically non-trivial closed curve in $\Sigma \setminus \tilde{\mu}$. Thus $\alpha$ is contained in a punctured disk $\mathbb{D}^\times \subset \Sigma \setminus \mu$. Therefore, $\alpha$ is  homotopic to a puncture, i.e.\ $\alpha$ is peripheral.
\end{proof}

Using the previous Lemma we will prove next that there is a lower bound for the length of essential curves with respect to a convergent sequence of representations:
\begin{lemma} \label{lem:SmallLengthImpliesPeripheral}
	Let $(\rho_n)_{n \in \NN} \subset \Rep^*(\Sigma)$ be a sequence converging to $\rho \in \Rep^*(\Sigma)$. Then there is $\epsilon > 0$ such that
	\[ \ell(\rho_n(\alpha)) < \epsilon \implies \alpha \text{ is peripheral}\]
	for every $\alpha \in \pi_1(\Sigma)$ and all $n \in \NN$.
\end{lemma}

\begin{proof}
	Suppose to the contrary that there is a subsequence $(n_k)_{k \in \NN} \subset \NN$ and non-peripheral elements $\alpha_{n_k} \in \pi_1(\Sigma)$ such that $\ell(\rho_{n_k}(\alpha_{n_k})) \to 0$ as $k \to \infty$.
	Choose a collection of curves $\mu = \{\gamma_1, \ldots, \gamma_r\} \subset \pi_1(\Sigma)$ as in Lemma \ref{lem:CharctPeripheralElements}. Then there is $j_k \in \{1, \ldots, r\}$ for every $k \in \NN$ such that $i(\gamma_{j_k}, \alpha_{n_k}) \neq 0$. Up to a subsequence we may assume that $\gamma_{j_k} = \hat{\gamma} \in \mu$ is constant. But then by the Collar Lemma \ref{lem:CollarLemma}
	\[ \ell(\rho_{n_k}(\hat{\gamma})) \to \infty \qquad (k \to \infty).\]
	This contradicts the fact that $\ell(\rho_{n_k}(\hat{\gamma})) \to \ell(\rho(\hat{\gamma}))$ as $k \to \infty$.
\end{proof}

The following Lemma shows how to obtain an upper bound on the diameter of a connected subset $C$ of a hyperbolic surface given a lower bound $\epsilon$ for the injectivity radius and an upper bound for the volume of the $\epsilon$-neighborhood of $C$.

\begin{lemma} \label{lem:diam_bound_thick_part}
	Let $X$ be a hyperbolic surface and let $\epsilon > 0$. Further, let $C \subseteq X$ be a path-connected Borel set, and suppose that $\InjRad_X(x) \geq \epsilon$ for every $x \in C$. Then
	\[ \diam_X(C) \leq \frac{4 \epsilon}{v(B_o(\epsilon))} \cdot \vol_X(N_\epsilon(C))  \qquad (o \in \HH^2)\]
	where $N_\epsilon(C) = \{ x \in X \, | \, d_X(x,C) < \epsilon\}$ is the open $\epsilon$-neighborhood of $C$ in $X$.
\end{lemma}

\begin{proof}
	Note that for every $x \in C$ the ball $B_x(\epsilon) \subset N_\epsilon(C)$ is embedded and has the same measure as a ball of radius $\epsilon$ in the hyperbolic plane.
	
	Let us now consider
	$$ S = \{ Y \subset C : d(y_1, y_2) \geq 2 \epsilon \quad \forall y_1, y_2 \in Y, y_1 \neq y_2 \}.$$
	By Zorn's Lemma we may choose a maximal element $Y_0 \in S$ with respect to inclusion $\subseteq$. We claim that the collection of balls $ \{ B_y(2 \epsilon) : y \in Y_0 \} $ covers $C$. Indeed, if there is $y' \in C$ which is not in any $\{B_y(2\epsilon)\}_{y \in Y_o}$ it has distance greater or equal than $2 \epsilon$ from any point $y \in Y_0$. But then $\{y'\} \cup Y_0 \in S$ which contradicts the maximality of $Y_0$.
	
	On the other hand the balls of radius $\epsilon$ centered at $y \in Y_0$ are disjoint by definition of $S$ whence
	\[ \bigsqcup_{y \in Y_0} B_\epsilon(y) \subseteq N_\epsilon(C), \]
	such that
	\[ \vol_X(N_\epsilon(C)) \geq \sum_{y \in Y_0} \vol_X(B_\epsilon(y)) = \# Y_0 \cdot v(B_o(\epsilon)). \]
	It follows that
	\[\# Y_0\leq \frac{\vol_X(N_\epsilon(C))}{v(B_o(\epsilon))}.\]
	
	Any path in $C$ is covered by $\{ B_y(2 \epsilon) \}_{y \in Y_0}$ and we obtain that
	\[ \diam_X(C) \leq 4 \epsilon \cdot \#Y_0  \leq  \frac{4 \epsilon}{v(B_o(\epsilon))} \cdot \vol_X(N_\epsilon(C)) .\]
	
\end{proof}

By the previous Lemmas, we obtain an upper bound on the diameter of the thick-part of the convex core, as follows.

\begin{lemma}\label{lem:concrete_diam_bound_thick_part}
	Let $(\Gamma_n)_{n \in \NN} \subset \mc{D}(\Sigma)$ be a sequence converging to $\Gamma \in \mc{D}(\Sigma)$. Then there is $\epsilon >0$ such that for every $0< \epsilon' < \epsilon$ the $\epsilon'$-thick-part of the convex core
	\[ C(\Gamma_n)_{\geq \epsilon'} \coloneqq C(\Gamma_n) \cap (\Gamma_n \backslash \HH^2)_{\geq \epsilon'} = \{ x \in C(\Gamma_n) \, | \, \InjRad_{\Gamma_n \backslash \HH^2}(x) \geq \epsilon' \} \]
	is path-connected for every $n \in \NN$.
	
	In particular, for every $0< \epsilon ' < \epsilon$ there is $R=R(\epsilon')>0$ such that 
	\[ \diam_{\Gamma_n \backslash \HH^2}(C(\Gamma_n)_{\geq \epsilon'}) \leq R  \]
	for every $n \in \NN$.
\end{lemma}

\begin{proof}
	We may choose $\rho_n \to \rho \in \Rep^*(\Sigma)$ such that $\im \rho_n = \Gamma_n$ and $\im \rho = \Gamma$. Let $\epsilon>0$ be as in Lemma \ref{lem:SmallLengthImpliesPeripheral}. Without loss of generality we may assume that $\epsilon$ is smaller than the Margulis constant; see \cite{benedettipetronio} or \cite{martelli} for example. Let $0 < \epsilon' < \epsilon$, $n \in \NN$, and let $T \subset (\Gamma_n \backslash \HH^2)_{<{\epsilon'}}$ be a tube component of the $\epsilon'$-thin-part. Let $\alpha_n \in \pi_1(\Sigma)$ such that $\rho_n(\alpha_n) \in \Gamma_n$ corresponds to the waist geodesic of $T$. Then $\ell(\rho_n(\alpha_n)) < \epsilon' < \epsilon$ such that $\alpha_n$ is peripheral. Therefore, all tube components of the $\epsilon'$-thin-part are peripheral such that $C(\Gamma_n)_{\geq \epsilon'}$ is path-connected.
	
	We want to apply Lemma \ref{lem:diam_bound_thick_part} to $C(\Gamma_n)_{\geq \epsilon'}$. Note that $N_{\epsilon'}(C(\Gamma_n)_{\geq \epsilon'}) \subseteq N_\epsilon(C(\Gamma_n))$. Further, $N_\epsilon(C(\Gamma_n)) \setminus C(\Gamma_n)$ consists of half-collars of width $\epsilon$ about the boundary curves of $C(\Gamma_n)$. Since the lengths of the boundary curves converge there is a uniform bound $V>0$ such that $\vol_{\Gamma_n \backslash \HH^2}(N_\epsilon(C(\Gamma_n)) \setminus C(\Gamma_n)) \leq V$ for all $n \in \NN$. Recall that
	$\vol_{\Gamma_n \backslash \HH^2}(C(\Gamma_n)) = 2 \pi \abs{\chi(\Sigma)}$, such that,
	\begin{align*}
	\vol_{\Gamma_n \backslash \HH^2}( N_{\epsilon}(C(\Gamma_n)_{\geq \epsilon'})) &\leq \vol_{\Gamma_n \backslash \HH^2}(C(\Gamma_n)) + \vol_{\Gamma_n \backslash \HH^2}(N_\epsilon(C(\Gamma_n)) \setminus C(\Gamma_n))\\ 
	&\leq 2\pi \abs{\chi(\Sigma)} + V.
	\end{align*}

	Setting 
	\[ R(\epsilon') = \frac{4 \epsilon'}{v(B_o(\epsilon'))} \cdot \left( 2\pi \abs{\chi(\Sigma)} + V \right) \]
	the assertion follows from Lemma \ref{lem:diam_bound_thick_part}.
\end{proof} 

Finally, we will need to know the area of the residual thin-parts.

\begin{lemma} \label{lem:MeasureThinParts}
	\begin{enumerate}
		\item Let $\delta > 0$ and let $\gamma \in G \cong \Isom^+(\HH^2)$ be defined by $\gamma(z) = z+1$ for every $z \in \HH^2$. Consider the fundamental domain for the corresponding cusp region
		\[ C_\delta := \{ z = x + i y \in \HH^2 \, | \,0 \leq x \leq 1, d(z, \gamma(z)) \leq \delta\} , \]
		that consists of all the points that are moved less than  $\delta$ by $\gamma$.
		
		Then 
		\[ v(C_\delta) = 2 \sinh(\delta/2).\]
		
		\item Let $\delta > \delta_0 > 0$ and $\gamma \in G \cong \Isom^+(\HH^2)$ be defined by $\gamma(z) = e^{\delta_0} z$ for every $z \in \HH^2$. Consider the fundamental domain for the corresponding funnel
		\[ F_\delta := \{ z = x + i y \in \HH^2 \, | \, x \geq 0, 1 \leq \abs{z} \leq e^{\delta_0}, d(z, \gamma(z)) \leq \delta\}  \]
		to the right of the axis $\ax{\gamma} = i \RR$, that consists of all the points that are moved less than $\delta$ by $\gamma$.
		
		Then
		\[ v(F_\delta) \leq 2 \sinh(\delta/2). \]
	\end{enumerate}
	
\end{lemma}

\begin{proof}
	Recall the following formulas in hyperbolic geometry.
	\[ \sinh(d(z,w)/2) = \frac{\abs{z-w}}{2 \sqrt{\Im(z) \Im(w)}} \]
	for every $z,w \in \HH^2$, and
	\[ v(A) = \int_A \frac{1}{y^2} \, dx \,dy\]
	for every Borel set $A \subseteq \HH^2$; see \cite{beardon} for example.
	
	\begin{enumerate}
		\item We compute that $z = x + iy \in C_\delta$ if and only if
		\[ \sinh(\delta/2) \geq \sinh(d(z,z+1)/2) = \frac{1}{2y} \iff y \geq \frac{1}{2 \sinh(\delta/2)} =: y_\delta. \]
		Hence,
		\begin{align*}
		v(C_\delta) = \int_0^1 \int_{y_\delta}^\infty \frac{1}{y^2} \, dy \, dx =  \frac{1}{y_\delta} = 2 \sinh(\delta/2).
		\end{align*}
		
		\item For $z = r e^{i \alpha} \in \HH^2$, $r > 0, \alpha \in (0, \pi)$, we have that  $d(z,\gamma(z)) \leq \delta$ if and only if
		\begin{align*}
		\sinh(\delta/2) &\geq \sinh(d(z,e^{\delta_0} z)/2) =  \frac{\abs{ r e^{i \alpha} - r e^{\delta_0} e^{i \alpha}}}{2 \sqrt{r \sin(\alpha) \cdot r e^{\delta_0} \sin(\alpha)}} \\
		& = \frac{1}{\sin(\alpha)} \frac{e^{\delta_0}-1}{2  e^{\delta_0/2}} = \frac{\sinh(\delta_0/2)}{\sin(\alpha)} \\
		\iff & \sin(\alpha) \geq \frac{\sinh(\delta_0/2)}{\sinh(\delta/2)}.
		\end{align*}
		There is a unique $\alpha_\delta \in (0, \pi/2)$ such that $\sin(\alpha_\delta) = \frac{\sinh(\delta_0/2)}{\sinh(\delta/2)}$.
		
		Using polar coordinates we obtain
		\begin{align*}
		v(F_{\delta}) &= \int_{F_\delta} \frac{1}{y^2} \, dx \, dy =  \int_{\alpha_\delta}^{\pi/2} \int_1^{e^{\delta_0}} \frac{r}{r^2 \sin^2 \varphi} \, dr  \, d \phi\\
		&= \delta_0 \cdot \int_{\alpha_\delta}^{\pi/2} \frac{1}{\sin^2 \varphi} \, d \phi
		= \delta_0 \cdot \left[ \cot \phi \right]_{\phi=\alpha_\delta}^{\pi/2}\\
		&= \delta_0 \cdot \cot(\alpha_\delta) = \delta_0 \cdot \frac{\cos(\alpha_\delta)}{\sin(\alpha_\delta)}
		\leq \frac{\delta_0}{\sin(\alpha_\delta)}\\
		&= \frac{\delta_0}{\sinh(\delta_0/2)} \sinh(\delta/2) \leq 2 \sinh(\delta/2)
		\end{align*}
		where we used in the last inequality that $x \leq \sinh(x)$ for all $x \geq 0$.
	\end{enumerate}
	
\end{proof}

We are ready to prove Lemma \ref{lem:L1ConvTruncDom} now. 

\begin{proof}[Proof of Lemma \ref{lem:L1ConvTruncDom}]
	By definition $\hat{D}_o(\Gamma) = \tilde{C}(\Gamma) \cap D_o(\Gamma)$ such that
	\[ \II_{\hat{D}_o(\Gamma)} = \II_{\tilde{C}(\Gamma)} \cdot \II_{D_o(\Gamma)}. \]
	By Lemma \ref{lem:CoreConv} and Lemma \ref{lem:DomConv} we have that
	\[ \II_{\tilde{C}(\Gamma_n)}(x) \cdot \II_{D_o(\Gamma_n)}(x) \to  \II_{\tilde{C}(\Gamma)}(x) \cdot \II_{D_o(\Gamma)}(x)   \qquad (n \to \infty)\]
	for every $x \in \HH^2\setminus (\partial \tilde{C}(\Gamma) \cup \partial D_o(\Gamma))$. Note that $\partial \tilde{C}(\Gamma) \cup \partial D_o(\Gamma)$ has measure zero.
	
	Let $\epsilon>0$ be as in Lemma \ref{lem:concrete_diam_bound_thick_part} and let $0 < \epsilon' < \epsilon$. Then there is $R = R(\epsilon')>0$ such that $\pi_n(\ol{B}_o(R))$ contains $C(\Gamma_n)_{\geq \epsilon'}$ for every $n \in \NN$, where $\pi_n \colon \HH^2 \longrightarrow \Gamma_n\backslash\HH^2$ is the quotient map. The complement $C(\Gamma_n) \setminus \pi_n(\ol{B}_o(R))=\bigsqcup_{k=1}^l W_k$ is a disjoint union of subsets $W_1, \ldots, W_l$ of peripheral cusp or tube components of the $\epsilon'$-thin-part. If $\Sigma$ has genus $g$ and $p$ punctures there are at most $l \leq p$ such components. By Lemma \ref{lem:MeasureThinParts} we have that
	\[ \vol(W_k) \leq 2\sinh(\epsilon'/2) \]
	for every $k=1, \ldots, l$, such that
	\[ \vol(C(\Gamma_n) \setminus \pi_n(\ol{B}_o(R))) \leq 2 p \sinh(\epsilon'/2) \to 0 \qquad (\epsilon' \to 0).\]
	Hence,
	\[ v(\hat{D}_o(\Gamma_n) \setminus \ol{B}_o(R)) \leq  2 p \sinh(\epsilon'/2) \to 0 \qquad (\epsilon' \to 0).\]
	
	Let $\epsilon'' >0$ be arbitrary, and choose $\epsilon'>0$ and $R=R(\epsilon')>0$ above, so that 
	\[v(\hat{D}_o(\Gamma_n) \setminus \ol{B}_o(R)) \leq \frac{\epsilon''}{3} \qquad \text{and}  \qquad v(\hat{D}_o(\Gamma) \setminus \ol{B}_o(R)) \leq \frac{\epsilon''}{3}.\]
	
	Then we compute that
	\begin{align*}
	& \quad \int_{\HH^2} \abs{ \II_{\hat{D}_o(\Gamma_n)}(x) - \II_{\hat{D}_o(\Gamma)}(x)} \, dv(x) \\
	&\leq \int_{\HH^2} \abs{ \II_{\hat{D}_o(\Gamma_n)}(x) - \II_{\hat{D}_o(\Gamma_n) \cap \ol{B}_o(R)}(x)} \, dv(x) \\
	&+ \int_{\HH^2} \abs{ \II_{\hat{D}_o(\Gamma_n) \cap \ol{B}_o(R)}(x) - \II_{\hat{D}_o(\Gamma) \cap \ol{B}_o(R)}(x)} \, dv(x) \\
	&+ \int_{\HH^2} \abs{ \II_{\hat{D}_o(\Gamma) \cap \ol{B}_o(R)}(x) - \II_{\hat{D}_o(\Gamma)}(x)} \, dv(x) \\
	&\leq \frac{2}{3} \epsilon'' + \int_{\HH^2} \II_{\ol{B}_o(R)}(x) \cdot \abs{ \II_{\hat{D}_o(\Gamma_n)}(x) - \II_{\hat{D}_o(\Gamma)}(x)} \, dv(x) .
	\end{align*}
	Observe that the function $\II_{\ol{B}_o(R)}(x) \cdot \abs{ \II_{\hat{D}_o(\Gamma_n)}(x) - \II_{\hat{D}_o(\Gamma)}(x)} $ converges pointwise almost everywhere to $0$ and is dominated by the $L^1$-function $\II_{\ol{B}_o(R)}(x)$. Hence, by the dominated convergence theorem we conclude that
	\[\int_{\HH^2} \abs{ \II_{\hat{D}_o(\Gamma_n)}(x) - \II_{\hat{D}_o(\Gamma)}(x)} \, dv(x) < \epsilon''\]
	for large $n$.
	
	Because $\epsilon''>0$ was arbitrary, the asserted convergence in $L^1(\HH^2,v)$ follows.
\end{proof}

	
	\bibliographystyle{amsalpha}
	\bibliography{bibliography}
\end{document}